\documentclass[11pt,bull-l]{amsart}
\usepackage[top = 1.1in, bottom = 1.1in, left =1in, right = 1in]{geometry}
\usepackage{amsfonts,amscd,amssymb,amsmath,mathrsfs}
\usepackage{graphicx,tikz,bbm,bm}
\usepackage{xparse}
\usepackage{enumerate,subfigure}
\usepackage{cite}
\usepackage{url}
\newcounter{alg}
\newcounter{e_alg}

\usepackage[notocbasic]{nomencl}

\makenomenclature
\let\nomnopage\nomenclature

\renewcommand{\Re}{\mathrm{Re}\,}
\renewcommand{\Im}{\mathrm{Im}\,}
\newcommand{\Db}{\mathbf{D}}
\newcommand{\Vb}{\mathbf{V}}
\newcommand{\OO}{\mathrm{O}}
\newcommand{\oo}{\mathrm{o}}
\newcommand{\ri}{\mathrm{i}}
\newcommand{\dd}{\mathrm{d}}

\renewcommand{\vec}{\bm }

\newtheorem{theorem}{Theorem}[section]
\newtheorem{definition}{Definition}
\newtheorem{remark}{Remark}[section]
\newtheorem{lemma}[theorem]{Lemma}
\newtheorem{proposition}[theorem]{Proposition}
\newtheorem{corollary}[theorem]{Corollary}
\newtheorem{assum}{Assumption}

\numberwithin{equation}{section}

\newcommand{\LP}{\operatorname{L}}

\DeclareMathOperator{\e}{e}
\newcommand{\I}{\mathrm{i}}
\newcommand{\sd}{\mathrm{d}}

\newcommand{\prob}{\mathbb{P}}

\renewcommand{\Pr}{\prob}
\DeclareDocumentCommand \one { o }
{%
\ensuremath
\IfNoValueTF {#1}
{\mathbf{1}  }
{\ensuremath{\mathbf{1}\left\{ {#1} \right\} }}%
}

\DeclareDocumentCommand{\Prto} {o} {
\IfNoValueTF {#1}
 {\overset{\Pr}{\longrightarrow}}
 { \xrightarrow[ #1 \to \infty]{\Pr }}
}
\DeclareDocumentCommand{\Asto} {o} {
\IfNoValueTF {#1}
 {\overset{\operatorname{a.s.}}{\longrightarrow}}
 {
 \xrightarrow[ #1 \to \infty]{\operatorname{a.s.} }
 }
}
\DeclareDocumentCommand{\Mgfto} {o} {
\IfNoValueTF {#1}
{\overset{\operatorname{mgf}}{\longrightarrow}}
{ \xrightarrow[ #1 \to \infty]{\operatorname{mgf} }}
}

\DeclareDocumentCommand{\Wkto} {o} {
\IfNoValueTF {#1}
 {\overset{(d)}{\longrightarrow}}
 { \xrightarrow[ #1 \to \infty]{(d) }}
}

\DeclareDocumentCommand{\To} {o} {
\IfNoValueTF {#1}
 {\rightarrow}
 { \xrightarrow[]{#1 \to \infty }}
}

\DeclareDocumentCommand \LPto { O{1} }
{\overset{\operatorname{\LP^{#1}}}{\longrightarrow}}

\title{The conjugate gradient algorithm on a general class of spiked covariance matrices}

\author{Xiucai Ding}
\address{University of California, Davis}
\email{xcading@ucdavis.edu}
  
\author{Thomas Trogdon}
\address{University of Washington, Seattle, WA}
\email{trogdon@uw.edu}

\thanks{The authors gratefully acknowledge support from the US National Science Foundation under grant NSF-DMS-1945652 (TT). Any opinions, findings, and conclusions or recommendations expressed in this material are those of the authors and do not necessarily reflect the views of the funding sources.}

\keywords{Sample covariance matrices, conjugate gradient}

\subjclass[2010]{65F10, 60B20}

\usepackage{hyperref}
\hypersetup{colorlinks,linkcolor={blue},citecolor={blue},urlcolor={red}}  
\begin{document}

\maketitle

\begin{abstract}
We consider the conjugate gradient algorithm applied to a general class of spiked sample covariance matrices.  The main result of the paper is that the norms of the error and residual vectors at any finite step concentrate on deterministic values determined by orthogonal polynomials with respect to a deformed Marchenko--Pastur law. The first-order limits and fluctuations are shown to be universal.  Additionally, for the case where the bulk eigenvalues lie in a single interval we show a stronger universality result in that the asymptotic rate of convergence of the conjugate gradient algorithm only depends on the support of the bulk, provided the spikes are well-separated from the bulk. In particular, this shows that the classical condition number bound for the conjugate gradient algorithm is pessimistic for spiked matrices.
\end{abstract}


\nomnopage{$\mathbb C$}{The field of complex numbers}  
\nomnopage{$\mathbb F^{N}$}{The space of all $N$-dimensional vectors with entries in $\mathbb F$}
\nomnopage{$\mathbb F^{N \times M}$}{The space of all $N \times M$ matrices with entries in $\mathbb F$}
\nomnopage{$A_{j:k,\ell:p}$}{The submatrix of $A$ consisting of all entries in rows $j$ through $k$ and columns $\ell$ through $p$.}
\nomnopage{$A_{j,\ell:p}$}{$A_{j:j,\ell:p}$}
\nomnopage{$A_{j:k,\ell}$}{$A_{j:k,\ell:\ell}$}


\section{Introduction}
Large-dimensional covariance matrices are fundamental objects in high-dimensional statistics and applied mathematics. For example, many statistical methodologies, including
 principal component analysis (PCA), clustering analysis, and regression analysis, require the knowledge of the covariance structure. Moreover, in applied mathematics, especially manifold learning, the kernel affinity matrix and graph Laplacian matrix are closely related to covariance matrices. We refer the readers to \cite{DW, NEK, MR2722294,MR3468554} for more details.

Sample covariance matrices play important roles in estimating and inferring population covariance matrices. Even though high-dimensional sample covariance matrices themselves cannot be applied directly, one can construct consistent estimators and useful statistics for inference  based on them.  In particular, researchers are often interested in understanding the asymptotics of the following random matrix
\begin{equation}\label{eq_generalmodel}
W=\Sigma^{1/2} XX^* \Sigma^{1/2},
\end{equation}
where $\Sigma$ is  the population covariance matrix and $X$ is an $N \times M$ random matrix with centered independent and identically distributed (iid) entries. In the literature, a popular, and quite delicate,
model is the spiked covariance matrix model \cite{DRMTA,Johnstone2001},
where a finite number of spikes (i.e., eigenvalues detached from the bulk of
the spectrum) are added to the spectrum of $\Sigma$; for a precise definition, we refer the readers to Section \ref{sec_model}.  Significant efforts have been made to understand the statistical properties of $W$ in (\ref{eq_generalmodel}) in the high-dimensional setting when $N$ is comparably large to $M$. For a comprehensive review, we refer the readers to \cite{BDWW, MR3449395,DRMTA,Johnstone2001,MR2399865,PAUL20141,MR3468554}.  



Despite the wide applications of sample covariance matrices within data science, most of the existing literature focuses on the study of the asymptotic statistical properties of $W$, and less is known on the algorithmic properties. More specifically, substantially less is known about how algorithms from numerical linear algebra and optimization act on sample covariance matrices. For the numerical solution of linear systems involving $W,$ when both $N$ and $M$ are large, Gaussian elimination is computationally expensive, and supposing exact arithmetic, the accuracy of the result may be entirely unnecessary.  Instead, iterative methods are often preferred.

Before proceeding to our main focus,  we pause to discuss some of the history of the analysis of algorithms on random matrices.  The first such analysis that we are aware of was that of Goldstine and von Neumann \cite{Goldstine1951} when they studied the conditioning of random matrices (see \cite{Trefethen1990} and \cite{Spielman2004} for more recent developments).  Subsequently, many authors (see, for example, \cite{Trotter1984,Silverstein1985,Edelman1988,Dumitriu2002}) analyzed the way in which classical factorization algorithms act on Gaussian matrices.  The analysis of fundamentally iterative methods applied to random matrices began with the work of Pfrang et al.~\cite{DiagonalRMT} and continued in \cite{Deift2014}.  Rigorous results were first obtained in \cite{DT18,DT17} for eigenvalue algorithms. For example, in \cite{DT17}, the authors analyzed the numerical performance of power iteration methods applied to calculate the largest eigenvalue of $W$ when $\Sigma=I.$ They prove that the halting time, i.e., the minimal number of iterations before the power method satisfies a given stopping rule, is universal and its distributional limit can be expressed  in terms of functionals of the limiting distribution of the largest eigenvalues of $W.$ The iteration errors and residuals can be analyzed similarly.

The main focus of the current work is towards the understanding of the solution of 
\begin{equation}\label{eq_system}
W \bm{x}=\bm{b}, 
\end{equation}
where $W$ is given in \eqref{eq_generalmodel}. In the applied mathematics literature, there exist many useful iterative algorithms for positive definite matrices (of which (\ref{eq_generalmodel}) is one such random model). One such algorithm is the conjugate gradient algorithm (CGA, c.f. Algorithm \ref{a:cg} below), which is one of the most important Krylov subspace methods \cite{MR1444820}. The CGA \cite{Hestenes1952} is an iterative method designed to solve (\ref{eq_system}). 
We highlight that when $\bm{b}$ is random, solving \eqref{eq_system}, can be related to high-dimensional regression via the normal equations  \cite[Section 2.3]{MR2722294}. More specifically, consider $\bm{a}=(a_1, \cdots, a_M),$ and set
\begin{equation}\label{eq_plus_noise}
a_i=\bm{x}^* \bm{y}_i+\epsilon_i,  \  1 \leq i \leq M,  
\end{equation}
where $\epsilon_i, \ 1 \leq i \leq M,$ are iid random noise and $\bm{y}_i=\Sigma^{1/2}X_i \in \mathbb{R}^N.$ Here $X_i$ refers to the $i$th column of $X.$ Then to obtain the ordinary least square estimator of $\bm{x}$ is equivalent to solving the normal equations
\begin{equation*}
W\bm{x}=Y \bm{a},    
\end{equation*}
where $Y$ collects the samples $\bm{y}_i$ and $W$ is the design matrix as in (\ref{eq_generalmodel}). 
In \cite{Deift2019b} the authors presented rigorous results for the halting time of the CGA for solving \eqref{eq_system}, when $\Sigma = I$ and $X$ has iid centered Gaussian entries.  The main result concerns the first-order limit of the norms of the error and residual vectors as $N \to \infty$.  This analysis was expanded in \cite{Paquette2020}, removing the Gaussian assumption, and providing the same results, i.e., proving universality, and determining the structure of the fluctuations.  These probabilistic results have strong connection to the deterministic results of \cite{Beckermann2001}.
We remark that since the methods employed in both \cite{Deift2019b,Paquette2020} rely on the Golub-Kahan bidiagonalization procedure as given in \cite{MR1936554}, they cannot be applied to $W$ in (\ref{eq_generalmodel}) when $\Sigma$ is not a scalar multiple of the identity matrix.    

\begin{figure}[ht]
    \centering
    \includegraphics[width=.49\linewidth]{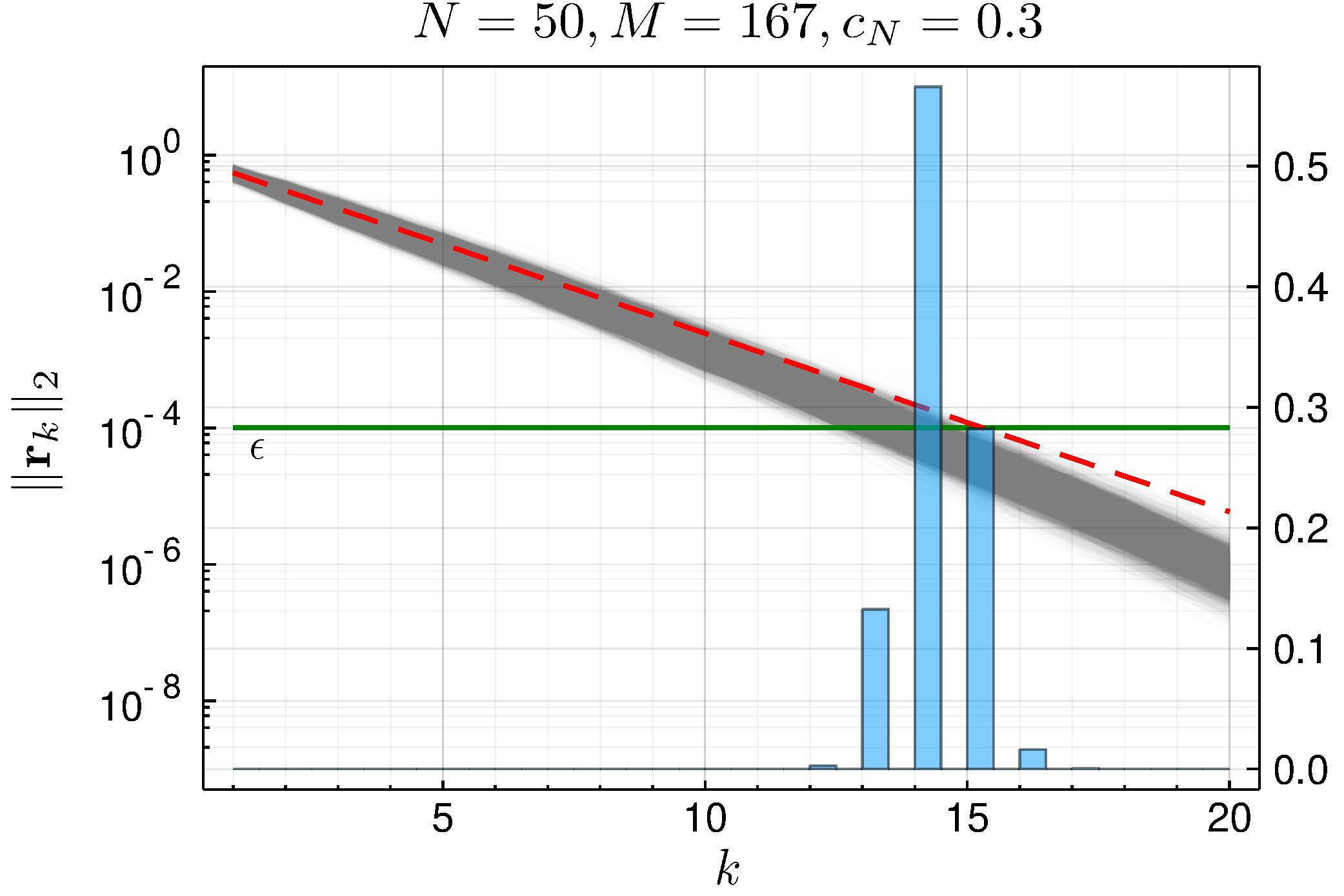}
    \includegraphics[width=.49\linewidth]{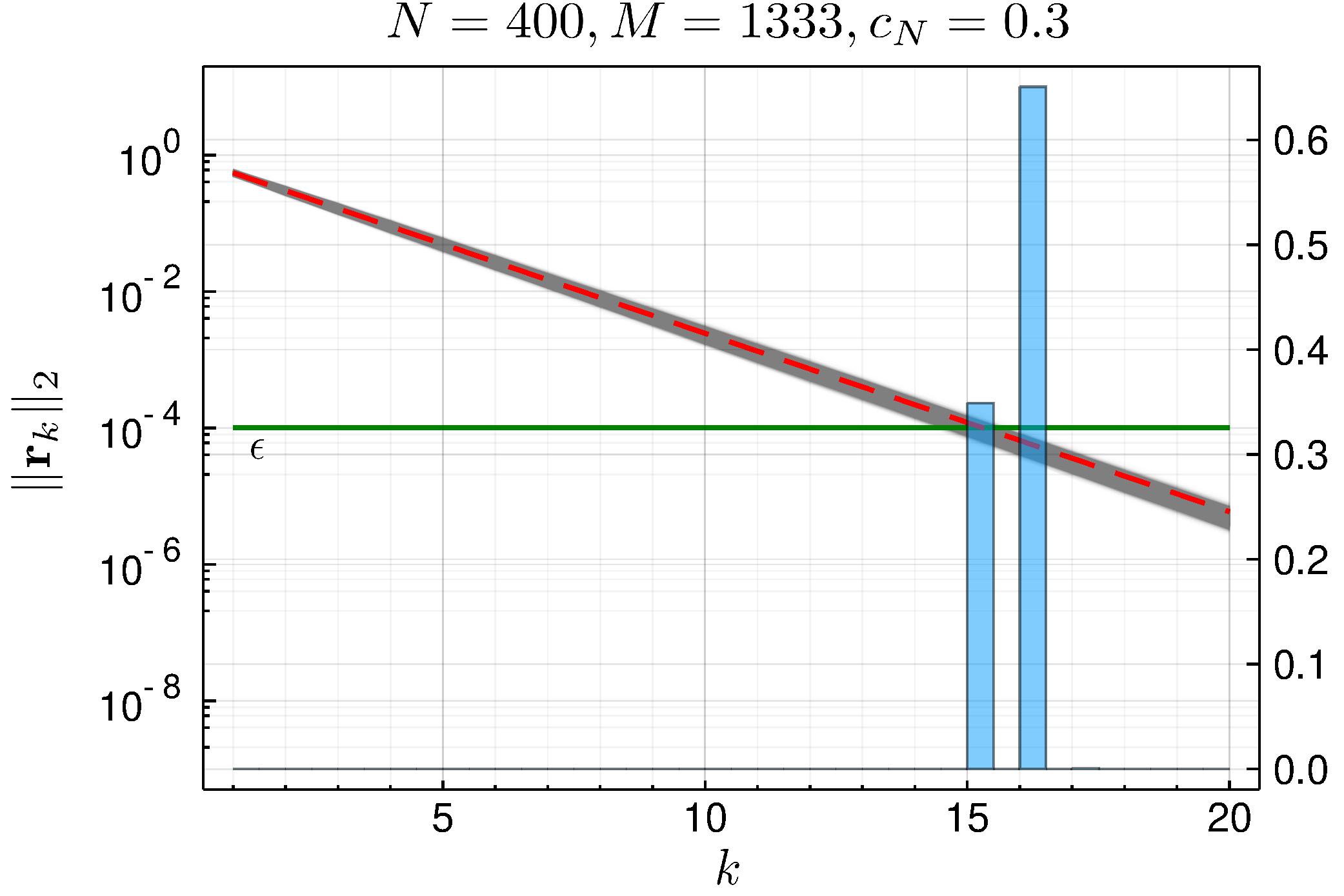}
    \includegraphics[width=.49\linewidth]{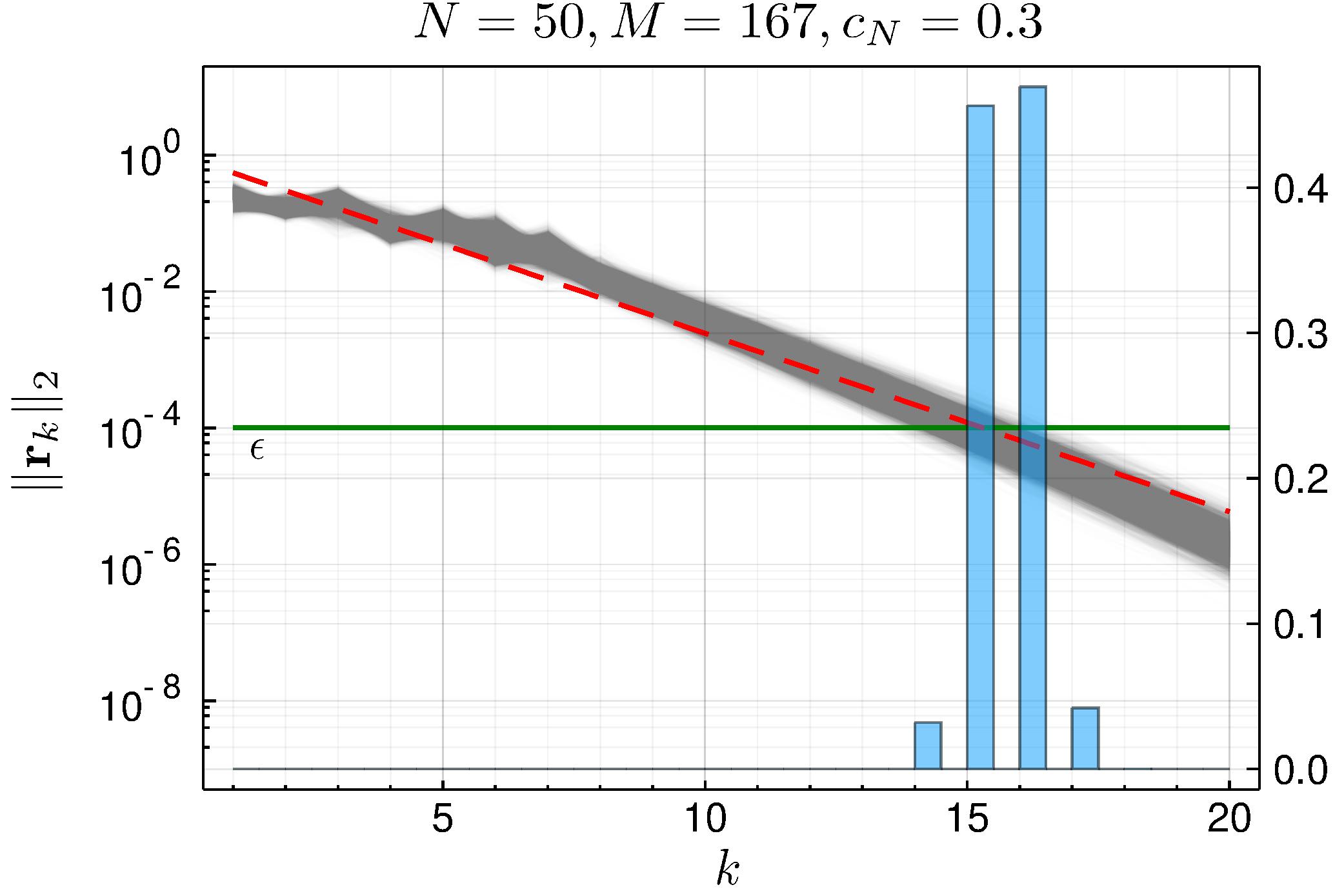}
    \includegraphics[width=.49\linewidth]{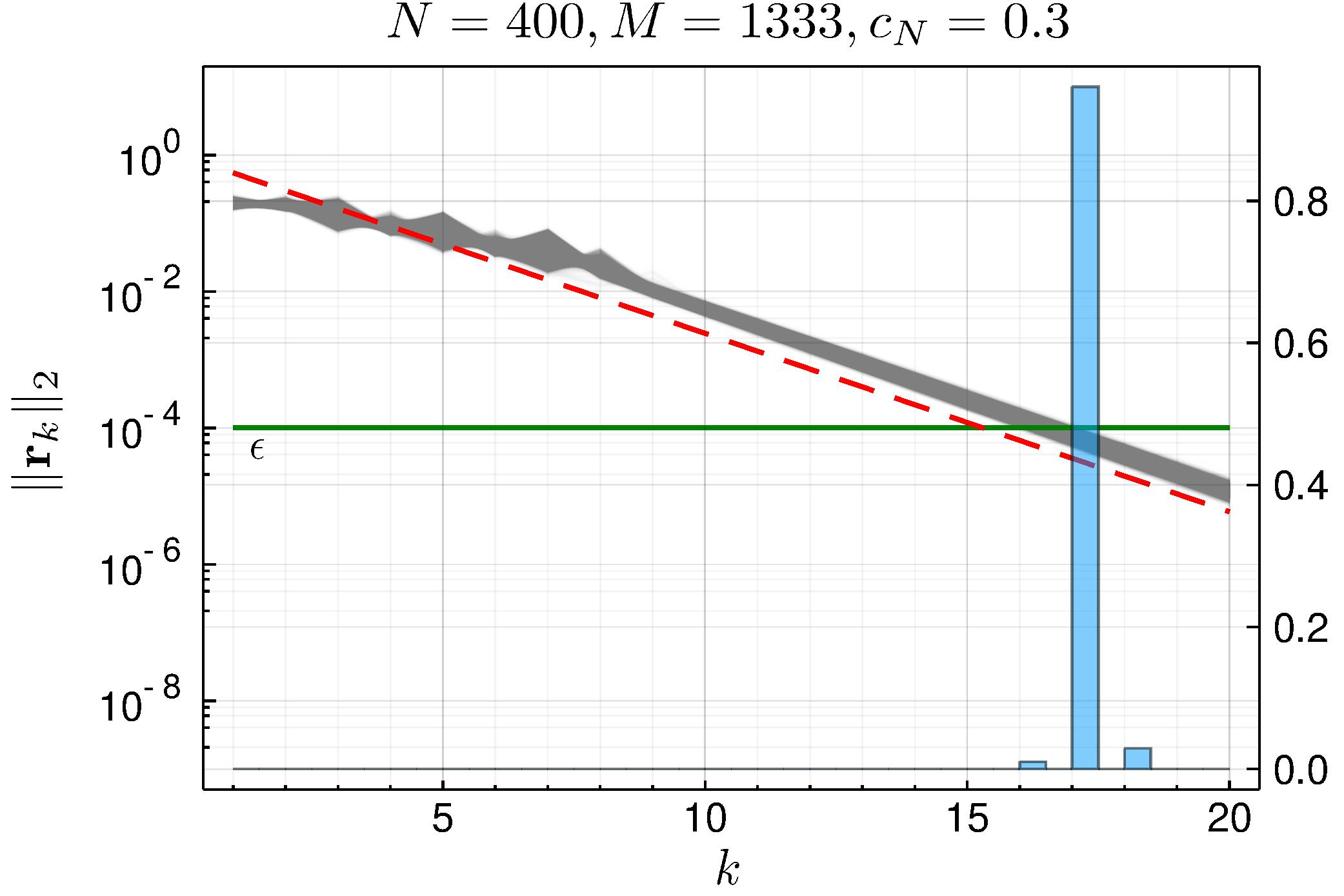}
    \caption{Top row:  A demonstration of the results in \cite{Deift2019b,Paquette2020}.  Shaded region consists of sampling 5000 matrices from the Wishart distribution (with $\Sigma = I$) and plotting the 2-norm of the residual versus $k$, the number of iterations in the CGA.  The dashed red line gives the asymptotic prediction from \cite{Deift2019b,Paquette2020}.  The blue histogram tallies the relative frequency of the halting time with $\epsilon = 10^{-4}$, i.e., the statistics of the number of iterations required to realize a residual with norm less than $\epsilon$.  Bottom row:  The same calculations as the top row but with $\Sigma^{1/2} = \mathrm{diag}(4,4,4,3.5,3.5,1,1,\ldots,1)$.  While the spikes induce a transient disturbance to the iteration, the asymptotic rate of convergence, for $k$ in a scaling region, is the same as when $\Sigma = I$.  The dashed red curve in the second row is the same as in the first, for comparison.}
    \label{fig:wishart_and_many_spikes}
\end{figure}

Motivated by the above applications and challenges, in the current paper, we develop a new strategy to analyze the first-order limits (including rates) of the residuals and errors in the CGA when $W$ is of the form (\ref{eq_generalmodel}); see Figure \ref{fig:wishart_and_many_spikes} for an illustration.  By using deterministic formulas (c.f. Proposition \ref{p:chol_inv_moment} and Lemma \ref{t:deterministic}), the residuals and errors of the CGA can be characterized using the entries of the Cholesky factorization of an associated semi-infinite Jacobi matrix (c.f. (\ref{eq_defnTnno})). It turns out that this Jacobi matrix coincides with the one produced from the well-known Lanczos iteration (c.f. Algorithm \ref{a:lanczos}). Moreover, we point out that the entries of the Jacobi matrix can be described as the three-term recurrence coefficients of the orthogonal polynomials generated by a spectral measure which is the \emph{eigenvector empirical spectral distribution} (VESD)\cite{MR2567175} (c.f. (\ref{eq_defnvesd})), which played a crucial role in \cite{DT17,DT18,Deift2019a,Paquette2020}.

\begin{remark}
The classical Chebyshev error bound for the CGA applied to $W \bm x = \bm b$ \cite{Hestenes1952} is
\begin{align*}
    \| \bm x - \bm x_k\|_W \leq 2 \left(\frac{ \sqrt{\lambda_{\max}} - \sqrt{\lambda_{\min}}}{ \sqrt{\lambda_{\max}} + \sqrt{\lambda_{\min}}}\right)^k \| \bm x - \bm x_0\|_W,
\end{align*}
where $\|\cdot\|_W$ is the $W$-norm, see \eqref{eq_krylovspace} below. The results of \cite{Deift2019a,Paquette2020} give that as $N \to \infty$
\begin{align*}
    \| \bm x - \bm x_k\|_W  = \left(\frac{ \sqrt{\lambda_{\max}} - \sqrt{\lambda_{\min}}}{ \sqrt{\lambda_{\max}} + \sqrt{\lambda_{\min}}}\right)^k \| \bm x - \bm x_0\|_W + o(1) =  (0.5477\ldots)^k \| \bm x - \bm x_0\|_W + o(1),
\end{align*}
when $\Sigma = I$ in \eqref{eq_generalmodel} demonstrating that the classical bound is quite good.  But this is no longer true in the presence of spikes as in the bottom row of Figure~\ref{fig:wishart_and_many_spikes}. The classical bound gives
\begin{align*}
    \| \bm x - \bm x_k\|_W \lessapprox 2 (0.8)^k \| \bm x - \bm x_0\|_W,
\end{align*}
since $\lambda_{\min}(W) = ( 1 - \sqrt{c_N})^2 + o(1)$ and $\lambda_{\max}(W) = 16.32 + o(1)$  ( see Lemma~\ref{lem_spikedmodeloutliereigenvalue} below), all as $N \to \infty$. Our estimates, see Theorem~\ref{thm:mainthree} give a better estimate
\begin{align*}
\| \bm x - \bm x_k\|_W = (0.5477\ldots) \| \bm x - \bm x_{k-1}\|_W + o(1),
\end{align*}
for sufficiently large $k$, i.e., after a transient period induced by the presence of spikes.
\end{remark}

When $\Sigma$ has no spikes, the concentration and convergence of the VESD can be established using the so-called anisotropic local laws \cite{Knowles2017} from random matrix theory.  Based on those results, we establish the concentration of the VESD for the spiked model (c.f. Lemma \ref{lem_connectionspikednonspiked}). Finally, as is well known and as was observed in \cite{Paquette2020}, since orthogonal polynomials can be fully constructed by its associated Hankel moment matrix of the VESD (c.f. Section \ref{sec_bascis} and \cite{DeiftOrthogonalPolynomials} for more detail), we can obtain our main results by only analyzing the convergence of the moments of the VESD.           

We emphasize that the aforementioned strategy can handle general spiked covariance matrices $W$ in (\ref{eq_generalmodel}). However, when $\Sigma$ in (\ref{eq_generalmodel}) does not contain spikes or when $\bm{b}$ satisfies certain conditions (c.f. (\ref{eq_conditionb})) , we simply the procedure and obtain simple asymptotic expressions: (1) The simplification first utilizes the asymptotic relation of the three-term recurrence coefficients that is most simply derived using the Riemann-Hilbert approach as in \cite{MR2087231}. It turns out that asymptotically, the associated Jacobi semi-infinite matrix has a very simple structure that can be described by the edges of the limiting VESD (c.f. Theorem \ref{thm_relationOP}). (2) Then a straightforward calculation for the Cholesky factorization will result in simple expressions. (c.f. Theorem \ref{thm:mainone}).  (3) The edges of the limiting VESD can be calculated using the critical points of an analytic function as in (\ref{eq_defnstitlesjtransform}).

Finally, we mention that the main focus on this manuscript is to develop a new strategy and novel formulas for the first order limits and rates of the CGA. However, we also establish the second order universality on the distributions of the residuals and errors. More specifically, we show that they only depend on the first four moments of the entries of $X$ in (\ref{eq_generalmodel}). The universality indicates that we can construct useful statistics based on the algorithms to infer the population covariance matrix $\Sigma$ in (\ref{eq_generalmodel}). This opens a new door for high-dimensional statistical inference; see Remark \ref{rmk_statistics} for more details.  To have a complete description of the performance of the CGA applied to (\ref{eq_generalmodel}), we still need to consider the second
order asymptotics, i.e. the limiting distribution of the residuals and errors. This will be included in our future works, for example, see \cite{DT2}. 






%

This paper is organized as follows. In Section \ref{sec_background}, we introduce the conjugate gradient algorithm and the general spiked covariance matrix model. In Section \ref{sec_mainresults}, we state our main results. In Section \ref{s.example}, we provide some examples and conduct some numerical simulations for illustration. In Section \ref{sec_theoryoforthogonalpolynomial}, we provide the theory of orthogonal polynomials and prove some essential asymptotics of the three-term recurrence relations. In Section \ref{sec_rmt}, we provide and prove the key ingredients regarding eigenvector empirical spectral distribution. The main technical proofs are summarized in Sections \ref{sec_mainproof} and \ref{sec_proofuniversality}. Some formulas, additional technical proofs and auxiliary lemmas are collected in Appendices \ref{sec_appa}, \ref{sec_appb} and \ref{sec_appc}.     

\vspace{3pt}

\noindent{\bf Conventions.} We denote by $\{\bm{f}_k\}_{k \geq 1} \subset \mathbb{R}^N$ the standard Euclidean basis of $\mathbb{R}^N. $ We denote $\mathbb{C}_+:=\{z=E+\mathrm{i} \eta \in \mathbb{C}: \eta>0 \}. $
The fundamental large parameter is $M$ and we always assume that $N$ is comparable to and depends on $M$. All quantities that are not explicitly constant may depend on $M$, and we usually omit $M$ from our notations. We use $C$ to denote a generic large positive constant, whose value may change from one line to the next. Similarly, we use $\epsilon$, $\tau$, $c$, etc. to denote generic small positive constants. If a constant depends on a quantity $a$, we use $C(a)$ or $C_a$ to indicate this dependence.  For two quantities $a_N$ and $b_N$ depending on $N$, the notation $a_N = \OO(b_N)$ means that $|a_N| \le C|b_N|$ for some constant $C>0$, and $a_N=\oo(b_N)$ means that $|a_N| \le c_N |b_N|$ for some positive sequence $c_N\downarrow 0$ as $N\to \infty$.  We use capital letters to refer to matrices and boldface to refer to vectors.  Lower-case letter will be used to refer to individual entries of a matrix, for example, $x_{ij}$ denotes the $(i,j)$ entry of a matrix $X$.  We use $X_{i:j,\ell:k}$ to denote the subblock of a matrix $X$ consisting of all entries in rows $i$ through $j$ and columns $\ell$ through $k$. If either $j$ or $k$ are absent then this notation refers to all entries in rows $\geq i$ or columns $\geq \ell$, respectively. \\

\noindent{\bf Disclaimer.}  All of our results concern running algorithms with exact arithmetic.  It is well-known that the Lanczos iteration and the CGA suffer from instabilities due finite-precision arithmetic \cite{Greenbaum1989,PAIGE1976}.  So, in the current paper, to simulate full precision arithmetic, we, when necessary, use an appropriately modified Householder reflection-based tridiagonalization because of its superior numerical stability.  In general, we notice that for spiked random matrices, the Lanczos iteration, and hence the CGA,  loses accuracy.  When no spikes are present and there is only bulk spectrum, the Lanczos iteration closely tracks the Householder-based algorithm.

\section{The conjugate gradient algorithm and the model}\label{sec_background}
This section is devoted to introducing the necessary background. In Section \ref{sec:lanczos-detail}, the CGA is stated and its connection with Lanczos iteration is discussed. In Section \ref{sec_model}, we introduce the spiked covariance matrix model that will be used throughout the current paper.

\subsection{The conjugate gradient algorithm and Lanczos iteration}\label{sec:lanczos-detail} In this subsection, we provide the background on the CGA. The actual CGA is given by Algorithm \ref{a:cg} below. The CGA can also be characterized in its varational form. Define the Krylov space 
\begin{equation}\label{eq_krylovspace}
\mathcal{K}_k=\operatorname{span}\left\{ \bm{b}, W \bm{b}, \cdots, W^{k-1} \bm{b} \right\}. 
\end{equation} 
Starting with $\bm{x}_0=\bm{0},$ the $k$th iterate, $\bm{x}_k,$ of  the CGA satisfies (see \cite[Chapter 11]{MR3024913} or \cite[Lecture 38]{MR1444820})
\begin{equation}\label{eq_optimizationproblem}
\bm{x}_k=\operatorname{argmin}_{\bm{y} \in \mathcal{K}_k} \| \bm{x}-\bm{y} \|_W. 
\end{equation}
Here we use the notation that for any vector $\bm{z}$ and positive definite matrix $A,$
\begin{equation*}
\| \bm{z} \|_A^2=\bm{z}^* A \bm{z}. 
\end{equation*}  

\noindent\fbox{%
\refstepcounter{alg}
    \parbox{\textwidth}{%
\flushright \boxed{\text{Algorithm~\arabic{alg}: Conjugate Gradient Algorithm (CGA) \label{a:cg}}}
\begin{enumerate}
    \item $\bm{x}_0$ is the initial guess. 
    \item Set $\bm{r}_0 = \bm{b}-W \bm{x}_0$, $\bm{p}_0=\bm{r}_0.$
    \item For $k = 1,2,\ldots,n$, $n \leq N$ is the maximum steps of iterations
    \begin{enumerate}
        \item Compute $\displaystyle a_{k-1} = \frac{\bm{r}^*_{k-1} \bm{r}_{k-1}}{\bm{r}^*_{k-1} W \bm{p}_{k-1}}$.
        \item Set $\vec x_k = \vec x_{k-1}+a_{k-1} \bm{p}_{k-1}$.
       \item Set $\vec r_k = \vec r_{k-1}-a_{k-1} W \bm{p}_{k-1}$.
        \item Compute $\displaystyle b_{k-1} = -\frac{\bm{r}^*_{k-1} \bm{r}_{k-1}}{\bm{r}^*_{k-1}  \bm{r}_{k-1}}$.
        \item Set $\vec p_k=\vec r_k-b_{k-1} \vec p_{k-1}.$
    \end{enumerate}
\end{enumerate}
    }%
}

\vspace{.1in}

The primary goal of the analysis of the CGA is to analyze the residual and error vectors, denoted by $\bm{r}_k(W, \bm{b})$ and $\bm{e}_k(W, \bm{b}),$ respectively, and defined as
\begin{equation*}
\bm{r}_k(W, \bm{b}):=\bm{b}-W \bm{x}_k, \ \bm{e}_k(W, \bm{b}):=\bm{x}-\bm{x}_k.
\end{equation*}

It can be seen from (\ref{eq_optimizationproblem}) that the Krylov subspace plays a central role in the analysis of the CGA. It is well-known that the Lanczos iteration \cite[Lecture 36]{MR1444820}  can be used to produce an orthonormal basis for the Krylov subspace. As a result, the CGA is closely related to Lanczos iteration \cite[Section 11.3.5]{MR3024913}. In fact, as discussed in Theorem \ref{t:deterministic} (reproduced from \cite{Paquette2020} for the reader's convenience), the residuals and errors can be represented based on the outputs of Lanczos iteration. The Lanczos iteration can be applied to any symmetric or Hermitian matrix $W$ and it takes the following form:\\

\noindent\fbox{%
\refstepcounter{alg}
    \parbox{\textwidth}{%
\flushright \boxed{\text{Algorithm~\arabic{alg}: Lanczos Iteration\label{a:lanczos}}}
\begin{enumerate}
    \item $\vec q_1$ is the initial vector.  Suppose $\|\vec q_1\|_2^2 = \vec q_1^* \vec q_1 = 1$
    \item Set $b_{-1} = 1$, $\vec q_{0} = 0$
    \item For $k = 1,2,\ldots,n$, $n \leq N$
    \begin{enumerate}
        \item Compute $\displaystyle a_{k-1} = (W \vec q_k - b_{k-2} \vec q_{k-1})^* \vec q_k$.
        \item Set $\vec v_k = W \vec q_k - a_{k-1} \vec q_k - b_{k-2} \vec q_{k-1}$.
        \item Compute $b_{k-1} = \|\vec v_k\|_2$ and if $b_{k-1} \neq 0$, set $\vec q_{k+1} = \vec v_k/b_{k-1}$.
    \end{enumerate}
    \item Return $a_0,\ldots,a_{n-1},b_0,\ldots,b_{n-2}$
\end{enumerate}
    }%
}

\vspace{.1in}
The Lanczos algorithm at step $k \leq N$  produces a Jacobi matrix $T_k$ and vectors $\vec q_1,\ldots,\vec q_k,$ denoted as 
\begin{align*}
    Q_k &= \begin{bmatrix} \vec q_1 & \vec q_2 & \cdots &\vec q_k \end{bmatrix}, \quad T_k = T_k(W,\vec q_1) = \begin{bmatrix} a_0 & b_0 \\
    b_0 & a_1 & \ddots \\
    & \ddots & \ddots & b_{k-2} \\
    & & b_{k-2} & a_{k-1} \end{bmatrix},\ 
    a_j \in \mathbb R,\quad b_j > 0,
\end{align*}
such that
\begin{align}\label{eq:Tk}
 W Q_k = Q_k T_k + b_{k-1} \vec q_{k+1} \vec f_k^*.
\end{align}
We use the notation $T = T(W,\vec q_1) = T_n(W,\vec q_1)$ for the matrix produced when the Lanczos iteration runs for its maximum of $n$ steps. We point out that the columns of $Q_k$ provide an orthonormal basis for the Krylov subspace $\operatorname{span}\{\bm{q}_1, W \bm{q}_1, \cdots, W^{k-1} \bm{q}_1\}$ \cite[Lecture 36]{MR1444820}.


\begin{remark}
In this paper, we focus on the analysis of the CGA. However, the arguments can be easily generalized to many other numerical algorithms involving large dimensional random matrices. For example, in Section \ref{sec_otheralgo}, we provide the results for another iteration algorithm MINRES. 
Additionally, our results provide the existence of first-order limits for the algorithms discussed in \cite{Paquette2020a}.
\end{remark}

\subsection{General spiked covariance matrix model}\label{sec_model}
In this paper, we are interested in the setting when $W$ is random and  the high dimensional  scenario when $M$ is comparably large to $N$ such that for some small constant $0<\tau<1,$
\begin{equation}\label{eq_dimensionality}
\tau \leq c_N:=\frac{N}{M} \leq \tau^{-1}.
\end{equation} 

In this subsection, we introduce the model for $W$ as in (\ref{eq_generalmodel}) .  Moreover, we assume that $X=(x_{ij})$ is an $N \times M$ random matrix whose entries $x_{ij}, 1 \leq i \leq N, 1 \leq j \leq M,$ are real or random variables satisfying
\begin{equation}\label{eq_meanvariancexij}
\mathbb{E} x_{ij}=0, \ \mathbb{E} x_{ij}^2=\frac{1}{M}.
\end{equation}
For definiteness, in this paper, we focus on the real case, i.e., the random variables $x_{ij}$ are real.  However, we remark that our proof can be applied to the complex case after minor modifications if we assume in addition
that $\Re x_{ij}$ and $\Im x_{ij}$ are independent centered random variables with variance $(2M)^{-1}. $ We also assume that the random variables $x_{ij}$ have arbitrarily high moments, in the sense that for any fixed $k \in \mathbb{N}$, there is
a constant $C_k>0$ such that 
\begin{equation}\label{eq_momentassumption1}
\max_{i,j} \left( \mathbb{E} |x_{ij}|^k \right)^{1/k} \leq C_k M^{-1/2}.  
\end{equation}
The assumption that (\ref{eq_momentassumption1}) holds for all $k \in \mathbb{N}$ may be easily relaxed. For instance, it
is easy to check that our results and their proofs remain valid, after minor adjustments using some suitable truncation and comparison techniques, if we only
require that (\ref{eq_momentassumption1}) holds for all $k \leq C$ for some finite constant $C$. As this is not the main focus of our current paper, we do not pursue such
generalizations. 

For the population covariance matrix,  we consider the spiked covariance matrix model following the setting of \cite{DRMTA}. Let $\Sigma$ be a spiked population covariance matrix that admits the following spectral decomposition 
\begin{equation}\label{eq_defnsigma}
\Sigma=\sum_{i=1}^N \widetilde{\sigma}_i \bm{v}_i \bm{v}_i^*, \ \widetilde{\sigma}_i=(1+d_i) \sigma_i,
\end{equation}
where $\sigma_1 \geq \sigma_2 \geq \cdots \geq \sigma_N>0$ and for some fixed integer $r \geq 0,$
\begin{equation*}
d_i>0, \ i \leq r; \  d_i=0,\ i>r.  
\end{equation*} 
The first $r$ eigenvalues of $\Sigma$ are the \emph{spikes} which may result in outlying eigenvalues of $W$. Throughout the paper, we will call (\ref{eq_generalmodel}) the \emph{spiked covariance matrix model}. Except for a few outliers, the limiting empirical spectral distribution of $W$ coincides with the associated \emph{non-spiked covariance matrix model}, which is defined as follows
\begin{equation}\label{eq_definitioncovariance}
W_0=\Sigma_0^{1/2} XX^* \Sigma_0^{1/2},
\end{equation}     
where
%
%
%
%
%
%
%
%
%
%
$\Sigma_0$ has the following spectral decomposition
\begin{equation}\label{eq_defnsigma0}
\Sigma_0=\sum_{i=1}^N \sigma_i \bm{v}_i \bm{v}_i^*.  
\end{equation}
Note that $\Sigma_0$ is the non-spiked version of $\Sigma$ in (\ref{eq_defnsigma}) with $r=0.$ 

\begin{remark}
We distinguish $\Sigma_0$ from $\Sigma$ because if a limit is desired for certain spectral statistics of \eqref{eq_generalmodel}, then $\Sigma_0$ will require some additional assumptions to be placed on it.  Specifically, one might want to take the $\sigma_i$'s to be the quantiles of some sufficiently regular distribution.  This aspect is discussed further in \eqref{eq_sigma_tau} and \eqref{eq_regularedgeandbulk} below.
\end{remark}

When $\sigma_i \equiv 1, 1 \leq i \leq N,$  it is well-known that the eigenvalues of $W_0$ obey the Marchenko-Pastur (MP) law \cite{Marcenko1967} and for general $\Sigma_0,$ they are governed by the \emph{deformed Marchenko-Pastur law} \cite{MR2567175, Knowles2017}.  When $r>0$ and $d_i, i \leq r,$ in (\ref{eq_defnsigma}) are above some critical values, the corresponding eigenvalues of $W$ will detach from the  bulk (or the support of the deformed MP law) and become outliers; see  Lemma \ref{lem_spikedmodeloutliereigenvalue} below for a more precise statement. 

In this paper, we consider both the non-spiked and spiked covariance matrix models. As we will see later, the discussion of the spiked model $W$ is based on that of the non-spiked model $W_0.$

\begin{remark}
In \cite{Paquette2020}, the authors studied the CGA for the non-spiked model under a specific setting when $\Sigma_0=I$ in (\ref{eq_defnsigma0}). Their arguments are based on (\ref{eq:Tk}), which implies that
\begin{equation}\label{eq_calculationreduced}
Q_k^* W Q_k=T_k.
\end{equation}
Since $Q_k$ is orthonormal, when $X$ is invariant (e.g. $X$ is a Gaussian matrix), the spectral distribution of $W$ can be studied via those of $T_k.$ However, when $\Sigma_0 \neq I,$ even when $X$ is Gaussian, this method fails. 

While we focus on the covariance type random matrix model  (\ref{eq_generalmodel}) we note that our framework and results can be generalized to other types of random matrix models, for example, the separable covariance matrix model in \cite{DYAOS} when $W=A^{1/2}XBX^* A^{1/2}$ for two positive definite matrices $A$ and $B.$ We will consider such generalizations in the future works. 

\end{remark}

\section{Main results}\label{sec_mainresults}

In this section, we state our main results. We first provide an overview of this section.  Section \ref{sec_sub_notationandassumption} is devoted to introducing some notations and the technical assumptions. In Section \ref{sec_lanczosanalysis}, we analyze the Lanczos algorithm. In Section \ref{sec_generalrecipe}, we conduct the error analysis for the CGA when $\bm{b}$ is deterministic. First, we propose a general algorithm, Estimation Algorithm \ref{a:generalaccurate},  to calculate some essential quantities. Armed with these quantities,  we establish the first-order limits and rates for norms of $\bm{e}_k$ and $\bm{r}_k$. Second, under additional regularity assumptions, we can push the calculation further and obtain simple formulas, see, for example, Theorem \ref{thm:mainone}. 

 In Section \ref{sec_sub_normalequation}, we give results when $\bm{b}$ is random such that the linear system becomes the normal equations $YY^* \bm{x}=Y^* \bm{a}, Y=\Sigma^{1/2}X$ for the spiked model and $Y=\Sigma^{1/2}_0 X$ for the non-spiked model. It turns out that the residuals and errors for the normal equation have the same asymptotics regardless of whether $\Sigma$ is spiked or not; see Theorem \ref{thm_normalequation} for more details. In Section \ref{s.universality}, we study the second-order fluctuations and prove that the results are universal --- they depend only on the first four moments of $x_{ij}.$ Finally, in Section  \ref{sec_otheralgo}, we discuss implications of the results and  and apply the results to another iterative Krylov subspace algorithm, the minimal residual method (MINRES) \cite{MR383715} to illustrate the generality of our proposed error analysis framework.


\subsection{Notations and assumptions}\label{sec_sub_notationandassumption} We provide some necessary notation and assumptions in this subsection. For any $N \times N$ Hermitian matrix $Z,$  denote its empirical spectral distribution (ESD) as 
\begin{equation}\label{defn_esd}
\mu_Z=\frac{1}{N} \sum_{i=1}^N \delta_{\lambda_i(Z)}. 
\end{equation}
Denote by $m_{\mu_Z}(z)$ the Stieltjes transform of $\mu_Z,$ i.e.,  
\begin{equation*}
m_{\mu_Z}(z)=\int \frac{1}{x-z} \mu_Z(\dd x), \ z \in \mathbb{C}_+. 
\end{equation*} 
We then denote the companion of $W_0$ in (\ref{eq_definitioncovariance}) as 
\begin{equation*}
\mathcal{W}_0=X^* \Sigma_0 X. 
\end{equation*} 
Note that $W_0$ and $\mathcal{W}_0$ have the same non-zero eigenvalues. 

It is well known that \cite{MR2567175}, in general, the asymptotic density function of the ESD of $\mathcal{W}_0$ follows the deformed Marchenko--Pastur law, denoted as $\varrho.$ The deformed MP law is best characterized by its Stieltjes transform. Let $z \in \mathbb{C}_+,$ the Stieltjes transform $m(z) \equiv m_{\varrho}(z)$ of $\varrho$ can be characterized as the unique solution of the following equation \cite[Lemma 2.2]{Knowles2017}
\begin{equation}\label{eq_inverserelation}
z=f(m), \ \Im m(z) \geq 0, 
\end{equation}
where $f(x)$ is defined as
\begin{equation}\label{eq_defnstitlesjtransform}
f(x)=-\frac{1}{x}+\frac{1}{M} \sum_{k=1}^N \frac{1}{x+\sigma_k^{-1}}.
\end{equation}
Based on $\varrho,$ we denote the density function $\varrho_{\bm{b}}$ as
\begin{equation}\label{eq_varrhob}
\varrho_{\bm{b}}(x)
= \frac{\varrho(x)}{x} \bm{b}^* \Sigma_0 \left[I +2 \Re m(x+ \I 0^+) \Sigma_0+|m(x + \I 0^+)|^2 \Sigma_0^2 \right]^{-1} \bm{b}. 
\end{equation} 
Moreover, we define the moments of $\varrho_{\bm{b}}$ as 
\begin{equation}\label{eq_momentdefinition}
 \mathfrak{m}_{k,\bm{b}}:=\int_{\mathbb{R}} \lambda^k \varrho_{\bm{b}} (\lambda) \mathrm{d} \lambda. 
\end{equation}
For any integer $n \leq N,$ denote the Hankel moment matrix of $\varrho_{\bm{b}}$ by
\begin{equation} \label{eq_hankeldeterminant}
    D_n = \det M_n, \quad (M_n)_{ij} = \mathfrak{m}_{i+j-2, \bm{b}}, \quad n \geq 0, \quad 1 \leq i,j \leq n+1,\quad D_{-1} = 1,
\end{equation}
and, since $\varrho_{\bm{b}}$ does not vanish identically if $\bm{b} \neq 0$, define the associated quantities
\begin{align}\label{eq:sandl}
    \ell_n = \sqrt{\frac{D_{n-1}}{D_{n}}}, \quad s_n = - \frac{\det \mathtt{M}_n}{\sqrt{D_n D_{n-1}}},
\end{align}
where $\mathtt{M}_n$ is the matrix formed by removing the last row and second-to-last column of $M_n.$ Similarly, we  define the relevant quantities for the spiked model. Specifically, we set 
\begin{equation} \label{eq_hankeldeterminantspiked}
    \widetilde{D}_n = \det \widetilde{M}_n, \quad (\widetilde{M}_n)_{ij} = \widetilde{\mathfrak{m}}_{i+j-2,\bm{b}}, \quad 1 \leq i,j \leq n+1,
\end{equation}
where $\widetilde{\mathfrak{m}}_{k, \bm{b}}$ is defined by 
\begin{equation}\label{eq_pertubedmoment}
\widetilde{\mathfrak{m}}_{k, \bm{b}}:=\sum_{i=1}^N \frac{\mathtt{b}_i^2}{1+d_i}\left( \mathfrak{m}_{k,\bm{v}_i}- \mathbf{1}(i \leq r)\frac{f'(-\widetilde{\sigma}_i^{-1}) \left( f(-\widetilde{\sigma}_i^{-1})\right)^{k-1}}{\sigma_i} \right),
\end{equation}
with the convention that 
\begin{equation}\label{eq_defnbi}
\mathtt{b}_i=\left\langle \bm{b}, \bm{v}_i \right\rangle, \ 1 \leq i \leq N. 
\end{equation}
We analogously define $\widetilde{\mathtt{M}}_n$, $\widetilde{\ell}_n$ and $\widetilde{s}_n$ using $\widetilde{\mathfrak{m}}_{k, \bm{b}}.$

For the ease of the statement of our results, we use the following notion of \emph{stochastic domination} which provides precise meaning to a statement of the form ``$x_N$ is bounded by $y_N$ up to a small power of $N$ with high probability".  
\begin{definition}
(i) Let
\[\xi=\left(\xi^{(N)}(u):N \in \mathbb{N}, u\in U^{(N)}\right),\hskip 10pt \zeta=\left(\zeta^{(N)}(u):N \in\mathbb{N}, u\in U^{(N)}\right)\]
be two families of nonnegative random variables defined on the same probability space, where $U^{(N)}$ is a possibly $n$-dependent parameter set. We say $\xi$ is stochastically dominated by $\zeta$, uniformly in $u$, if for any fixed (small) $\epsilon>0$ and (large) $D>0$, 
\[\sup_{u\in U^{(N)}}\mathbb{P}\left(\xi^{(N)}(u)>N^\epsilon\zeta^{(N)}(u)\right)\le n^{-D}\]
for large enough $N \ge N_0(\epsilon, D)$, and we shall use the notation $\xi\prec\zeta$. Throughout this paper, the stochastic domination will always be uniform in all parameters that are not explicitly fixed (such as matrix indices, and $z$ that takes values in some compact set). Note that $N_0(\epsilon, D)$ may depend on quantities that are explicitly constant, such as $\tau$ in Assumption \ref{assum_summary}. If for some complex family $\xi$ we have $|\xi|\prec\zeta$, then we will also write $\xi \prec \zeta$ or $\xi=\OO_\prec(\zeta)$.


(ii) We say an event $\Xi$ holds with high probability if for any constant $D>0$, $\mathbb P(\Xi)\ge 1- N^{-D}$ for sufficiently large $N$.
\end{definition}

Then we summarize the main technical assumptions which will be used throughout this paper.

\begin{assum}\label{assum_summary} We assume that the following assumptions hold:
\begin{enumerate}
\item{\bf On dimensionality} We consider the high-dimensional regime and  assume that (\ref{eq_dimensionality}) holds. 
\item{\bf On $X$ in (\ref{eq_generalmodel}).} For $X=(x_{ij}),$ we assume that $x_{ij}, 1 \leq i \leq N,  1 \leq j \leq M,$ are iid real random variables such that (\ref{eq_meanvariancexij}) and (\ref{eq_momentassumption1}) hold.  

\item{\bf On $\Sigma_0$ in (\ref{eq_defnsigma0}).} We assume that for some small constant $0<\tau_1<1,$ the following holds
\begin{equation}\label{eq_sigma_tau}
\tau_1 \leq \sigma_N \leq \sigma_{N-1} \leq \cdots \leq \sigma_1 \leq \tau_1^{-1}. 
\end{equation}
For definiteness,  we also assume that $\varrho$ is supported on a single bulk component such that
$\operatorname{supp} \varrho =[\gamma_-, \gamma_+]$ and that there exists $\tau_2 > 0$ such that, for a choice of the sign $\pm$, $w(x) := \varrho(x)(\gamma_+ - x)^{-1/2}(x - \gamma_-)^{\pm 1/2}$ and $1/w(x)$ have analytic extensions to $\{z \in \mathbb C : \min_{x \in [\gamma_+,\gamma_-]} |x - z| < \tau_2\}$.  
Moreover, we assume that 
\begin{equation}\label{eq_regularedgeandbulk}
 \gamma_+ \geq \tau_1,\ |\sigma_1^{-1}+m(\gamma_\pm)| \geq \tau_1,
\end{equation}
where, as above, $m(\cdot)$ is the Stieltjes transform of $\varrho.$

\item{\bf On the spikes in (\ref{eq_defnsigma}).} For some fixed integer $r$ and $i \leq r,$ we assume that there exists some constant $\varpi$ such that 
\begin{equation}\label{eq_outlierassumption}
\widetilde{\sigma}_i>-\frac{1}{m(\gamma_+)}+\varpi, \ i \leq r.  
\end{equation}
We also assume that $\widetilde{\sigma}_i, 1 \leq i \leq r,$ are bounded. 
\end{enumerate}
\end{assum}


%
%

The assumption (1) states that we consider the high dimensional regime which is commonly used in the random matrix theory literature. The assumption (2) imposes some conditions for the random matrix $X$. We refer the readers to the discussion below (\ref{eq_momentassumption1}) for more details. 
The assumption (3) is relatively standard in random matrix theory literature. These conditions rule out the existence of spikes in $\Sigma_0$ so that all the possible spikes are generated by those of $\Sigma,$ and also guarantee that $\varrho$ has a regular square root behavior near the edges $\gamma_{\pm}$. These conditions are satisfied by many commonly used examples. We refer the readers to \cite[Definition 2.7]{Knowles2017} for more details and Section \ref{s.example} for examples. Moreover, we mention that $\gamma_{\pm}$ can be fully calculated via $f(x)$ defined in (\ref{eq_defnstitlesjtransform}) as follows. Let $x_-<x_+$ be the critical points of $f(x).$ Then we have that $\gamma_{\pm}=f(x_{\pm}).$

 Finally, assumption (4) imposes the condition that  $\widetilde{\sigma}_i, 1 \leq i \leq r,$ are the spikes (c.f. (\ref{eq_outlierassumption})) which are well-separated from the upper edge with $\OO(1)$ distance. We remark that we can replace $\varpi$ with $\OO(M^{-1/3})$ and allow $\widetilde{\sigma}_i \equiv \widetilde{\sigma}_i(M)$ to diverge with $M$. Since these technical generalizations are not the main focus of the current paper, we do not pursue these generalizations here and leave it as future work. For more details on this aspect, we refer the readers to \cite{BDWW, MR3449395, DRMTA, DYAOS}.

\begin{remark}\label{rmk_assumption}
In this paper, for definiteness and convenience of statement, we assume that the support of $\varrho$ is a single interval. On one hand, a general class of $\Sigma_0$ satisfy this requirement. For example, this condition will be satisfied when the limiting spectral distribution of $\Sigma_0$ is supported on some interval $[a,b] \subset (0, \infty)$ and its density function is bounded from both above and below; see \cite[Example 2.9]{Knowles2017} or \cite[Corollary 3]{MR2308592} for more details. One the other hand, this constraint is expected to be removed in the future. In fact, as stated in  \cite[Lemma 2.4]{DY}, in general,  the support of $\varrho$ is a union of connected components on $\mathbb{R}_+,$ i.e., $\operatorname{supp} \varrho =\bigcup_{k=1}^q [\mathfrak{a}_{2k}, \mathfrak{a}_{2k-1}] \subset (0, \infty), $ where $q$ depends on the ESD of $\Sigma_0.$ As we will see later (c.f. Section \ref{sec_theoryoforthogonalpolynomial}), our arguments rely on the asymptotics of three-term recurrence relation of the orthogonal polynomials associated with $\varrho.$ These asymptotic formulae can only be established for $\varrho$ supported on a single interval (see  \cite{MR2087231}) and do not hold more generally. The generalization to multiple bulk components requires a substantial treatments using the Riemann-Hilbert approach \cite{DeiftOrthogonalPolynomials, MR1702716,MR1469319,MR2022855,PEHERSTORFER2011814,YATTSELEV201573}, which is out of the scope of the current paper. We will pursue this direction in the future, for example see \cite{DT2}.      

\end{remark}


%

\subsection{Lanczos for high-dimensional matrices: deterministic $\bm b$}\label{sec_lanczosanalysis}

We begin with our most critical result concerning the leading-order behavior of the matrix that results from the Lanczos iteration. The results are summarized in Estimation Algorithm \ref{a:lan_generalaccurate} and Theorem \ref{thm:lan_general} below.

\begin{theorem}\label{thm:lan_general}
Fix some small constant $\tau_1>0.$ and suppose Assumption \ref{assum_summary} holds, $\gamma_- \geq \tau_1$, $N \leq M$, and $\| \bm{b} \|_2=1$. Let  $T_k(W,\bm{b})$ and $\mathcal T_k$ denote the upper-left $k \times k$ subblocks the matrices calculated from Steps (1) and (3) of Estimation Algorithm~\ref{a:lan_generalaccurate}, respectively.

Then there exists some constant $\mathtt{C}_{l,k}>0$ such that   
   
    \begin{align}
        T_k(W,\bm{b}) &= \mathcal T_k+{\OO_{\prec}( \mathtt{C}_{l, k} M^{-1/2})},
    \end{align}
    where the approximation is in the sense of operator norm.
    Additionally, 
    \begin{align*}
        \bm{b}^* W^{-1} \bm b = \mathtt{m} +{\OO_{\prec}(  M^{-1/2})},
    \end{align*}
    where $\mathtt{m}=\mathfrak{m}_{-1, \bm{b}}$ for the non-spiked model and $\mathtt{m}=\widetilde{\mathfrak{m}}_{-1, \bm{b}}$ for the spiked model.
\end{theorem}

{ \vspace{0.1in}
\noindent\fbox{%
\refstepcounter{e_alg}
    \parbox{\textwidth}{%
\flushright \boxed{\text{Estimation Algorithm~\arabic{e_alg}: Analysis of the Lanczos iteration \label{a:lan_generalaccurate}}}
\begin{enumerate}[(1)]
    \item Suppose that the Lanczos iteration Algorithm \ref{a:lanczos} applied to the pair $(W, \bm{b})$  runs until step $n \leq N$ in the sense that $b_{n-1} = 0$.  Set $a_k = 1, b_k = 0$ for $k \geq n$.  Let $T(W,\bm{b})$ denote the associated Jacobi matrix.
    \item    \begin{enumerate}
    \item If $W$ is a spiked model as in (\ref{eq_generalmodel}), construct the sequence of $\mathfrak{a}_k$ and $\mathfrak{b}_k$ following
    \begin{equation}\label{eq_lnsn0}
    \mathfrak{b}_k=\frac{\widetilde{\ell}_k}{\widetilde{\ell}_{k+1}}, \ \mathfrak{a}_k=\frac{\widetilde{s}_k}{\widetilde{\ell}_k}-\frac{\widetilde{s}_{k+1}}{\widetilde{\ell}_{k+1}}, \quad k = 0,1,\ldots.
    \end{equation}

   \item Otherwise, if $W \equiv W_0$ is a non-spiked model as in (\ref{eq_definitioncovariance}),  
    construct the sequence of $\mathfrak{a}_k$ and $\mathfrak{b}_k$ following
    \begin{equation}\label{eq_lnsn}
    \mathfrak{b}_k=\frac{{\ell}_k}{{\ell}_{k+1}}, \ \mathfrak{a}_k=\frac{{s}_k}{{\ell}_k}-\frac{{s}_{k+1}}{{\ell}_{k+1}}, \quad k = 0,1,\ldots.
    \end{equation}
    \end{enumerate}
 \item   Build the  Jacobi matrix 
 \begin{equation}\label{eq_defnTnno}
 \mathcal T :=
     \begin{bmatrix} \mathfrak{a}_0 & \mathfrak{b}_0 \\
   \mathfrak{b}_0 & \mathfrak{a}_1 & \mathfrak{b}_1 \\
   & \mathfrak{b}_1 & \mathfrak{a}_2 &\ddots \\
    && \ddots & \ddots  \end{bmatrix}.
    \end{equation}       
\end{enumerate}
    }%
}

\vspace{0.1in}
}

\subsection{The CGA for high-dimensional linear systems: deterministic $\bm{b}$ in (\ref{eq_system})}\label{sec_generalrecipe}  In this subsection, we provide a framework to analyze the residuals and errors of the CGA when applied to (\ref{eq_system}) for some deterministic vector $\bm{b}$ for both spiked and non-spiked covariance matrices.

The framework contains three steps. First, we build up a tridiagonal Jacobi matrix $\mathcal T$ (c.f. (\ref{eq_defnTnno})) utilizing the Hankel moment matrix as in (\ref{eq_hankeldeterminant}). Second, we apply the Jacobi matrix Cholesky factorization algorithm,  Algorithm \ref{a:chol}, to obtain the Cholesky factorization of $\mathcal T$, denoted $\mathcal L$ (c.f. (\ref{eq_cholesky})). Third, we provide the limits and rates based on the entries of $\mathcal L.$  We summarize  the above procedure in Estimation Algorithm \ref{a:generalaccurate}.

{ \vspace{0.1in}
\noindent\fbox{%
\refstepcounter{e_alg}
    \parbox{\textwidth}{%
\flushright \boxed{\text{Estimation Algorithm~\arabic{e_alg}: Error analysis of the CGA \label{a:generalaccurate}}}
\begin{enumerate}[(1)]
    \item Suppose that the Lanczos iteration Algorithm \ref{a:lanczos} applied to the pair $(W, \bm{b})$  runs until step $n \leq N$ and $\bm{r}_n=0.$  Set $\bm{r}_k = 0$ for $k > n$. 
    \item 
    \begin{enumerate}
        \item If $W$ is a spiked model as in (\ref{eq_generalmodel}), construct the sequence of $\mathfrak{a}_k$ and $\mathfrak{b}_k$ following (\ref{eq_lnsn0}), for $k=0,1,\cdots, n-1.$

   \item Otherwise, if $W \equiv W_0$ is a non-spiked model as in (\ref{eq_definitioncovariance}),  
    construct the sequence of $\mathfrak{a}_k$ and $\mathfrak{b}_k$ following (\ref{eq_lnsn}).
 \end{enumerate}
 \item   Build the  Jacobi matrix following (\ref{eq_defnTnno}). 
    
    \item Apply the Jacobi matrix Cholesky factorization (c.f. Algorithm \ref{a:chol}) to $\mathcal T$ to obtain 
     \begin{equation}\label{eq_cholesky}
    \mathcal L = \begin{bmatrix} \alpha_0 \\
   \beta_0 & \alpha_1 \\
   & \beta_1 & \alpha_2 \\
   && \ddots & \ddots
   \end{bmatrix}.
\end{equation}
    \item Based on $\mathcal L$ from Step (4),  employ Theorem \ref{thm:general} below to obtain estimates of the errors encountered in the CGA. 
\end{enumerate}
    }%
}

\vspace{0.1in}
}

Based on Algorithm \ref{a:generalaccurate}, we prove the first order convergence limits and rates for the residuals and errors of the CGA in Theorem \ref{thm:general}.  Denote 
\begin{equation}\label{defn_sk}
\mathcal S_{k} = \mathcal L_{k+1:,k+1:} ~. 
\end{equation}
\begin{theorem}\label{thm:general}
Fix some small constant $\tau_1>0$ and suppose Assumption \ref{assum_summary} holds, $\gamma_- \geq \tau_1$, $N \leq M$, and $\| \bm{b} \|_2=1$. Let $\{\alpha_i\}$ and $\{\beta_j\}$ be the outputs calculated from Step (4) of Algorithm  \ref{a:generalaccurate}. Then we have that with $\vec x_0 = 0$, for $k <n$, there exists some constant $\mathtt{C}_{r,k}>0$ such that   
    \begin{align}\label{eq_residualexpression}
        \|\vec r_k\|_2 &= \prod_{j=0}^{k-1} \frac{\beta_j}{\alpha_j}+{\OO_{\prec}( \mathtt{C}_{r, k} M^{-1/2})}.
    \end{align}    
Recall (\ref{defn_sk}). Moreover, for some constant $\mathtt{C}_{e,k}>0,$ we have that   
  \begin{align}\label{eq_errornorm_Sk}
        \|\vec e_k\|_{W} &= \|\vec r_k\|_2\sqrt{\vec f_1^* (\mathcal S_{k} \mathcal S_{k}^*)^{-1} \vec f_1}+{\OO_{\prec}( \mathtt{C}_{e, k} M^{-1/2})}.
    \end{align}
Recall (\ref{eq_momentdefinition}). Equivalently, we have 
\begin{equation}\label{eq_errornormgeneral}
\|\vec e_k\|_W^2=\mathtt{m}-\frac{1}{\alpha_0^2} \sum_{\ell=0}^{k-1} \prod_{j=1}^\ell \frac{\beta_{j-1}^2}{\alpha_j^2}+{\OO_{\prec}( \mathtt{C}_{e, k} M^{-1/2})}, 
\end{equation}    
where $\mathtt{m}=\mathfrak{m}_{-1, \bm{b}}$ for the non-spiked model and $\mathtt{m}=\widetilde{\mathfrak{m}}_{-1, \bm{b}}$ for the spiked model. 
\end{theorem} 

\begin{remark}\label{rmk_explicit}
Employing Proposition~\ref{prop:inv_moment_minus} below to \eqref{eq_errornormgeneral} gives the following expression
\begin{align*}
    \mathtt{m}-\frac{1}{\alpha_0^2} \sum_{\ell=0}^{k-1} \prod_{j=1}^\ell \frac{\beta_{j-1}^2}{\alpha_j^2} = \left( \prod_{j=0}^{k-1} \frac{\beta_j}{\alpha_j}\right)^2\frac{1}{\alpha_k^2} \sum_{\ell=0}^{k-1} \prod_{j=1}^\ell \frac{\beta_{j+k-1}^2}{\alpha_{j+k}^2},
\end{align*}
which is then used to derive \eqref{eq_errornorm_Sk} by computing $\vec f_1^* (\mathcal S_{k} \mathcal S_{k}^*)^{-1} \vec f_1$ using forward substitution.
\end{remark}

\begin{remark}
Theorem \ref{thm:general} provides a first order description for the CGA applied to the linear system with deterministic $\bm{b}.$ The assumption that $\bm{b}$ is a unit vector is just to ease the statement of the results and can be removed by minor modification. 
The constants $\mathtt{C}_{e,k}$ and $\mathtt{C}_{r,k}$ crucially depend on $k.$ As we can see in the proof of Theorem \ref{thm:general},  these constants can be trivially bounded by $\mathsf{a}^k,$ for some constant $\mathsf{a}>1.$ In this sense, the error becomes negligible for $k \leq  C\log N$ where $C>0$ is some universal constant. The discussion of the optimal choices of these constants are out of the scope of the current paper. We will pursue this direction in the future work; for example, see \cite{DT2}.    
\end{remark}

Theorem \ref{thm:general} provides us the general error analysis for CGA with a general covariance matrix. As we can see from Steps (1)--(3) of Algorithm \ref{a:generalaccurate}, it requires a large amount of non-trivial  computations in order to obtain the Jacobi matrix. However, under certain conditions of $W$ and $\bm{b}$, we can simplify Algorithm \ref{a:generalaccurate} and provide a simpler but less exact estimate. We find closed-form estimates for $\{\alpha_i\}$, $\{\beta_j\}$, $\| \bm{r}_k\|_2$ and $\|\bm{e}_k\|_W$ in the rest of this subsection.  
The framework is summarized in Estimation Algorithm \ref{a:simple}.  

\vspace{0.1in}

\noindent\fbox{%
\refstepcounter{e_alg}
    \parbox{\textwidth}{%
\flushright \boxed{\text{Estimation Algorithm~\arabic{e_alg}: Asymptotic analysis of the CGA for general model \label{a:simple}}}
\begin{enumerate}
    \item Calculate the support of $\varrho$ using $f$ in (\ref{eq_defnstitlesjtransform}). More specifically, calculate the critical points of $f$ as $x_{\pm}$ and the corresponding edges $\gamma_{\pm}=f(x_{\pm}).$ 
    \item Based on (1), set 
    \begin{equation}\label{eq_aabb}
    \mathfrak{a}=\frac{\gamma_++\gamma_-}{2}, \ \mathfrak{b}=\frac{\gamma_+-\gamma_-}{4}.
    \end{equation}
     Build the Jacobi matrix $\mathcal T$ as in (\ref{eq_defnTnno}) by setting
 \begin{equation}\label{eq_defnTnn}
\mathfrak{a}_k \equiv \mathfrak{a}, \  \mathfrak{b}_k \equiv \mathfrak{b}, \ k \geq 0.  
    \end{equation}      
     \item Apply Jacobi matrix Cholesky factorization (c.f. Algorithm \ref{a:chol}) to $\mathcal T$ obtained from Step (2) and get the Cholesky factorization $\mathcal L$ as in (\ref{eq_cholesky}).  
     \item  Based on $\mathcal L$ from Step (3),  employ Theorem \ref{thm:mainone} below to obtain estimates of the errors encountered in the CGA. 
\end{enumerate}
    }%
}
\vspace{0.1in}








Compared to Estimation Algorithm \ref{a:generalaccurate}, the simplified algorithm, Estimation Algorithm \ref{a:simple} does not required the calculations of Hankel moment matrices and the related quantities. Instead, it only relies on the edges of the support of the deformed MP law, which can be easily calculated using the function in (\ref{eq_defnstitlesjtransform}). The calculation workload is significantly reduced.  Based on Estimation Algorithm \ref{a:simple}, we can establish Theorem \ref{thm:mainone} for the non-spiked covariance matrix or the spiked covariance matrix with certain choices of $\bm{b}$, which gives an asymptotic convergence rate for both the residual and error vectors.
%
\begin{theorem}\label{thm:mainone} Fix some small constant $\tau_1>0.$ Suppose Assumption \ref{assum_summary}(1-3) hold, $\gamma_-\geq \tau_1$, $N \leq M$ and $\| \bm{b} \|_2=1$. Let $\{\alpha_i\}$ and $\{\beta_j\}$ be the outputs calculated from Step (3) of Algorithm  \ref{a:simple}. Then we have that with $\vec x_0 = 0$, for $1 \leq k$: \\

\noindent (1)  For some constants $\mathtt{C}_{r,k}>0$, $c > 0$
    \begin{align*}
        \frac{\|\vec r_k(W_0, \bm{b})\|_2}{\|\vec r_{k-1}(W_0, \bm{b})\|_2} = \frac{\sqrt{\gamma_+} - \sqrt{\gamma_-}}{\sqrt{\gamma_+} +  \sqrt{\gamma_-}}+\OO_{\prec}( \mathtt{C}_{r,k} M^{-1/2})+\OO(e^{-ck}).
    \end{align*}

    \noindent (2)  For some constants $\mathtt{C}_{e,k}>0$, $c > 0$
    \begin{align*}
        \frac{\|\vec e_k(W_0, \bm{b})\|_W}{\|\vec e_{k-1}(W_0, \bm{b})\|_W} = \frac{\sqrt{\gamma_+} - \sqrt{\gamma_-}}{\sqrt{\gamma_+} +  \sqrt{\gamma_-}}+\OO_{\prec}( \mathtt{C}_{e,k} M^{-1/2})+\OO(e^{-ck}).
    \end{align*} 
    
\noindent In addition, suppose Assumption \ref{assum_summary}(4) holds and suppose for each $i = 1,2,\ldots,r$ that either
\begin{align}\label{eq_conditionb}
    \langle \bm b, \bm{v}_i \rangle = 0 \quad \text{or} \quad |\langle \bm b, \bm{v}_i \rangle| \geq \tau_1.
\end{align}
Then: \\

\noindent (3)  For some constants $\mathtt{C}_{r,k}>0$, $c > 0$
    \begin{align*}
        \frac{\|\vec r_k(W, \bm{b})\|_2}{\|\vec r_{k-1}(W, \bm{b})\|_2} = \frac{\sqrt{\gamma_+} - \sqrt{\gamma_-}}{\sqrt{\gamma_+} +  \sqrt{\gamma_-}}+\OO_{\prec}( \mathtt{C}_{r,k} M^{-1/2})+\OO(e^{-ck}).
    \end{align*}

    \noindent (4)  For some constants $\mathtt{C}_{e,k}>0$, $c > 0$
    \begin{align*}
        \frac{\|\vec e_k(W, \bm{b})\|_W}{\|\vec e_{k-1}(W, \bm{b})\|_W} = \frac{\sqrt{\gamma_+} - \sqrt{\gamma_-}}{\sqrt{\gamma_+} +  \sqrt{\gamma_-}}+\OO_{\prec}( \mathtt{C}_{e,k} M^{-1/2})+\OO(e^{-ck}).
    \end{align*} 


\end{theorem}
\begin{remark}
In the case that $\bm b \in \mathrm{span}\{\bm v_1, \ldots, \bm v_r\}$ the calculations can be made more explicit in the sense that the Jacobi matrix $\mathcal T$ determined by $W$ and $\bm b$ \eqref{eq_defnTnno} can be written explicitly in terms quantities used in the analysis of the CGA applied to $W_0 \bm x = \bm b$.
\end{remark}

\begin{remark}
The formulas in Theorem \ref{thm:mainone} are explicit and only need the edges of the support of $\varrho.$ In fact, in many examples, the edges also have known formulas. For example, when $\Sigma_0=I,$ we have that $\gamma_{\pm}=(1\pm \sqrt{c_N})^2.$ Moreover, when the limiting spectral distribution of $\Sigma_0$ follows Marchenko--Pastur law with the same parameter $c_N,$ we have that (c.f. Lemma \ref{lem_expcilitformula})
\begin{equation}\label{eq_invariantformula}
\gamma_{\pm}=\frac{-1+20 c_N^{-1} +8 c_N^{-2} \pm (1+8 c_N^{-1})^{3/2}}{8 c_N^{-2}}.
\end{equation}  
For more general settings, we employ $f$ in (\ref{eq_defnstitlesjtransform}) to calculate the support using Newton's method. We refer the readers to Section \ref{s.example} for more examples.  
\end{remark}

\begin{remark}
    In the statement of Theorem \ref{thm:mainone} the potential vanishing of $\bm{r}_{k-1}$ appears to be ignored.  But, indeed, Theorem~\ref{thm:general} establishes that it does not vanish with high probability.
\end{remark}

Based on the formulas in Theorems ~\ref{thm:general} and \ref{thm:mainone}  we can derive expressions for the halting times of the CGA for the non-spiked model. Similar results hold for spiked model when $\bm{b}$ satisfies (\ref{eq_conditionb}). Define two CGA halting times as 
\begin{equation*}
t^{\bm{e}}(W_0, \bm{b}, \epsilon)=\min\{k: \| \bm{e}_k(W_0, \bm{b}) \|_{W_0}<\epsilon\}, \ t^{\bm{r}}(W_0, \bm{b}, \epsilon)=\min\{k: \| \bm{r}_k(W_0, \bm{b}) \|_2 <\epsilon\}. 
\end{equation*}
We summarize the results in the following theorem. Define deterministic halting times
\begin{align*}
    \tau^{\bm{e}}(\mathcal L, \epsilon) &= \min\left\{ k: \mathtt{e}_k(\mathcal L) < \epsilon \right\}, \quad \mathtt{e}_k(\mathcal L) : = \left( \prod_{j=0}^{k-1} \frac{\beta_j}{\alpha_j}\right)\frac{1}{\alpha_k} \sum_{\ell=0}^{k-1} \prod_{j=1}^\ell \frac{\beta_{j+k-1}}{\alpha_{j+k}}\\
    \tau^{\bm{r}}(\mathcal L, \epsilon) &= \min\left\{ k:  \mathtt{r}_k(\mathcal L)  < \epsilon \right\}, \quad \mathtt{r}_k(\mathcal L) : = \prod_{j=0}^{k-1} \frac{\beta_j}{\alpha_j}.
\end{align*}

\begin{theorem}\label{thm_haltingtime}
Suppose the assumptions of Theorem \ref{thm:mainone} hold.   Let $\mathcal L$ be as in \eqref{eq_cholesky}. \\ 
\noindent (1) If $\mathtt{r}_k(\mathcal L) \neq \epsilon$ for all $k$ then
\begin{align*}
    \lim_{M \to \infty} \mathbb P \left( t^{\bm{r}}(W_0, \bm{b}, \epsilon) = \tau^{\bm{r}}(\mathcal L, \epsilon) \right) =  1.
\end{align*}

\noindent (2) If $\mathtt{e}_k(\mathcal L) \neq \epsilon$ for all $k$ then
\begin{align*}
    \lim_{M \to \infty}\mathbb P \left( t^{\bm{e}}(W_0, \bm{b}, \epsilon) = \tau^{\bm{e}}(\mathcal L, \epsilon) \right) = 1.
\end{align*}

\noindent Since $\mathtt{e}_k(\mathcal L)$ is strictly decreasing, if $\mathtt{e}_K(\mathcal L) = \epsilon$ for some $K$ then as $M \to \infty$
\begin{align*}
    \mathbb P &\left( t^{\bm{e}}(W_0, \bm{b}, \epsilon) = \tau^{\bm{e}}(\mathcal L, \epsilon) \right) = p_M + o(1) \\
    \mathbb P &\left( t^{\bm{e}}(W_0, \bm{b}, \epsilon) = \tau^{\bm{e}}(\mathcal L, \epsilon)  + 1\right) = 1-p_M + o(1).
\end{align*}

\end{theorem}

We note that it is conjectured that one can take $p_M = \frac 1 2$ in the above theorem.  This will be established in a future work.

\begin{remark}
Often, in our numerical experiments, the estimate in Theorem~\ref{thm:mainone} appears to set in almost immediately in the sense that the finite-size matrix effects dominate the deviation from the first-order limit. Thus one might expect that
\begin{align*}
    \|\bm{r}_k(W_0,\bm b)\|_2 = \prod_{j=0}^{k-1} \left[ \frac{\sqrt{\gamma_+} - \sqrt{\gamma_-}}{\sqrt{\gamma_+} + \sqrt{\gamma_-}} ( 1 + E_j)\right]
\end{align*}
where $\mathtt{M_k} : = \prod_{j = 0}^{k-1} (1 + E_j)$ converges rapidly, or may even be nearly one.  Set $\mathtt{J} = \frac{\sqrt{\gamma_+} - \sqrt{\gamma_-}}{\sqrt{\gamma_+} + \sqrt{\gamma_-}}$ and then following is a very good first approximation to the halting time
\begin{align*}
    t^{\bm r}(W_0,\bm b,\epsilon) \approx \left\lceil \frac{\log \epsilon - \log (\lim_{k \to \infty}\mathtt{M_k})}{\log \mathtt{J}} \right\rceil.
\end{align*}
And even dropping $\log \lim_{k \to \infty} M_k$ contribution entirely often only effects the halting time estimate by an iteration or two, or maybe not at all.
\end{remark}

\subsection{The CGA for high-dimensional regression: random $\bm{b}$ in (\ref{eq_system})}\label{sec_sub_normalequation} In this subsection, we consider the scenario for the CGA  when applied to (\ref{eq_system}) for a specific random vector $\bm{b},$ which concerns the high dimensional linear regression via the normal equation. More specifically, denote $Y=\Sigma_0^{1/2} X$ or $\Sigma^{1/2} X$, for some deterministic vector $\bm{a} \in \mathbb{R}^M,$ and consider
\begin{equation}\label{eq_normalequationform}
YY^* \bm{x}=\bm{b}, \ \ \bm{b}=Y \bm{a}.
\end{equation}

As we will see in Theorem \ref{thm_normalequation},
the main difference between this random scenario and the deterministic case in Section \ref{sec_generalrecipe} is that, the spikes of $\Sigma$ will not affect the errors and residuals generated by the CGA. We first propose an algorithm analogous to Algorithm \ref{a:generalaccurate}. Denote
\begin{equation}\label{eq_momentnormalequation}
\mathtt{m}_k= \frac{1}{\sqrt{\mathtt{w}}} \int_{\mathbb{R}} \lambda^{k+1} \varrho(\lambda) \dd \lambda,  
\end{equation} 
where we recall that $\varrho$ is the asymptotic density function of the deformed MP law and
\begin{equation}\label{eq_averagetrace}
\mathtt{w}=\frac{1}{M} \sum_{i=1}^N \sigma_i. 
\end{equation}

Similar to (\ref{eq_hankeldeterminant}) and (\ref{eq:sandl}), we can define analogous quantities $\mathsf{l}_n$ and $\mathsf{s}_n$ using $\mathtt{m}_k$ as in (\ref{eq_momentnormalequation}).  The CGA for high-dimensional linear regression is summarized in the following algorithm. 

\vspace{0.1in}
\noindent\fbox{%
\refstepcounter{e_alg}
    \parbox{\textwidth}{%
\flushright \boxed{\text{Estimation Algorithm~\arabic{e_alg}: Analysis of the CGA for high-dimensional linear regression \label{a:generalnormalequation}}}
\begin{enumerate}[(1)]
    \item  Calculate the sequence $\{\mathfrak{a}_n\}$ and $\{\mathfrak{b}_n\}$ following 
        \begin{equation*}
    \mathfrak{b}_n=\frac{{\mathsf{l}}_n}{{\mathsf{l}}_{n+1}}, \ \mathfrak{a}_n=\frac{{\mathsf{s}}_n}{{\mathsf{l}}_n}-\frac{{\mathsf s}_{n+1}}{{\mathsf l}_{n+1}}.
    \end{equation*}
      
    \item Follow Steps (3)--(4) of Estimation Algorithm \ref{a:generalaccurate} to obtain the matrix $\mathcal L$ in \eqref{eq_cholesky}.
    \item Apply Theorem  \ref{thm_normalequation} to obtain estimates. 
\end{enumerate}
    }%
}

\vspace{0.1in}

\begin{remark}
Compared to Estmation Algorithm \ref{a:generalaccurate}, Estimation Algorithm \ref{a:generalnormalequation} has two major differences. First, the Hankel moment matrices are constructed using the deformed MP law directly (c.f. (\ref{eq_momentnormalequation})) whereas Algorithm \ref{a:generalaccurate} utilizes the density (\ref{eq_varrhob}). It can be seen that $\varrho_{\bm{b}}$ depends on the explicit form of $\bm{b}$ in (\ref{eq_varrhob}) but $\varrho$ is independent of the choice of $\bm{a}$ as in (\ref{eq_normalequationform}). Second, in Estimation Algorithm \ref{a:generalaccurate}, we need to use different Hankel moment matrices for the spiked and non-spiked  models. In contrast, when the CGA is applied to the normal equations, we always use the same moment regardless of the spikes. For a more precise statement, see (\ref{eq_closeone}) and (\ref{eq_closetwo}). 
\end{remark}

Based on Estimation Algorithm \ref{a:generalnormalequation}, we establish the theoretical results in Theorem \ref{thm_normalequation}. 
\begin{theorem}\label{thm_normalequation}
Fix some small constant $\tau_1>0.$
Suppose Assumption \ref{assum_summary} holds, $\gamma_- \geq \tau_1$ and $\| \bm{a}\|_2=1$. Let $\{\alpha_i\}$ and $\{\beta_j\}$ be the outputs calculated from Step (2) of Algorithm  \ref{a:generalnormalequation}.  Denote $Y=\Sigma^{1/2}X$ and $Y_0=\Sigma^{1/2}_0 X.$ Then for the non-spiked model,  there exist some constants $\mathtt{C}_{r,k}, \mathtt{C}_{e,k}>0$ such that 
\begin{equation*}
\| \bm{r}_k(W_0, Y_0 \bm{a})\|_2=\sqrt{\mathtt{w}}\prod_{j=0}^{k-1} \frac{\beta_j}{\alpha_j}+{\OO_{\prec}( \mathtt{C}_{r, k} M^{-1/2})},
\end{equation*}
and for $\mathcal S_k = \mathcal S_k(\mathcal)$ defined in (\ref{defn_sk})
\begin{equation*}
\| \bm{e}_k(W_0, Y_0 \bm{a})\|_{W_0}= \|\bm{r}_k(W_0, Y_0 \bm{a})\|_2\sqrt{\vec f_1^* (\mathcal S_{k} \mathcal S_{k}^*)^{-1} \vec f_1} +{\OO_{\prec}( \mathtt{C}_{e, k} M^{-1/2})},
\end{equation*}
or equivalently 
\begin{equation}\label{eq_simplenormal}
\| \bm{e}_k(W_0, Y_0 \bm{a})\|_{W_0}^2= \mathtt{w} \left( 1- \frac{1}{\alpha_0^2} \sum_{\ell=0}^{k-1} \prod_{j=1}^\ell \frac{\beta_{j-1}^2}{\alpha_j^2}\right)  +{\OO_{\prec}( \mathtt{C}_{e, k} M^{-1/2})},
\end{equation} 
where $\mathtt{w}$ is defined in (\ref{eq_averagetrace}).

Additionally,  for the spiked model, we have that 
\begin{equation}\label{eq_closeone}
\| \bm{r}_k(W, Y \bm{a})\|_2=\|\bm{r}_k(W_0, Y_0 \bm{a})\|_2+\OO_{\prec}(\mathtt{C}_{r,k} M^{-1/2}), 
\end{equation}
and 
\begin{equation}\label{eq_closetwo}
 \|\bm{e}_k(W, Y \bm{a})\|_W=\|\bm{e}_k(W_0, Y_0 \bm{a})\|_{W_0}+\OO_{\prec}(\mathtt{C}_{e,k} M^{-1/2}). 
\end{equation}
\end{theorem}

\begin{remark}\label{rek_normalequation}
We remark that compared to Theorem \ref{thm:general}, where the CGA is applied for a deterministic $\bm{b},$ Theorem \ref{thm_normalequation} exhibits several differences. First, an extra normalization constant $\mathtt{w}$ is used. In fact, $\mathtt{w}=\mathbb{E}\| Y_0 \bm{a} \|_2$ is used to scale $Y_0 \bm{a}$ such that the Lanczos Iteration, Algorithm \ref{a:lanczos} can be applied properly. Second, compared to (\ref{eq_errornormgeneral}), (\ref{eq_simplenormal}) has a simpler form  due to (\ref{eq_normalequationform}).  Third, (\ref{eq_closetwo})  implies that if we examine the performance of the CGA using the error norm $\|\cdot \|_{W},$ the spikes will be ignored. Therefore, even though this measurement is standard in numerical analysis,   for statisticians who are interested in understanding the performance of the estimation of high dimensional ordinary least square (OLS) coefficients, a better norm (i.e., loss function) should be considered and studied. We will pursue this direction in the future works.  

We point out that when $\Sigma_0=I,$  \cite{Paquette2020} used another approach to obtain a weak convergence formula. Their method relies on exploring the structure of the error. However, this method was not extended to give expressions for quantities beyond the $W$-norm of the error. 
Our methods amount to a combination of the generality of the distributions considered in \cite{Paquette2020} with the generality of the norms considered in \cite{Deift2019a} while extending it to general spiked covariance matrices.
Additionally, we can construct similar results based on asymptotic relations of the orthogonal polynomials as in Algorithm \ref{a:simple} and Theorem \ref{thm:mainone} as the Jacobi matrix $\mathcal T$ that is used to construct $\mathcal L$ in step (2) of Estimation Algorithm~\ref{a:generalnormalequation} is just the Jacobi matrix associated to the modified density $\frac{\lambda}{\sqrt{\mathtt w}} \varrho(\lambda)$. We omit the details here. 
\end{remark}

\begin{remark}
We have now demonstrated a guiding principle. We know that for $\|\bm b\|_2 =1$
\begin{align*}
    \bm{b}^* W^k \bm{b} = \int_{\mathbb R} \lambda^k \varrho_{\bm b}(\lambda) \sd \lambda +{\OO_{\prec}( \mathtt{C}_{k} M^{-1/2})},,
\end{align*}
and hence the performance of the CGA on $W\bm{x} = \bm b$ will be, up to some error, determined by the three-term recurrence for the orthogonal polynomials for $\varrho_{\bm b}(\lambda) \sd \lambda$. 

Theorem~\ref{thm_normalequation} relies on the fact that for $\|\bm a\|_2 =1$
\begin{align*}
    \bm{a}^* Y^* W^k Y \bm{a} = \bm{a}^* \mathcal W^{k+1} \bm{a} = \int_{\mathbb R} \lambda^{k} \frac{\lambda \varrho(\lambda)}{\sqrt{\mathtt w}} \sd \lambda +{\OO_{\prec}( \mathtt{C}_{k} M^{-1/2})}, \quad \mathcal W = Y^* Y.
\end{align*}
Combining these two facts allows one to analyze the classical regression problem \eqref{eq_plus_noise}. With
\begin{align*}
    \bm{b} = Y(Y^* \bm{z} + \bm{\epsilon}), \quad \|\bm z \|_2 =1, \quad \|\bm \epsilon\|_2 = 1 +{\OO_{\prec}(M^{-1/2})},
\end{align*}
one sees
\begin{align*}
    \bm{b}^* W^k \bm{b} = \bm{z}^* W^{k+2} \bm{x} + \bm{\epsilon}^* \mathcal W^{k+1} \bm{\epsilon} + 2 \bm{x}^* W^{k+1} Y \bm{\epsilon}.
\end{align*}
Supposing $\bm{\epsilon}$ is isotropic and independent of $W$, the last term has expectation zero and the asymptotic performance of the CGA on this regression problem will be determined by the three-term recurrence for the orthogonal polynomials for
\begin{align*}
    \left(\lambda^2 \varrho_{\bm b}(\lambda) + \frac{\lambda \varrho(\lambda)}{\sqrt{\mathtt w}} \right) \sd \lambda.
\end{align*}
This observation was previously made  in \cite{Paquette2020a}.  And by Theorem~\ref{thm_relationOP} the asymptotics of this three-term recurrence is determined by the support of the measure alone when the supports of $\varrho_{\bm b}$ and $\varrho$ coincide
\end{remark}


\subsection{Universality}\label{s.universality}
In this subsection, we establish the universality of the fluctuations of the norms of the error and residual vectors for the CGA. It demonstrates that the second order fluctuations of the residuals and errors of the CGA depend only on the first four moments of the entries $(x_{ij})$ for both spiked and non-spiked models.  

\begin{theorem}\label{thm:mainthree}
Suppose Assumption \ref{assum_summary} holds. 
Let $W$ be as in (\ref{eq_definitioncovariance}) and let $\widetilde{W}^Y$ be defined similarly by replacing $X$ with another random matrix $Y=(y_{ij})$ which satisfies (2) of Assumption \ref{assum_summary}. Moreover, assume that
\begin{equation}\label{assum_momentmatching}
\mathbb{E}x_{ij}^l=\mathbb{E} y_{ij}^l, \ 1 \leq l \leq 4,  \ 1 \leq i \leq N,  1 \leq j \leq M. 
\end{equation}
Then we have that for all $s_{i1}, s_{i2} \in \mathbb{R}$, $1 \leq i \leq k$,
\begin{align*}
&\lim_{N \rightarrow \infty}\left[ \mathbb{P}^X \left(  \left(M^{1/2}\left[ \| \bm{r}_i(W, \bm{b}) \|_2 - \mathtt{r}_i(\mathcal L) \right] \leq s_{i1},   M^{1/2} \left[\| \bm{e}_i(W, \bm{b}) \|_W - \mathtt{e}_i(\mathcal L) \right]\leq s_{i2} \right)_{1 \leq i \leq k}  \right) \right. \\
&- \left. \mathbb{P}^Y \left(  \left(M^{1/2}\left[ \| \bm{r}_i(W^Y, \bm{b}) \|_2 - \mathtt{r}_i(\mathcal L) \right] \leq s_{i1},   M^{1/2} \left[\| \bm{e}_i(W^Y, \bm{b}) \|_{W^Y} - \mathtt{e}_i(\mathcal L) \right]\leq s_{i2} \right)_{1 \leq i \leq k}  \right) \right] = 0
\end{align*} 

where $\mathbb{P}^X$ and $\mathbb{P}^Y$ denote the laws of $(x_{ij})$ and $(y_{ij}),$ respectively, and $\mathcal L$ is defined in \eqref{eq_cholesky}.
\end{theorem} 

\begin{remark}
Theorem \ref{thm:mainthree} proves the universality for the distributions of the errors and residuals. 
We point out that the exact distributions for the residuals and errors are generally unknown even when $X$ is Gaussian. To our best knowledge, these results are only established in the null case when $\Sigma=I$  in \cite{Paquette2020}. For general covariance matrix and spiked model, it requires more careful treatment and is beyond the scope of the current paper. We will consider this problem in the future work (c.f. \cite{DT2}). 
\end{remark}

\begin{remark}\label{rmk_statistics}
We remark that Theorem \ref{thm:mainthree} can be used to conduct statistical inference on the structure of population covariance matrix. For example, in the literature \cite{MR3468554}, researchers are particularly interested in testing 
\begin{equation*}
\mathbf{H}_0: \Sigma=\Lambda_0,
\end{equation*} 
where $\Lambda_0$ is some given positive definite matrix. We focus on our explanation on the non-spiked model. Many statistics can be constructed based on Theorems \ref{thm:general} and \ref{thm:mainone}, or Theorem \ref{thm_haltingtime}. Even though the distributions of the halting times are unknown, according to Theorem \ref{thm:mainthree}, when the fourth moment is assumed to be $3$, we can always simulate their distributions using Gaussian random variables. In this sense, Theorems \ref{thm:general} and \ref{thm:mainthree} can be combined to provide new statistics for high-dimensional inference. This opens a new door for high-dimensional statistics and demonstrates that in contrast to the standard testing procedure where testing statistics are mostly based on the estimation procedure, we can also propose useful statistics based on the computational and algorithmic viewpoint. We will pursue this direction in the future works.  
\end{remark}

\subsection{Some extensions and discussion}\label{sec_otheralgo} 

We employ the error analysis framework established in Section \ref{sec_generalrecipe} to analyze the minimal residual algorithm (MINRES) \cite{MR383715}.  The actual algorithm is recorded in Algorithm \ref{a:minres} in Appendix \ref{appendix_minres}. Similar to the CGA, MINRES is applied to solve linear systems of the form $W \bm{x} = \bm{b}$, $W \in \mathbb R^{N \times N}$ but for MINRES $W$ need not be definite.  MINRES  can also be described in its varational form. Recalling (\ref{eq_krylovspace}), MINRES, at iteration $k$, gives the solution of 
\begin{align*}
\bm{x}_k=\operatorname{argmin}_{\bm{y} \in \mathcal{K}_k} \| \bm{b}-W\bm{y} \|_2. 
\end{align*} 
For simplicity, we focus on analyzing the residuals of MINRES using Estimation Algorithm \ref{a:generalaccurate}. The results are collected in Theorem \ref{thm_minres}. 
\begin{theorem}\label{thm_minres} 
Fix some small constant $\tau_1>0.$ Suppose Assumption \ref{assum_summary} holds, $\gamma_- \geq \tau_1$ and $\| \bm{b} \|_2=1.$ Let $\{\alpha_i\}$ and $\{\beta_j\}$ be the outputs calculated from Step (4) of Algorithm  \ref{a:generalaccurate}.  Then we have that with $\vec x_0 = 0$, for $k <n$, there exists some constant $\mathtt{C}_{r,k}>0$ such that   
    \begin{align*}
        \|\vec r_k\|_2 &= \left( \sum_{j=0}^k \prod_{\ell=0}^{j-1} \frac{\alpha^2_{\ell}}{\beta_{\ell}^2} \right)^{-1/2} +{\OO_{\prec}( \mathtt{C}_{r, k} M^{-1/2})}.
    \end{align*}    
\end{theorem}  

We point out that even though Estimation Algorithm \ref{a:generalaccurate} is designed for the error analysis for the CGA, it can also be used to analyze the residuals of MINRES because MINRES is also closely connected to the Lanczos iteration. Compared to (\ref{eq_residualexpression}) for the CGA, the main difference lies in the leading order expression. These expressions are derived deterministically using the variational forms of these algorithms. In this sense, any numerical algorithm which is based on the Krylov space $\mathcal K_k$ and has errors that depend only on the matrix $\mathcal L$ constructed in \eqref{eq_cholesky} can be analyzed using our proposed framework. 


\section{Examples and numerical simulations}    \label{s.example} In what follows, we provide a few examples satisfying our assumptions, with accompanying numerical simulations, to better explain the calculations and illustrate our theoretical results. We focus on the discussion on $\Sigma_0$, the construction of $f(x)$ and the edges of $\varrho$ since they are the essential quantities.    We mention that there exist many other important examples of $\Sigma_0$, beyond which we discuss, having been used in applications that satisfy our assumptions. For instance, one can consider $\Sigma_0$ such that its limiting ESD satisfies either the truncated Gamma distribution in \cite{KML} or some Jacobi measure as in \cite{DJ}. All these cases can be analyzed using our methods.  For our numerical experiments we effectively keep $c_N$ fixed by setting $M = \lfloor N/r \rfloor$ for $r$ fixed.

In some situations, see \eqref{eq_JSM_formula}, we know the first-order limit of the norms of the residual and error vectors $\bm r_k$, $\bm e_k$.  In other situations, we do not.  When we do not we either estimate or derive the bulk edges $\gamma_\pm$ --- estimation involves rootfinding on $f'(x)$.  This then gives the large $k$ behavior of the first-order limits via Theorem~\ref{thm:mainone}. For small $k$ we take the following estimation approach:
\begin{itemize}
    \item Using a single sample with $N = 2000$, compute the Lanczos matrix $T_\ell(W,\bm b)$, for $\ell$ small (all all plots we use $\ell = 5$).
    \item Extend $T_k$ to an approximation of $\mathcal T$ by setting $\mathfrak a_k = \mathfrak a$, $\mathfrak b_{k-1} = \mathfrak b$ as in \eqref{eq_defnTnn} for $k \geq \ell$.
    \item Lastly, use Theorem~\ref{thm:general} to give an estimate of the first-order limits of $\|\bm r_k\|_2$ and $\|\bm e_k\|_W$.
\end{itemize}

%

\subsection{Johnstone's spiked covariance matrix model \cite{Johnstone2001}}  We consider the standard spiked covariance matrix model when $c_N<1$. In this case, $\Sigma_0=I$ and the rank-one spiked model
\begin{equation}\label{eq_JSM}
\Sigma=I+ \ell \bm{v} \bm{v}^*.
\end{equation} 
It is clear that (3) of Assumption \ref{assum_summary} is satisfied. Moreover,    according to (\ref{eq_defnstitlesjtransform}), we have that 
\begin{equation*}
f(x)=-\frac{1}{x}+\frac{c_N}{x+1}, \ f'(x)= \frac{1}{x^2}-\frac{c_N}{(x+1)^2}.
\end{equation*}
Consequently, we have that its critical points and the edges of the support are
\begin{equation*}
\gamma_+=(1+\sqrt{c_N})^2, \ x_+=-\frac{1}{\sqrt{c_N}+1}; \ \gamma_-=(1-\sqrt{c_N})^2, \ x_-=\frac{1}{\sqrt{c_N}-1}.
\end{equation*}
Therefore, it is easy to see that $\alpha_j \equiv 1 + O(e^{-c n})$ and $\beta_j = \sqrt{c_N} + O(e^{-c n}).$ According to the bidiagonalization in \cite{MR3078286} for the Gaussian case, if $\bm b = \bm v$ we have,
\begin{align*}
    \mathcal L = \begin{bmatrix} \sqrt{1 + \ell} \\ \sqrt{c_N} & 1 \\ & \sqrt{c_N} & 1 \\
    && \ddots & \ddots \end{bmatrix}.
\end{align*}
Supposing that $c_N \To[N] d$, this gives the formulae
\begin{align}\label{eq_JSM_formula}
\begin{split}
    \| \bm r_k(W,\bm v)\|_2 = \sqrt{\frac{d}{1 + \ell}} \begin{cases}  1 + {\OO_{\prec}( \mathtt{C}_{r, k} M^{-1/2})}& k = 1,\\
    d^{(k-1)/2} + {\OO_{\prec}( \mathtt{C}_{r, k} M^{-1/2})}& k > 1,\end{cases}\\
    \| \bm e_k(W,\bm v)\|_W = \sqrt{\frac{d}{(1 + \ell)(1 - d)}} \begin{cases}  1+ {\OO_{\prec}( \mathtt{C}_{r, k} M^{-1/2})} & k = 1,\\
    d^{(k-1)/2} + {\OO_{\prec}( \mathtt{C}_{r, k} M^{-1/2})}& k > 1.\end{cases}
    \end{split}
\end{align}
We demonstrate the convergence of the CGA in Figure~\ref{fig:johnstone}.  In Figure~\ref{fig:johnstone2} we modify the projection of $\bm b$ onto $\bm v$.  We demonstrate the case of two distinct spikes in Figure~\ref{fig:two-spike}.

\begin{figure}[tbp]
    \centering
    \includegraphics[width=.48\linewidth]{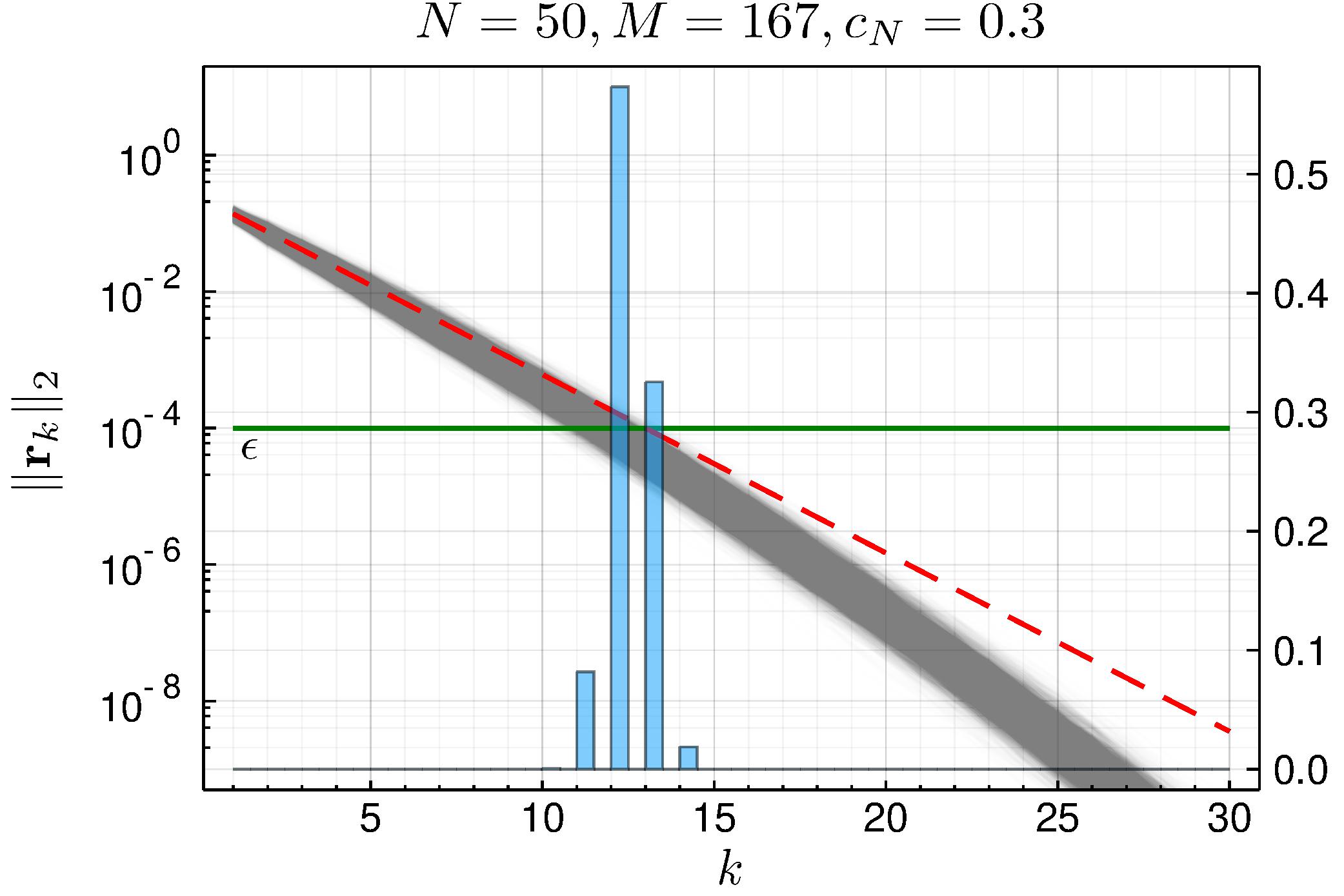}
    \includegraphics[width=.48\linewidth]{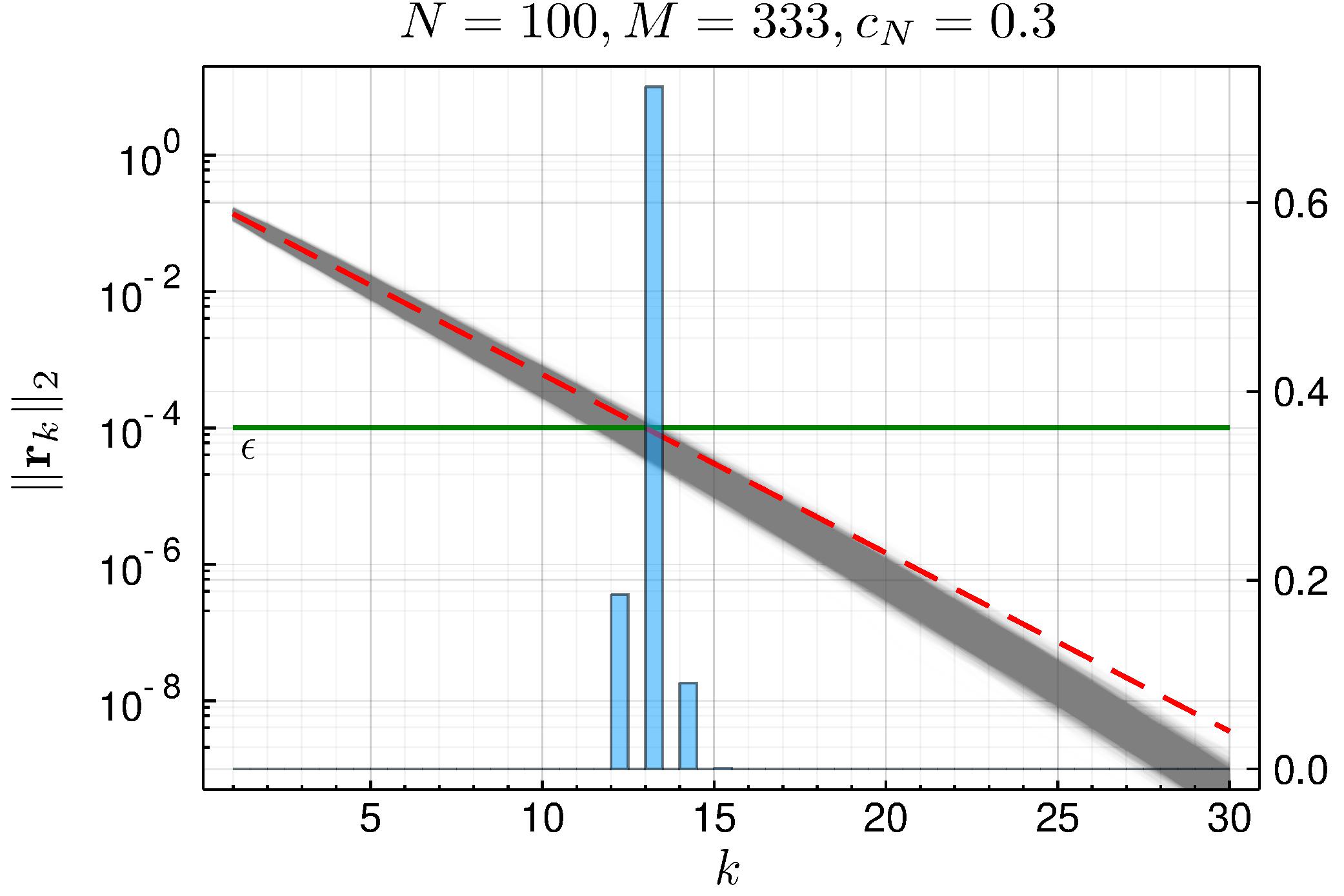}
    \includegraphics[width=.48\linewidth]{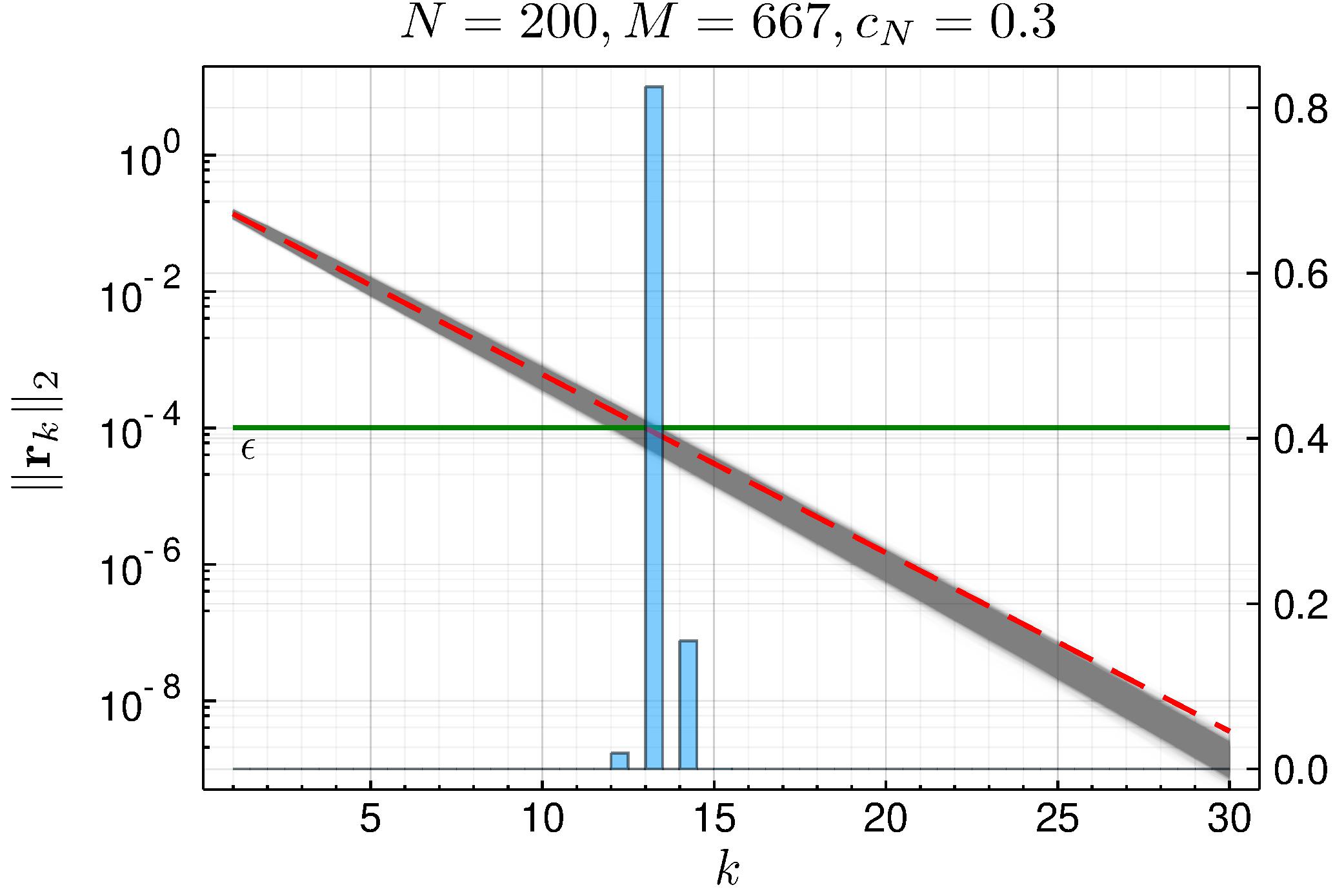}
    \includegraphics[width=.48\linewidth]{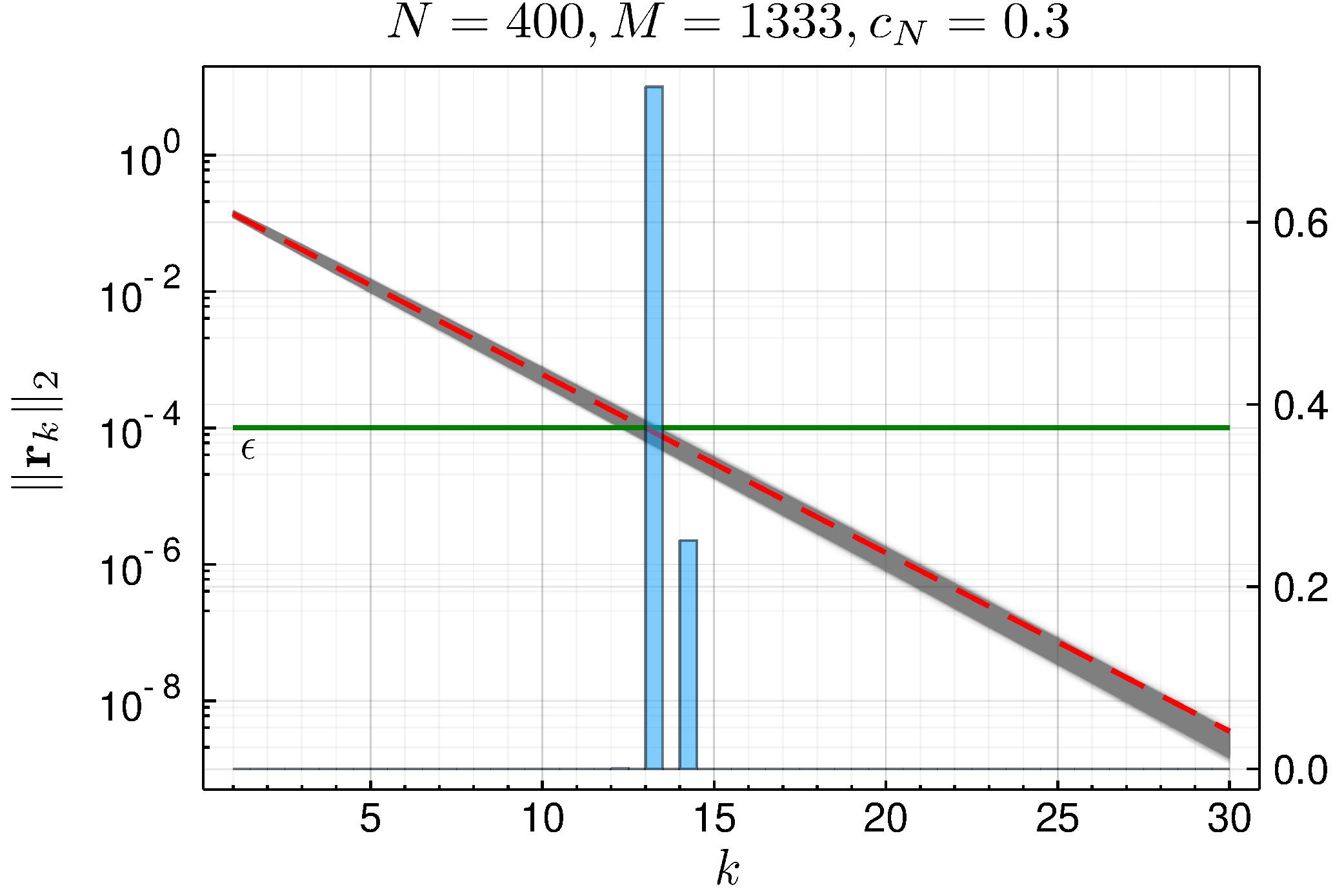}
    \caption{A numerical demonstration of the concentration of the residual in the CGA on Johnstone's spiked covariance model \eqref{eq_JSM} when $X$ is an iid Gaussian matrix. Here we take $\ell = 15, \bm v = \bm f_1$ and $\bm b = \bm v$.  In this case the bidiagonalization in  \cite{MR3078286} gives the matrix $\mathcal L$ in the large $M$ limit and the resulting predicted errors are given by the dashed curve. See Figure~\ref{fig:wishart_and_many_spikes} for a description of what these plots demonstrate. }
    \label{fig:johnstone}
\end{figure}

\begin{figure}[tbp]
    \centering
    \includegraphics[width=.48\linewidth]{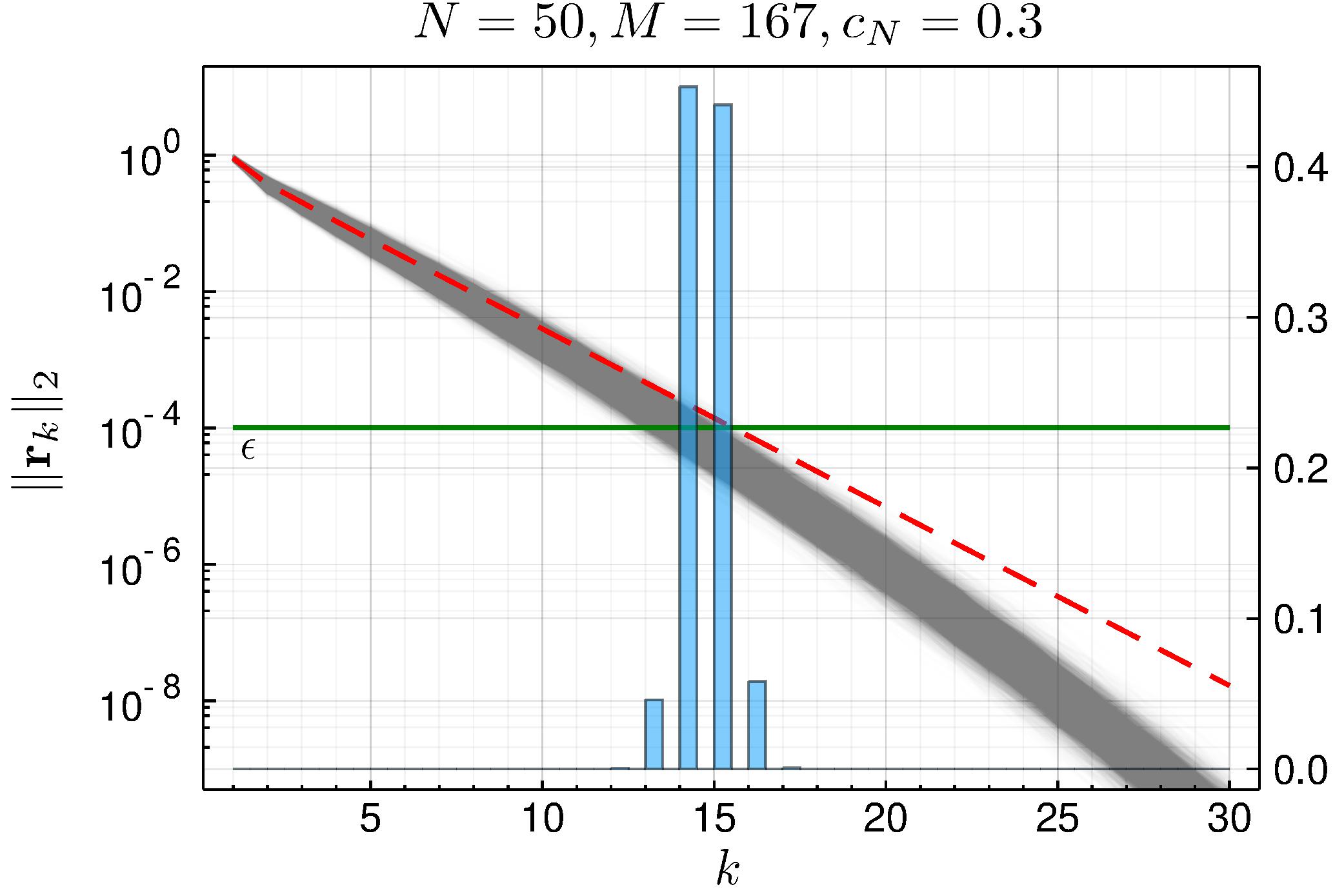}
    \includegraphics[width=.48\linewidth]{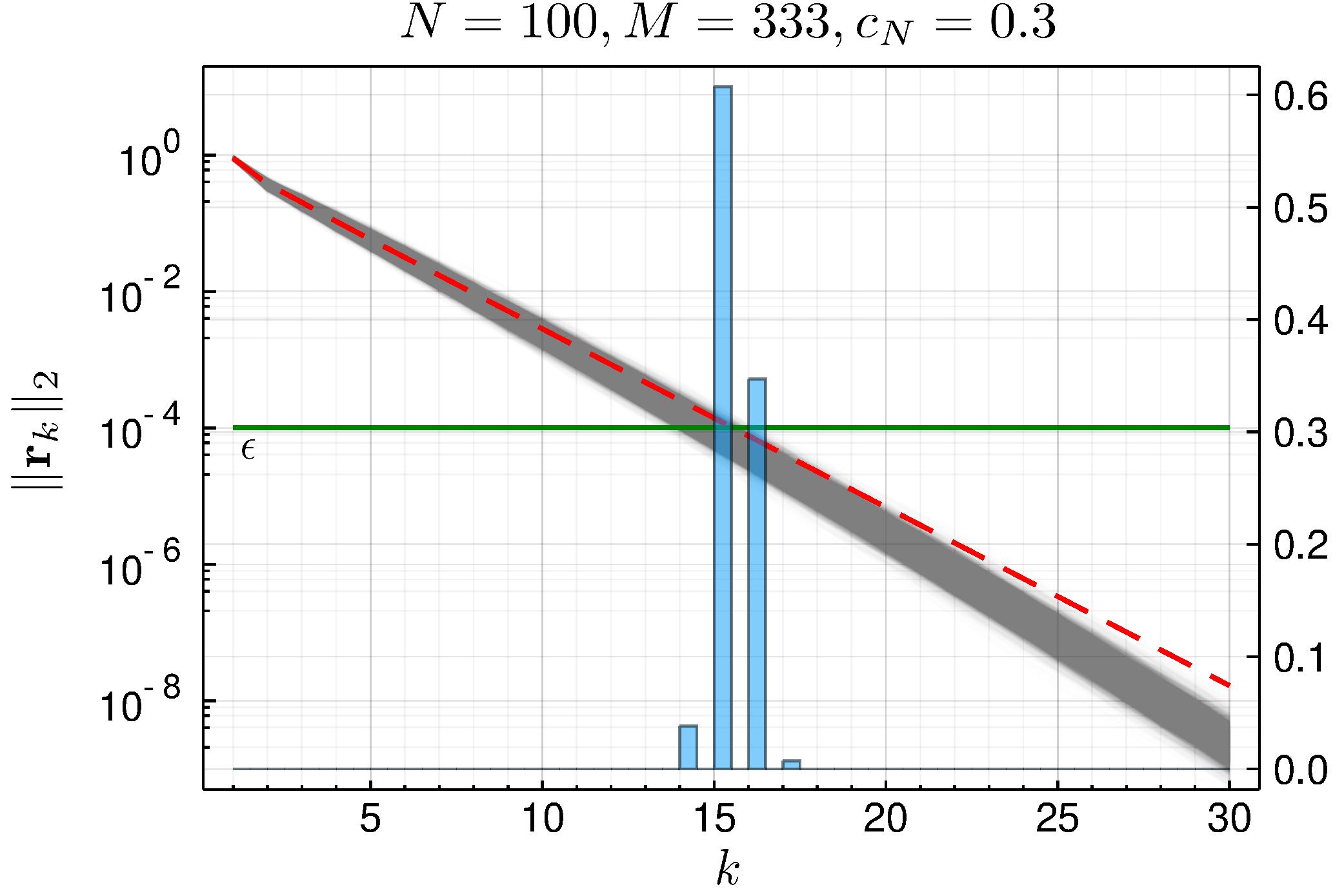}
    \includegraphics[width=.48\linewidth]{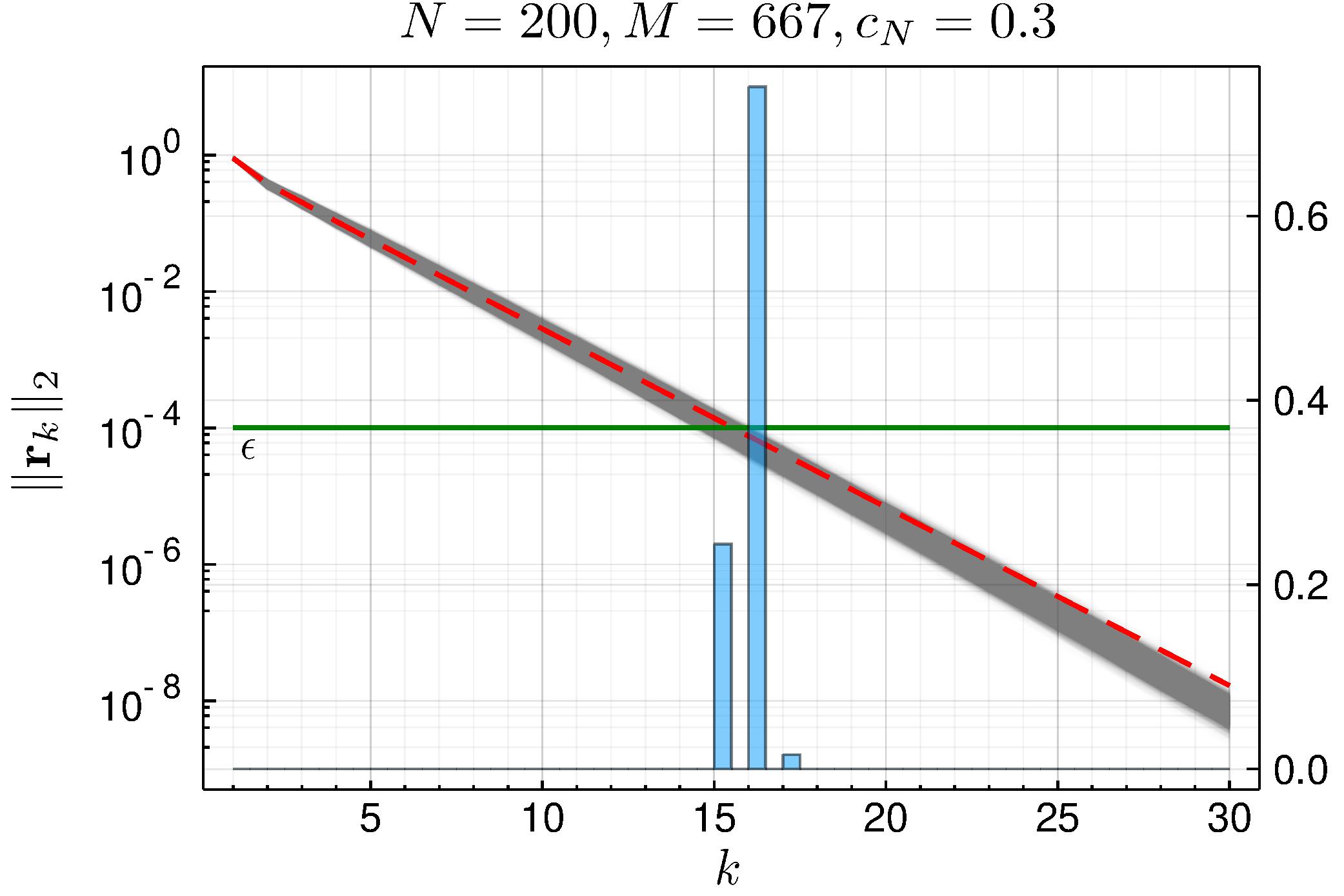}
    \includegraphics[width=.48\linewidth]{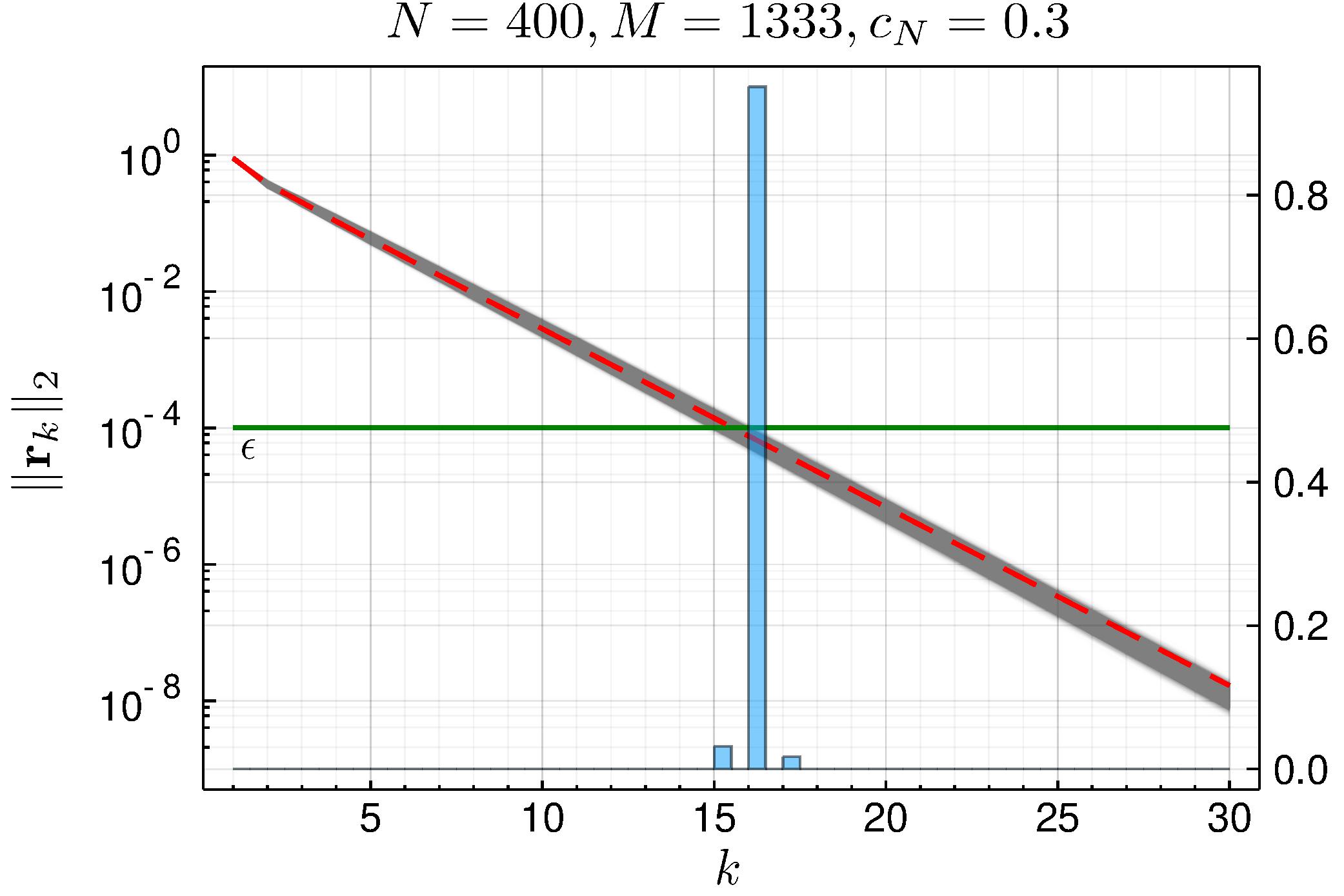}
    \caption{A numerical demonstration of the concentration of the residual in the CGA on Johnstone's spiked covariance model \eqref{eq_JSM} when $X$ is an iid Gaussian matrix. Here we take $\ell = 15, \bm v = \bm f_1$ and $\bm b = \frac{1}{\sqrt{2}} \bm f_1 +  \frac{1}{\sqrt{2}} \bm w$ where $\bm w = [0,\bm w']^T$, and $\bm w'$ is distributed uniformly on the hypersphere in $\mathbb R^{N-1}$.  Since we do not have a closed-form expression for the limiting dashed curve, we estimate it using the procedure outlined at the beginning of this section. The modification of $\bm b$, in comparision to Figure~\ref{fig:johnstone}, modifies the behavior of the first couple iterations --- but the same asymptotic rate of convergence persists. See Figure~\ref{fig:wishart_and_many_spikes} for a description of what these plots demonstrate.}
    \label{fig:johnstone2}
\end{figure}

%

\begin{figure}[tbp]
    \centering
    \includegraphics[width=.48\linewidth]{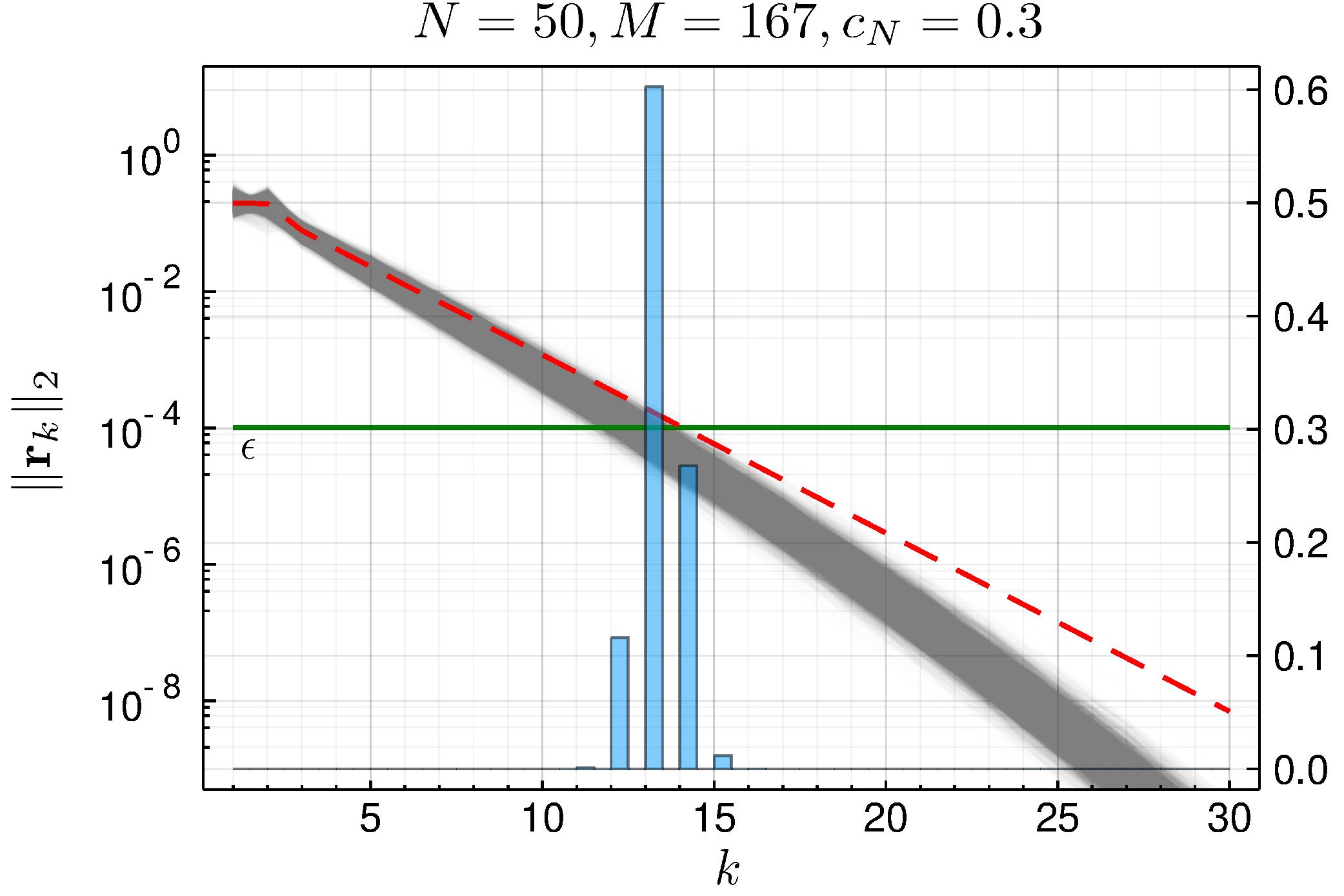}
    \includegraphics[width=.48\linewidth]{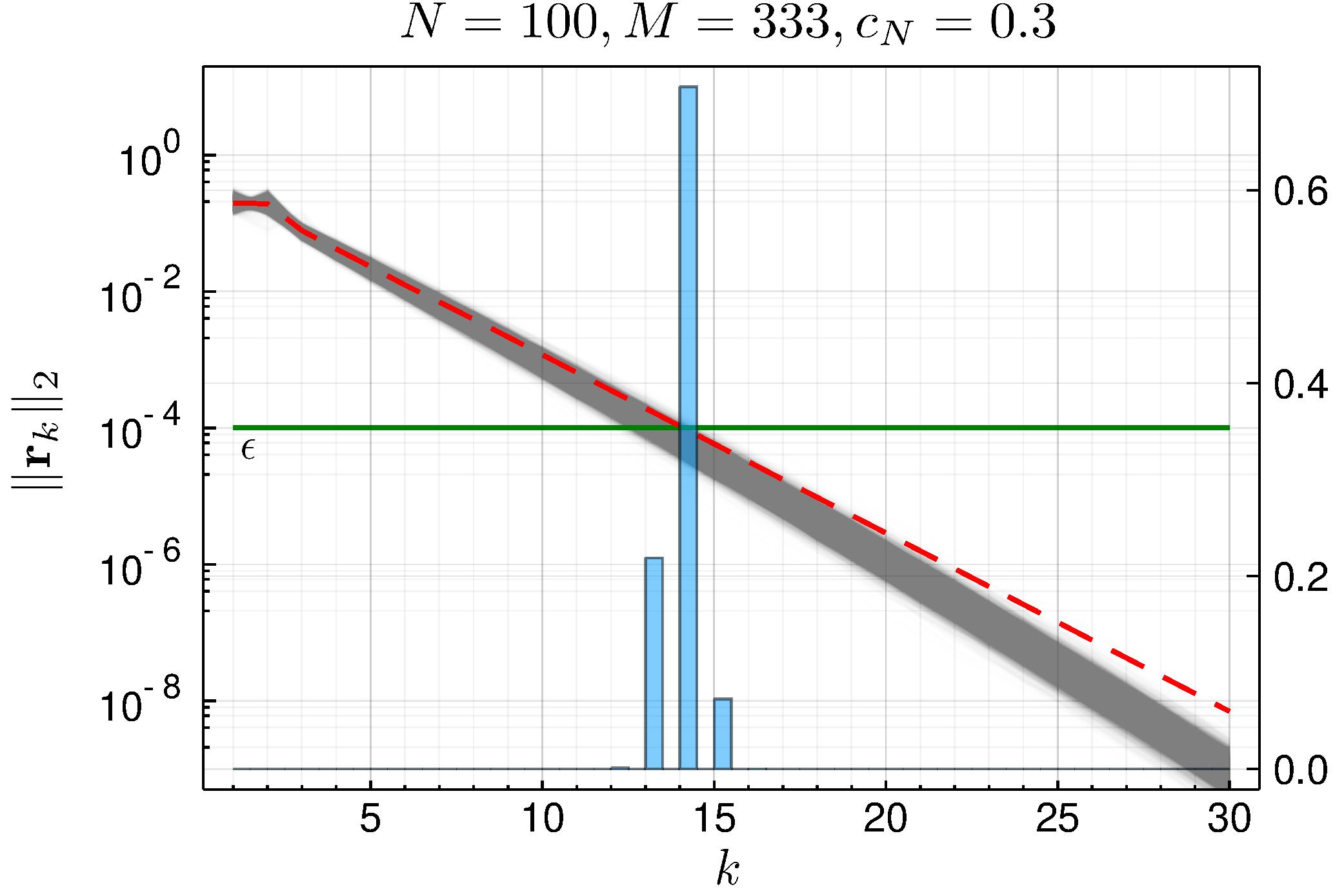}
    \includegraphics[width=.48\linewidth]{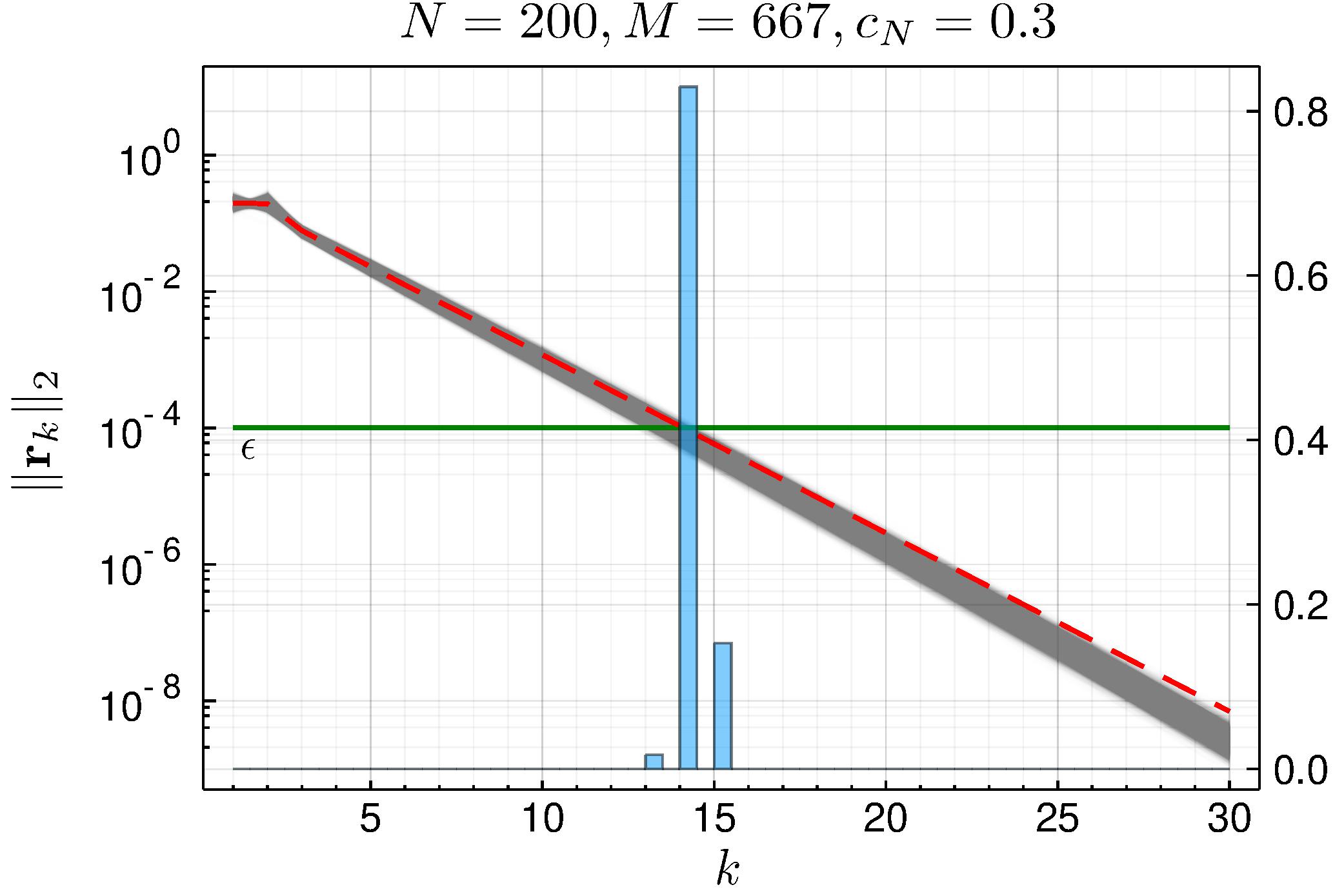}
    \includegraphics[width=.48\linewidth]{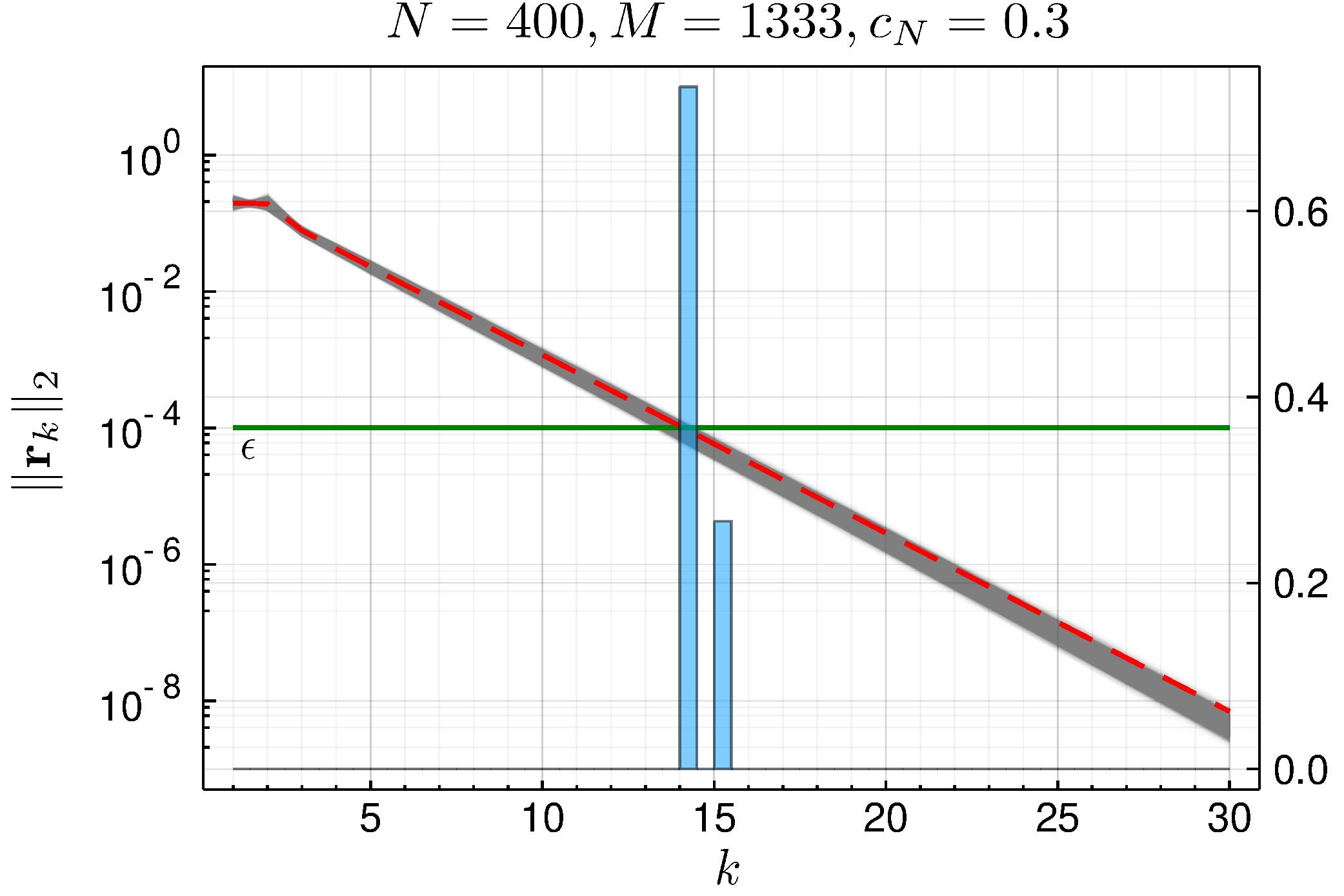}
    \caption{A numerical demonstration of the concentration of the residual in the CGA when $\Sigma^{1/2} = \mathrm{diag}(4,3.5,1,1,\ldots,1)$ and $X$ is an iid Gaussian matrix. Here we take $\bm b = \frac{1}{\sqrt{4}} \bm f_1 +  \frac{1}{\sqrt{4}} \bm f_2 + \frac{1}{\sqrt{2}} \bm w$ where $\bm w = [0,0,\bm w']^T$, and $\bm w'$ is distributed uniformly on the hypersphere in $\mathbb R^{N-2}$.  Since we do not have closed-form expression for the limiting dashed curve, we estimate it using the procedure outlined at the beginning of this section. See Figure~\ref{fig:wishart_and_many_spikes} for a description of what these plots demonstrate.}
    \label{fig:two-spike}
\end{figure}

\subsection{Spiked invariant model \cite{spikedinvariant, DJ}}

We consider the spiked invariant model where the ESD of $\Sigma_0$ converges to the standard MP law with parameter $c_N$ (c.f. (\ref{eq_mpdensity})). As discussed in Remark \ref{rmk_assumption}, (3) of Assumption \ref{assum_summary} is satisfied. It is well known that the asymptotic density $\varrho$ can be characterized as the free multiplicative convolution of two MP laws. In fact, the density function can be calculated explicitly as in Lemma \ref{lem_expcilitformula}. In this case, $f(x)$ can be replaced by
\begin{equation*}
 f(x)=-\frac{1}{x}+c_N \int \frac{1}{x+\lambda^{-1}} \mu_{\mathtt{MP}}(\dd \lambda),   
\end{equation*}
where $\mu_{\mathtt{MP}}$ is the standard MP law with parameter $c_N.$ Moreover, in this setting, $\gamma_{\pm}$ have closed form expressions, see (\ref{eq_invariantformula}). For the spiked model, we can calculate the essential quantities based on the above expressions.  See Figure~\ref{fig:invariant} for a demonstration. 

\begin{figure}[tbp]
    \centering
    \includegraphics[width=.48\linewidth]{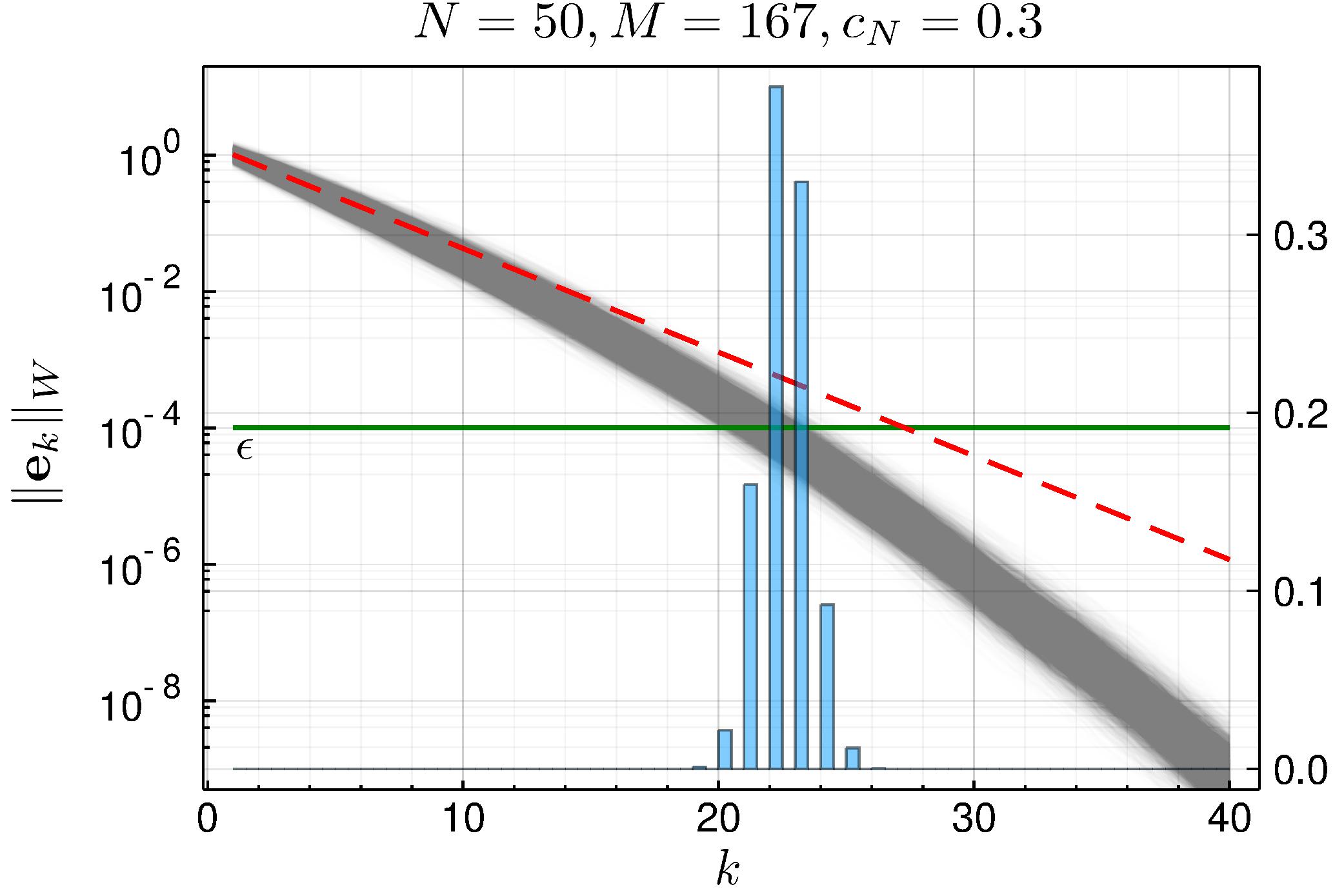}
    \includegraphics[width=.48\linewidth]{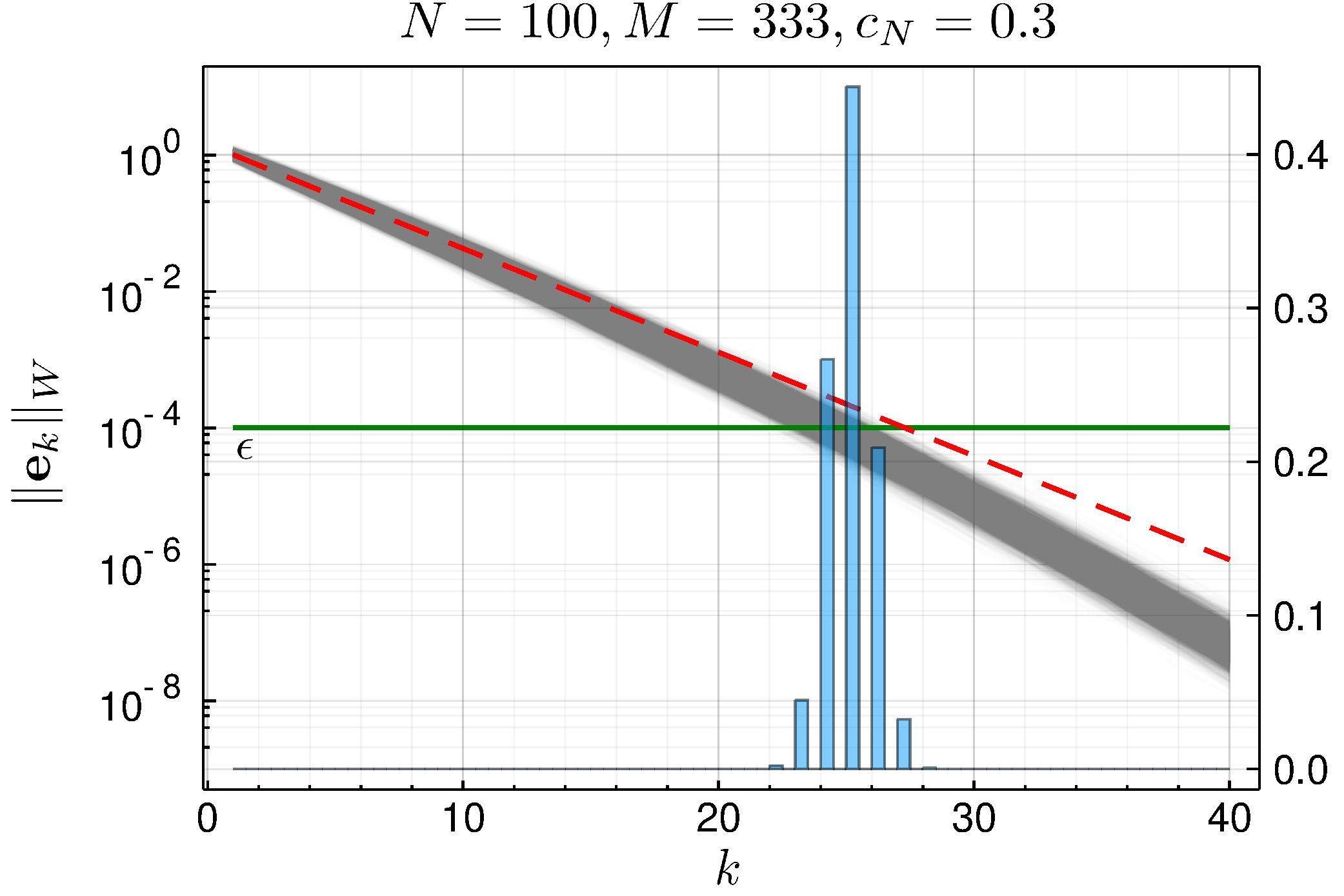}
    \includegraphics[width=.48\linewidth]{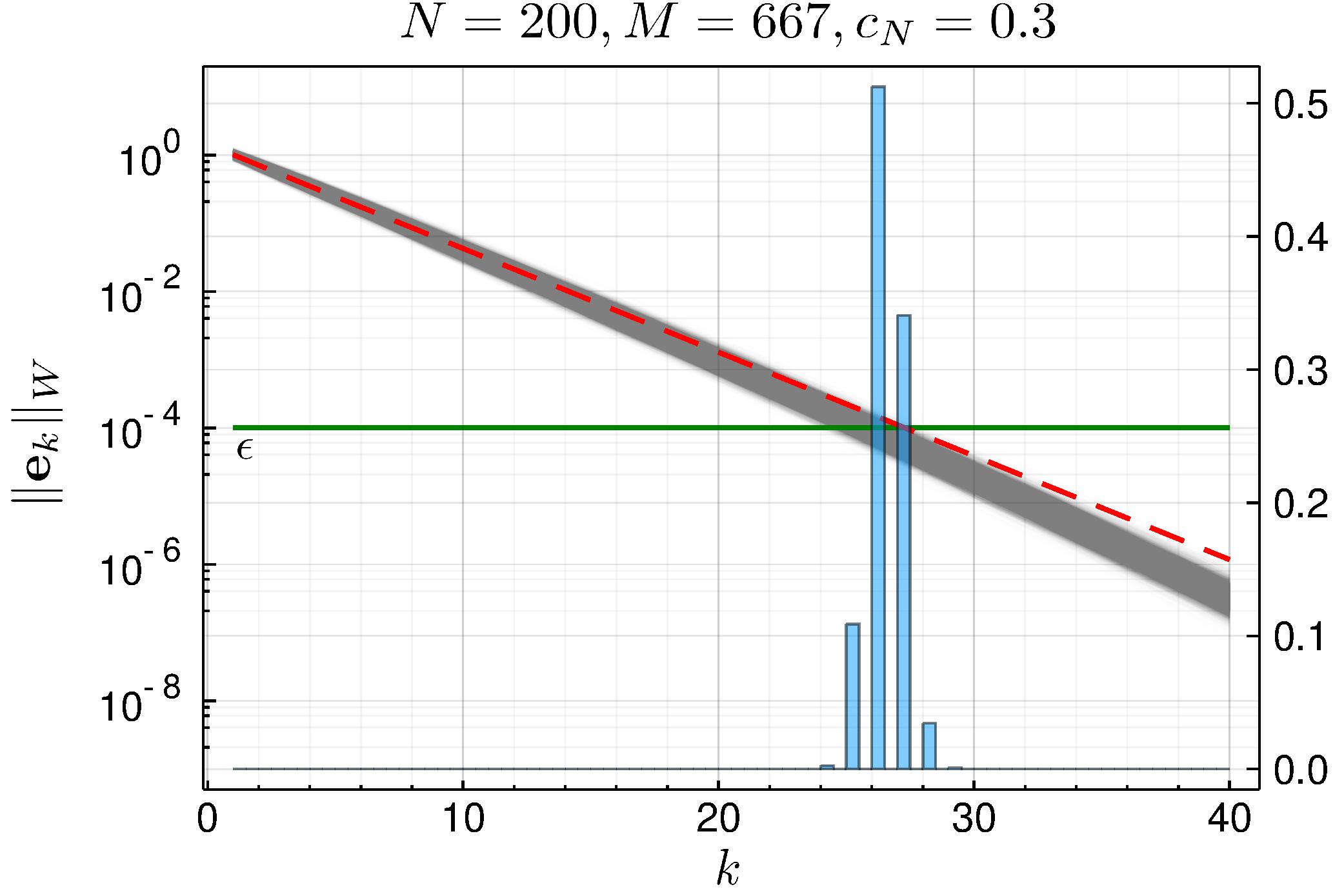} 
    \includegraphics[width=.48\linewidth]{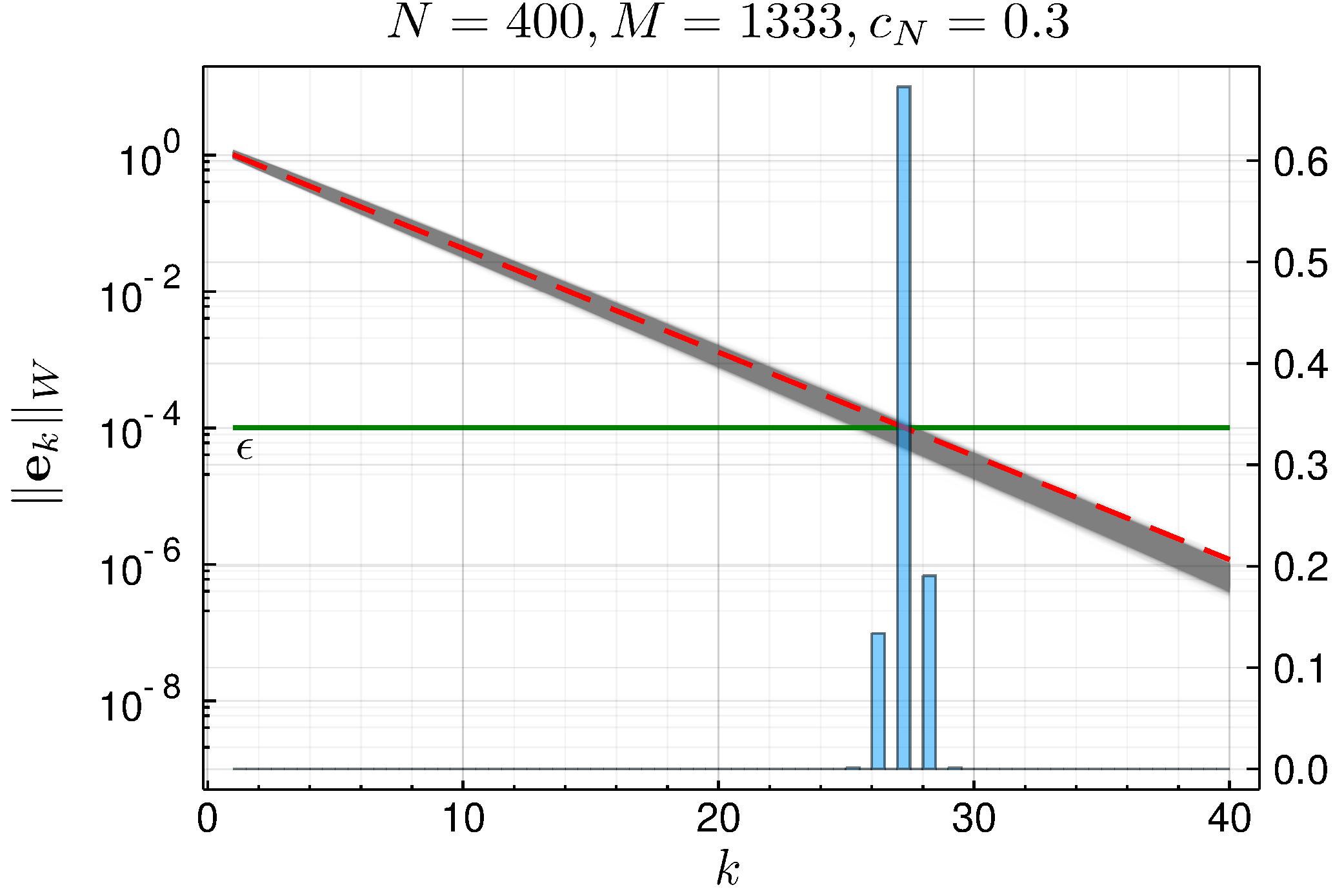}
    \caption{A demonstration of the concentration of $\|\bm e_k\|_W$ in the case of the spiked invariant model.  Since we do not have closed-form expression for the limiting dashed curve, we estimate it using the procedure outlined at the beginning of this section. See Figure~\ref{fig:wishart_and_many_spikes} for a description of what these plots demonstrate.  Note that these plots show the $W$-norm of $\bm e_k$, not the 2-norm of $\bm r_k$ as in Figure~\ref{fig:wishart_and_many_spikes}.}
    \label{fig:invariant}
\end{figure}

\subsection{Spiked covariance matrix with uniformly distributed eigenvalues \cite{DRMTA}}

We assume that the ESD of $\Sigma_0$ converges to the uniform distribution on $[a,b],$ where $a, b$ are some positive constants. As discussed in Remark \ref{rmk_assumption}, (3) of Assumption \ref{assum_summary} is satisfied.  In this case, $f(x)$ can be replaced by 
\begin{equation*}
f(x)=-\frac{1}{x}+\frac{c_N}{b-a}\left(\frac{b-a}{x}-\frac{1}{x^2} \ln \frac{bx+1}{ax+1} \right). 
\end{equation*}
Then the desired quantities can be calculated based on the above expressions. For a concrete example, we consider that $a=1, b=3$ and $c_N=0.5.$  The critical points $x_{\pm}$ can be calculated numerically using Newton's method and are approximately $-2, -0.25.$ Then the support of $\varrho$ only contains a single interval and the edges are approximately $0.15$ and $6.4,$ respectively. The essential quantities of the  spiked model can be calculated analogously; see Figure \ref{fig:uniform} for an illustration. 

\begin{figure}[tbp]
    \centering
    \includegraphics[width=.48\linewidth]{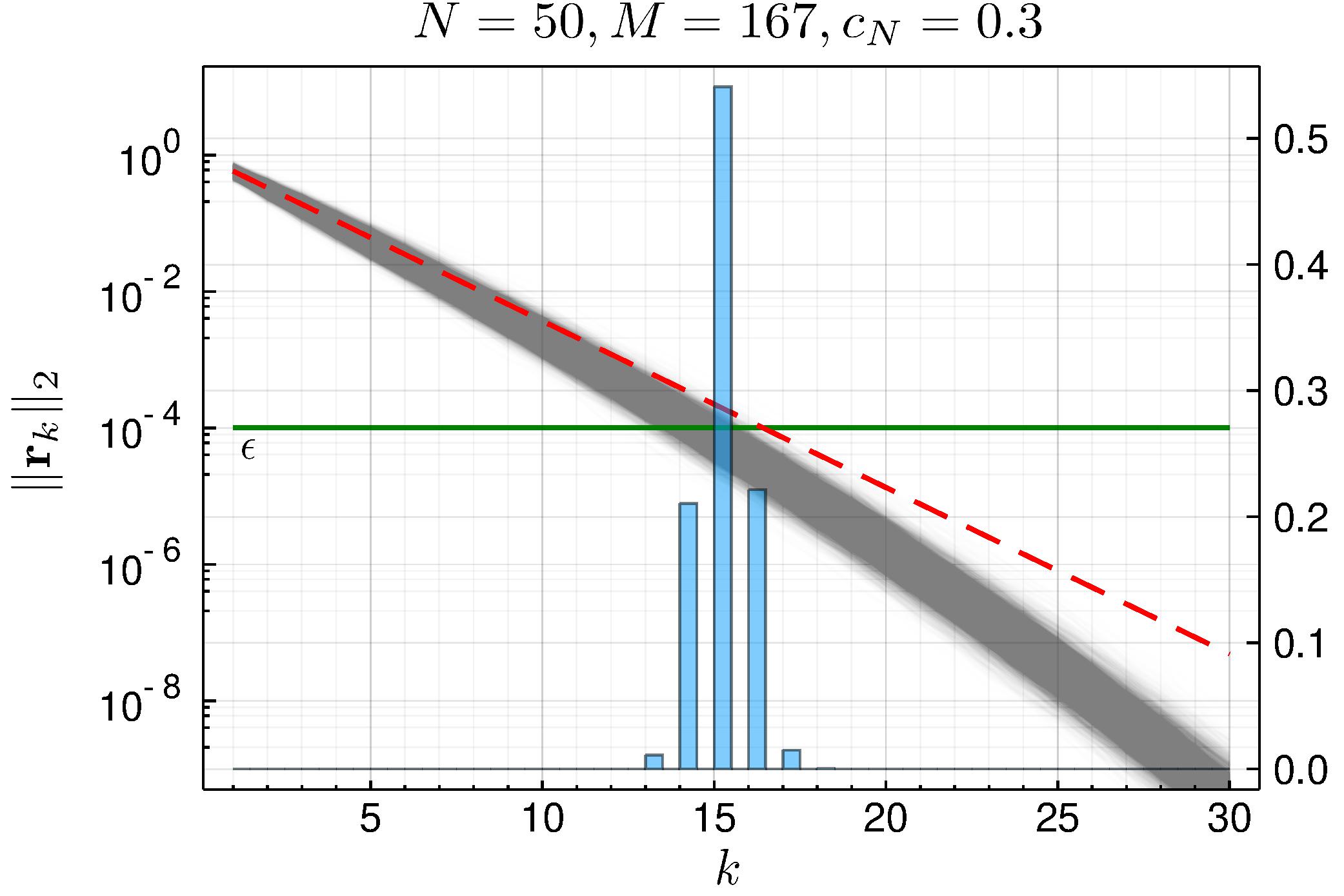}
    \includegraphics[width=.48\linewidth]{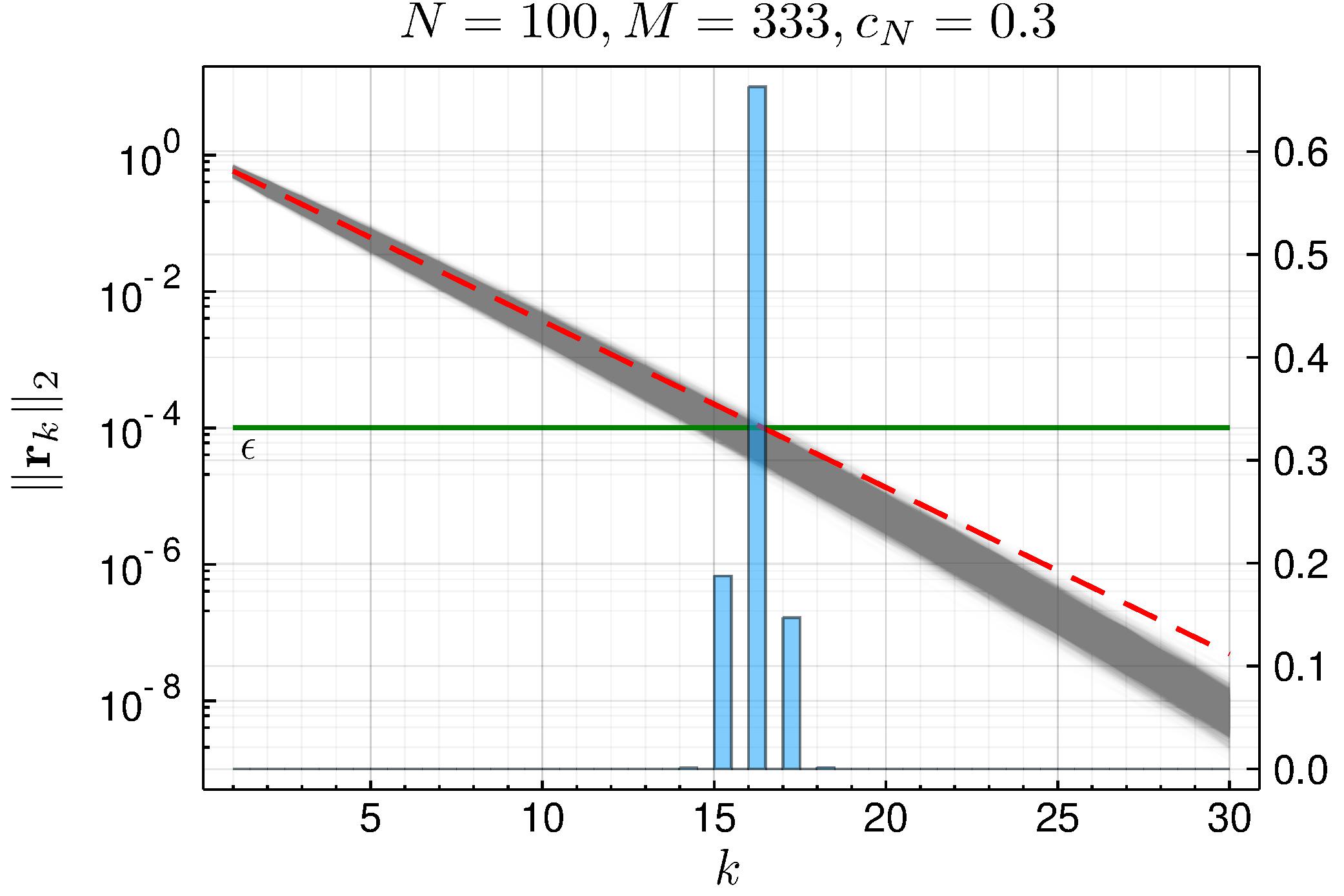}
    \includegraphics[width=.48\linewidth]{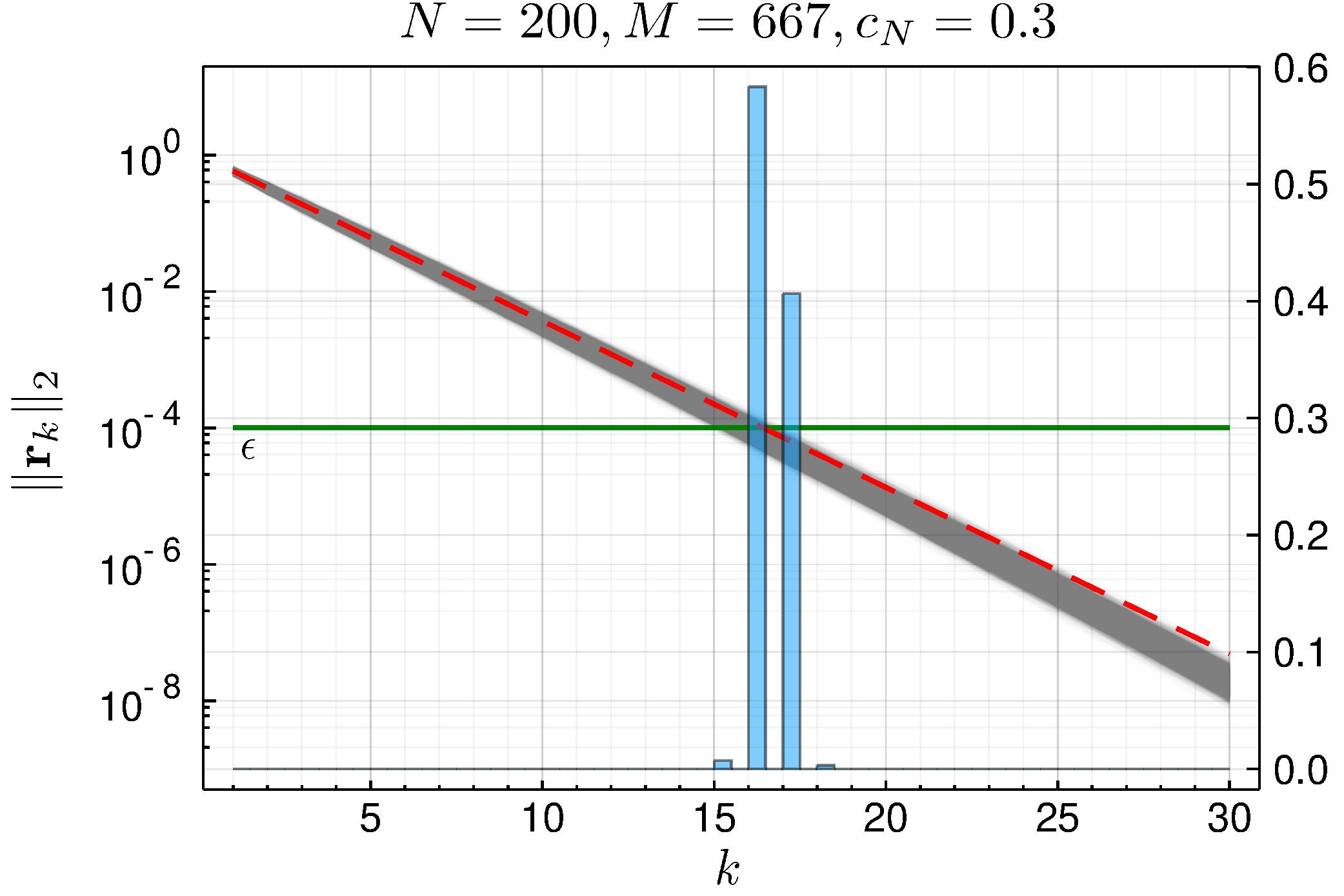} 
    \includegraphics[width=.48\linewidth]{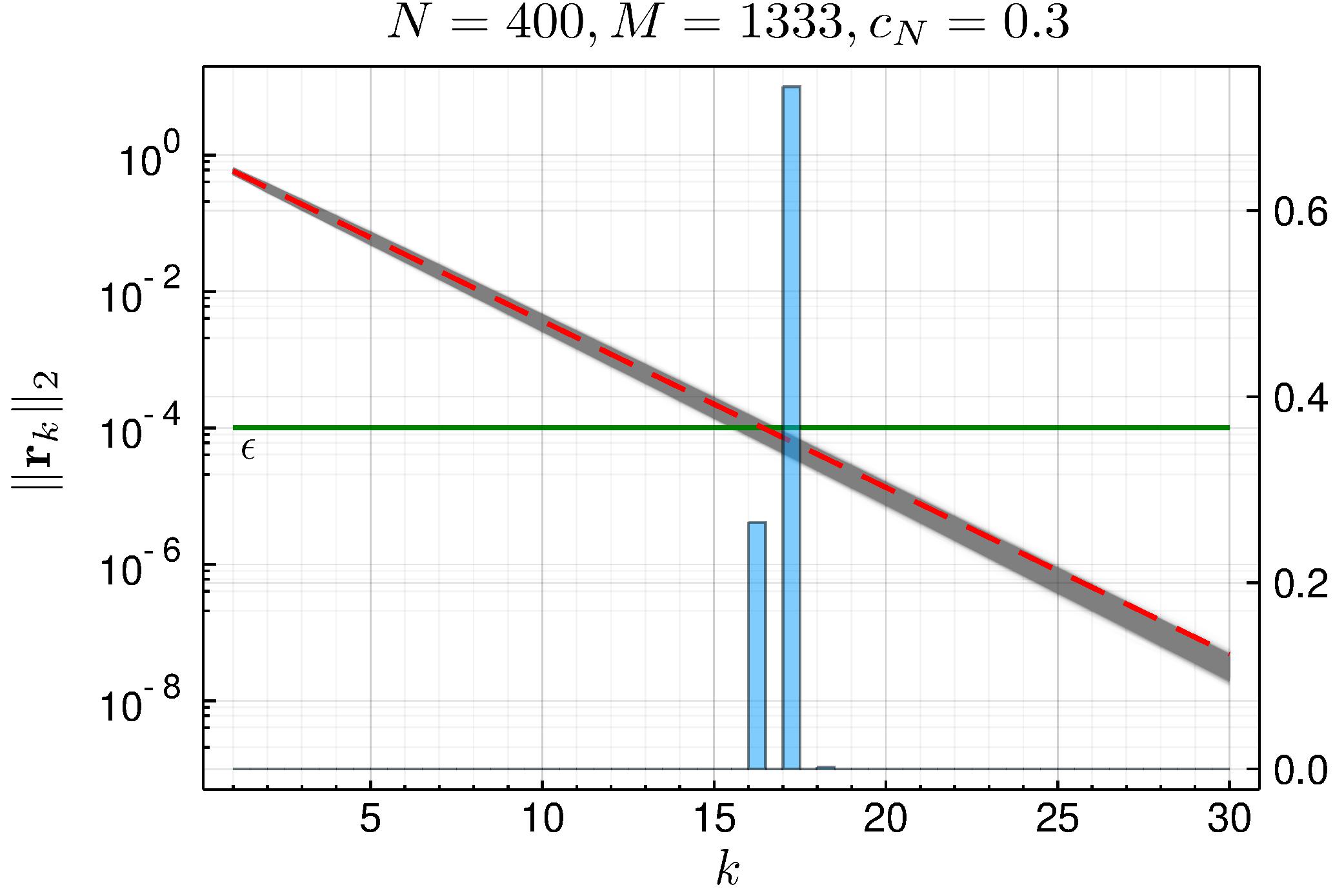}
    \caption{A demonstration of the concentration of $\|\bm r_k\|_2$ in  the uniformly deformed case.  Since we do not have closed-form expression for the limiting dashed curve, we estimate it using the procedure outlined at the beginning of this section. See Figure~\ref{fig:wishart_and_many_spikes} for a description of what these plots demonstrate.}
    \label{fig:uniform}
\end{figure}

\subsection{Spiked Toeplitz matrix \cite{MR2308592}}

Toeplitz matrices are a common object of study in time series analysis  since the covariance structure of a stationary
time-series is a Toeplitz matrix. Suppose that $\Sigma_0$ is a symmetric positive definite Toeplitz matrix satisfying the assumptions in \cite[Section A.3.4]{MR2308592}, then (3) of Assumption \ref{assum_summary} is satisfied. Since the eigenvalues of $\Sigma_0$ do not have closed-forms, in general, we need to numerically calculate calculate the eigenvalues of Topelitz matrix and the function $f$ in (\ref{eq_defnstitlesjtransform}). The other quantities can be calculated based on that. For a concrete example, let $\Sigma_0$ be the covariance matrix of an order one stationary autoregressive (AR) model such that the entries of $\Sigma_0$ satisfy 
\begin{equation}\label{eq_sigma_toep}
(\Sigma_0)_{i,j}=0.4^{|i-j|}.     
\end{equation}
For a concrete case when $c_N=1/2,$ according to \cite[Example 3.10]{DRMTA}, we use Newton's method to get the critical points of $f(x),$ which are $-0.33, -3.62$. As a result, $\gamma_- \approx 0.086, \gamma_+ \approx 4.385.$  Similarly, we can obtain the other quantities for the spiked Toeplitz matrix; see Figure \ref{fig:toeplitz} for a demonstration. 

\begin{figure}[tbp]
    \centering
    \includegraphics[width=.48\linewidth]{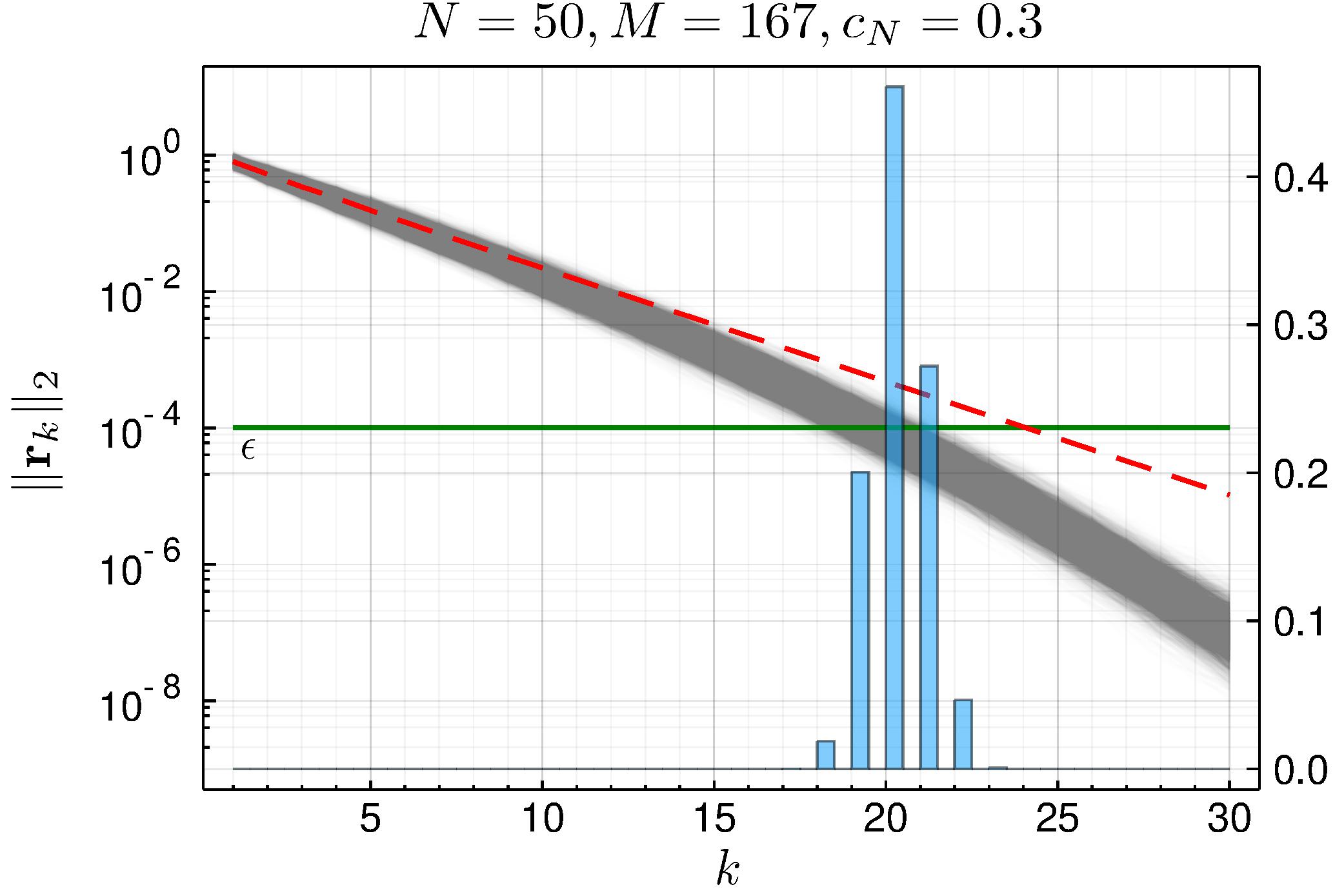} 
    \includegraphics[width=.48\linewidth]{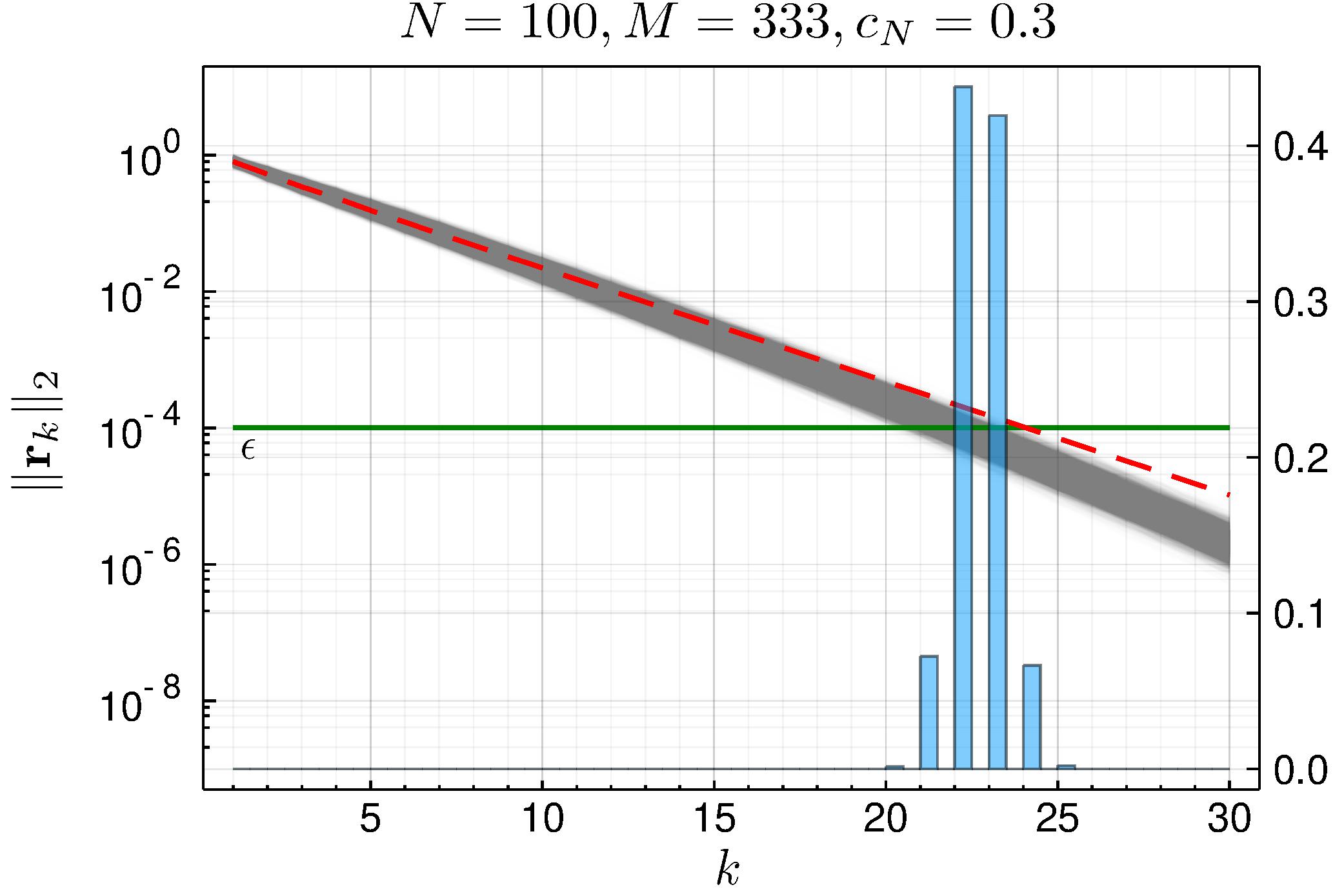}
    \includegraphics[width=.48\linewidth]{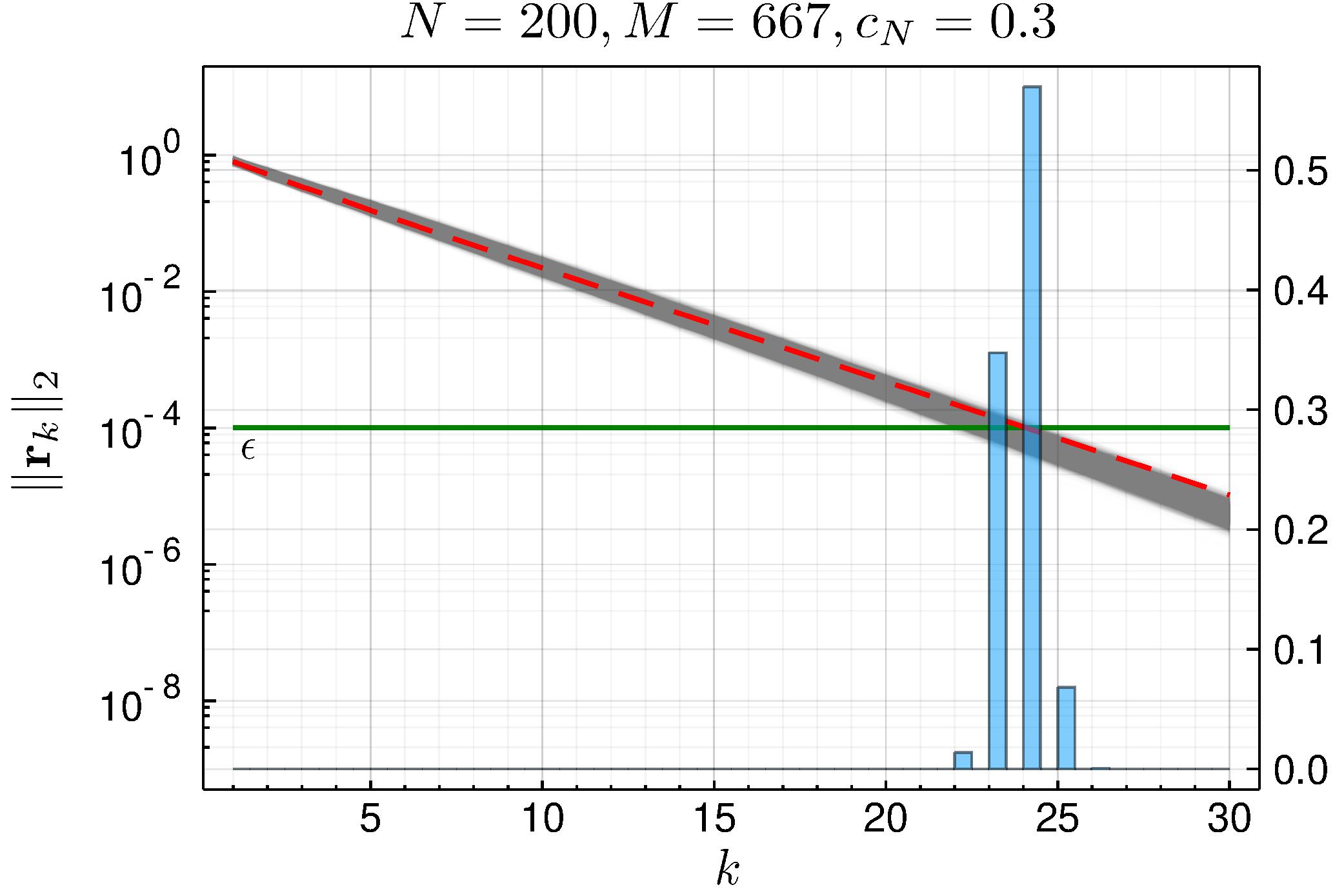} 
    \includegraphics[width=.48\linewidth]{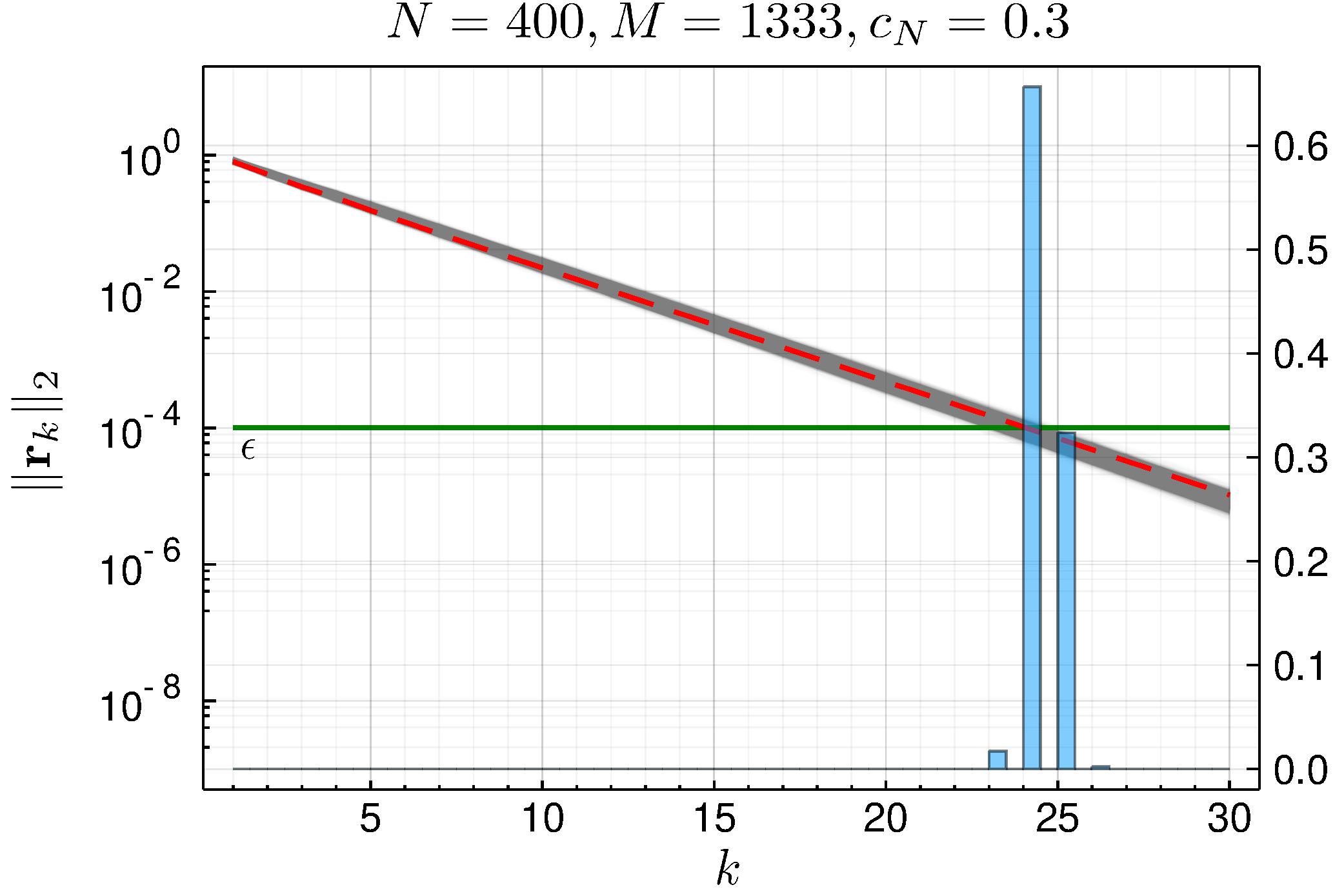}
    \caption{A demonstration of the concentration of $\|\bm r_k\|_2$ in the case of the unspiked Toeplitz case \eqref{eq_sigma_toep}.  Since we do not have closed-form expression for the limiting dashed curve, we estimate it using the procedure outlined at the beginning of this section. See Figure~\ref{fig:wishart_and_many_spikes} for a description of what these plots demonstrate.}
    \label{fig:toeplitz}
\end{figure}



%


%

\section{Asymptotics of orthogonal polynomials and Cholesky factorization}\label{sec_theoryoforthogonalpolynomial}
In this section, we provide results on the theory of orthogonal polynomials. 
\subsection{Hankel determinants, moments and the three-term recurrence relation}\label{sec_bascis}
In this subsection, we introduce the connection between Lanczos iteration and orthogonal polynomials \cite[Lecture 36]{MR1444820}. Let $T$ be the $N \times N$ Jacobi matrix generated from the Lanczos iteration for its maximum of $N$ steps. It produces a probability measure
\begin{align}\label{eq:muT}
    \mu_T = \sum_{j=1}^N \delta_{\lambda_j} \omega_j,
\end{align}
where $\lambda_j$'s are the eigenvalues of $T$ and $\omega_j$ is the squared modulus of the first component of the normalized eigenvector associated to $\lambda_j$.  For the $N \times N$ Hermitian matrix $W,$ denote its
eigenvectors as $\{\bm{u}_i\},$ and for any unit vector $\bm{b},$ denote the eigenvector empirical spectral distribution (VESD) as \cite{BMP}
\begin{equation}\label{eq_defnvesd}
\mu_{W,\bm{b}}=\sum_{i=1}^N |\langle \bm{u}_i, \bm{b} \rangle|^2 \delta_{\lambda_i(W)}.    
\end{equation}

The VESD $\mu_{W, \bm{b}}$ coincides with the spectral measure $\mu_T$. In fact, there is a bijection between such measures and Jacobi matrices \cite{DeiftOrthogonalPolynomials}. Moreover, Proposition \ref{prop_three} below indicates that universality and estimates for the spectral measure in an appropriate sense will translate to universality and estimates for the Lanczos matrix. 

Based on $\mu_{W, \bm{b}}$, we can construct a sequence of orthogonal polynomials $\{p_n(x)\}$ from the monomials via Gram-Schmidt. The polynomials obey the following three-term recurrence relation \cite{MR0372517}
\begin{equation}\label{eq_threeterm}
xp_n(x)=b_n p_{n+1}(x)+a_n p_n(x)+b_{n-1} p_{n-1}(x), \ n \geq 0, b_n>0,    
\end{equation}
with the convention $p_{-1}(x)=0$ and $b_{-1}=0.$ Here $a_n, b_n$ are called the recurrence coefficients. 
\begin{proposition}\label{prop_three}
The three-term recurrence coefficients for the orthogonal polynomials generated by the VESD of $\mu_{W, \bm{b}}$ coincide with the entries in the Lanczos matrix $T(W,\vec b)$.  
\end{proposition}
\begin{proof}
See \cite{DeiftOrthogonalPolynomials}.
\end{proof}

Recall the classical fact that the coefficients in a three-term recurrence relation can be recovered as a nearly rational function of the moments of the associated spectral measure.  We write $p_n(x) = \ell_n x^n + s_n x^{n-1} + \cdots$ and find by equating coefficients that
\begin{align*}
    \ell_n &= b_n \ell_{n+1},\\
    s_n &= a_n \ell_n + b_n s_{n+1},
\end{align*}
where $a_n$ and $b_n$ will be given in (\ref{eq_defngeneralb}) after necessary notations are introduced. 

Denote the Hankel moment matrix of $\mu_{W, \bm{b}}$ as $M_n$ and $D_n=\det M_n.$ Moreover, define $D_n(\lambda)$ by the determinants
\begin{align}
    D_n(\lambda) = \det M_n(\lambda), \label{eq:momentdef}
\end{align}
and $M_n(\lambda)$ is formed by replacing the last row of $M_n$ with the row vector $[1 ~\lambda ~ \lambda^2 \cdots \lambda^n]$.  Then, it is well-known that (see, e.g., \cite{DeiftOrthogonalPolynomials})
\begin{align*}
    p_n(\lambda) = \frac{D_n(\lambda)}{\sqrt{D_n D_{n-1}}}. 
\end{align*}
This gives
\begin{align}\label{eq_defngeneralb}
    b_n = \sqrt{\frac{D_{n-1} D_{n+1}}{D_n^2}}, \quad a_n = \frac{s_n - b_n s_{n+1}}{\ell_n} = \frac{s_n}{\ell_n} - \frac{s_{n+1}}{\ell_{n+1}}.
\end{align}
The above expression shows that $a_n$ and $b_n$ are infinitely differentiable functions of $m_0,m_1,m_2,\ldots,m_{2n+2}$ on the open set
\begin{align*}
    \{ D_j > 0, \quad j = 1,2,\ldots,n+1\}.
\end{align*}

\begin{remark}\label{rmk_moments}
Associated with the three-term recurrence (\ref{eq_threeterm}) is the following infinite-dimensional Jacobi matrix
\begin{equation*}
\mathcal T = \begin{bmatrix} a_0 & b_0 \\
    b_0 & a_1 & b_1 \\
    & b_1 & a_2 & \ddots \\
    & & \ddots & \ddots \end{bmatrix}.
\end{equation*}
Let $T_n$ be the upper left $n \times n$ subblock of $\mathcal T$. Then we readily see that $T_n$ is a differentiable function of $(m_0, m_1, \cdots, m_{2n}).$ We also note that \cite{DeiftOrthogonalPolynomials} 
\begin{equation*}
\bm{f}_1^* \mathcal T^k \bm{f}_1=\int \lambda^k \mu_{W, \bm{b}}(\sd \lambda).   
\end{equation*}
\end{remark}


\subsection{Asymptotics of three-term recurrence relations and the Cholesky
factorization
} 
In this subsection, we explore the asymptotic form of the Jacobi matrix and Cholesky decomposition when the VESD exhibits regular square root behavior near the edges. 

\begin{theorem}\label{thm_relationOP}
  Suppose $h: [a,b] \to \mathbb R$ is a positive real analytic function. Consider the measure $\mu$ defined by
  \begin{align*}
    \mu(\sd \lambda) = h(\lambda) \one_{[a,b]}(\lambda) (b - \lambda)^{\alpha} (\lambda - a)^\beta \sd \lambda + \sum_{j=1}^p w_j \delta_{c_j}(\sd \lambda)
  \end{align*}
  where $w_j > 0$ and $c_j \in \mathbb R \setminus [a,b]$, for all $1 \leq j \leq p$.  Suppose, in addition, that $\alpha = \pm \frac 1 2, \beta = \pm \frac 1 2$.  Then there exists $c > 0$ such that
  \begin{align*}
    a_n =  \frac{b + a}{2} + O(\e^{- c n}), \quad b_n = \frac{b - a}{4} + O(\e^{- c n}).
  \end{align*}
  Moreover, if there exists $0 < \tau < 1$ such that
  \begin{itemize}
      \item $\tau \leq w_j \leq \tau^{-1}$, for all $j = 1,2,\ldots,p$,
      \item $\tau \leq |h(z)| \leq \tau^{-1}$ and $h$ is analytic for all $z \in \mathbb C$ such that $\min_{\lambda \in [a,b]} |z - \lambda| < \tau$, and
      \item $\min\{|a - c_j|, |b - c_j|\} \geq \tau$ for all $j = 1,2,\ldots,p$,
  \end{itemize}
  then $c$ can be taken to be a function of $\tau$ alone.
\end{theorem}
\begin{proof}
It has been proved in \cite{MR2087231,Kuijlaars2003} for the case $a=-1, b=1$ without discrete contributions. The result follows from that with a simple modification if $w_j = 0$ for all $j$.  With spikes, as noted in \cite{Kuijlaars2003}, the result follows from \cite{Geronimo1994} for $h$, $w_j$, $c_j$ fixed.  To obtain uniformity, one introduces poles into the Riemann--Hilbert formulation in \cite{Kuijlaars2003} (originally due to \cite{Its1994}), turning residue conditions in to rational jump conditions and then inverting exponential growing jumps so that they tend to the identity matrix at a uniform exponential rate, see \cite[Section 8.2.2]{TrogdonSOBook}, for example.  
\end{proof}
\begin{remark}\label{mplaw_example}
  For the Marchenko--Pastur law, we have $a = (1- \sqrt{c_N})^2$ and $b = (1 + \sqrt{c_N})^2$ so that
  \begin{align*}
\frac{b + a}{2} = 1 + c_N, \quad \frac{b - a}{4} = \sqrt{c_N}.
  \end{align*}
\end{remark}

  The proof of the following Lemma \ref{l:En} is a direct consequence of $\varphi(T)\varphi(T)^* = T$ where $\varphi(T)$ is defined in Algorithm~\ref{a:chol} and the fact that the diagonal entries in the Cholesky factorization must be positive. Note that $\alpha$ in Lemma \ref{l:En} is always real since 
    \begin{align*}
    \frac{(b + a)^2}{4} - 4 \frac{(b-a)^2}{16} = \frac{1}{4} \left( (b+ a)^2 - (b-a)^2 \right) = ab > 0.
    \end{align*}

\begin{lemma}\label{l:En}
  Let $\gamma \geq 2 \beta \geq 0$ and set $\alpha = \frac{\gamma + \sqrt{\gamma^2 - 4\beta^2}}{2}$.  Suppose
  \begin{align}\label{eq:l:En}
    T = \begin{bmatrix} \alpha (1 + E_0) & \beta (1 + f_\beta(0)) \\
      \beta (1 + f_\beta(0)) & \gamma (1 + f_{\gamma}(0)) & \beta (1 + f_\beta(1)) \\
      & \beta (1 + f_\beta(1)) & \gamma (1 + f_{\gamma}(1))& \beta (1 + f_\beta(2)) \\
      && \beta (1 + f_\beta(2)) & \gamma (1 + f_{\gamma}(2)) & \ddots \\
      &&& \ddots & \ddots \\
      &&&&& \beta (1 + f_\beta(N-2))\\
      &&&&  \beta (1 + f_\beta(N-2)) & \gamma (1 + f_{\gamma}(N-2))
      \end{bmatrix} 
  \end{align}
  for functions $f_\beta, f_\gamma : \mathbb N \cup \{0\} \to \mathbb (-1,\infty)$ and $E_0 > -1$.  Then if $T$ is invertible,
  \begin{align*}
    \varphi(T) = \begin{bmatrix}
      \sqrt{\alpha} \sqrt{1 + E_0} \\
      \frac{\beta}{\sqrt{\alpha}}  \frac{1 + f_{\beta}(0)}{\sqrt{1 + E_0}} & \sqrt{\alpha} \sqrt{1 + E_1} \\
      & \frac{\beta}{\sqrt{\alpha}} \frac{1 + f_{\beta}(1)}{\sqrt{1 + E_1}} & \sqrt{\alpha} \sqrt{1 + E_2}\\
      && \frac{\beta}{\sqrt{\alpha}} \frac{1 + f_{\beta}(2)}{\sqrt{1 + E_2}} & \sqrt{\alpha} \sqrt{1 + E_3}\\
      &&& \ddots & \ddots \\
      &&&& \frac{\beta}{\sqrt{\alpha}} \frac{1 + f_{\beta}(N-2)}{\sqrt{1 + E_{N-2}}} & \sqrt{\alpha} \sqrt{1 + E_{N-1}}
      \end{bmatrix},
  \end{align*}
  where $E_n >-1$ satisfies
  \begin{align*}
    E_{n+1} = f_{\gamma}(n) + \frac{\beta^2}{\alpha^2} \left[ 1 + f_{\gamma}(n) - \frac{(1 + f_{\beta}(n))^2}{1 + E_n} \right].
  \end{align*}
\end{lemma}



\begin{theorem}\label{t:Asymptotic_Cholesky}
  Suppose $0 < a < b$ and set $\gamma = \frac{ a + b}{2}$, $\beta = \frac{b-a}{4}$ and $\alpha = \frac{\gamma + \sqrt{\gamma^2 - 4\beta^2}}{2} = \frac{(\sqrt{a} + \sqrt{b})^2}{4}$.  Suppose $T = T_N = \mathcal T_{1:N,1:N}$, the upper-left $N\times N$ block of a Jacobi operator $\mathcal T$ is of the form \eqref{eq:l:En} and satisfies the assumptions of Lemma~\ref{l:En} for every $N$.  Suppose, in addition, that there exists $\sigma > 0$ such that
  \begin{align*}
   \sigma^{-1} \leq \vec x^* \mathcal T \vec x \leq \sigma, \quad \|\vec x\|_2 = 1.
  \end{align*}
  If $\lim_{n\to \infty} f_{\gamma}(n) = 0 = \lim_{n\to \infty} f_{\beta}(n)$, then $\lim_{n\to \infty} E_n = 0$.
\end{theorem}
\begin{proof}
  Let $\varphi(T)$ be as in Lemma~\ref{l:En}.   Since $\sqrt{\alpha}\sqrt{1 + E_n}$ is an eigenvalue of $\varphi(T)$ we find that there exists a unit vector $\vec v$ such that
  \begin{align*}
    \| \varphi(T) \vec v\|_2^2 = \alpha (1 + E_n) = \vec v^T T \vec v \geq \sigma^{-1}.
  \end{align*}
  Thus
  \begin{align*}
    E_n \geq \frac{1}{\alpha \sigma} - 1, \quad \text{ for all } n \geq 0.
  \end{align*}
  Then, because $1/(1 + E_n) \geq 0$, we have
  \begin{align*}
    E_{n+1} \leq |f_{\gamma}(n)| + \frac{\beta^2}{\alpha^2} \left[ 1 + |f_{\gamma}| \right].
  \end{align*}
Thus $(E_n)_{n \geq n_0}$ forms a bounded sequence and any subsequence has a further subsequence that converges.  Supposing that $f_\gamma(n), f_{\beta}(n) \to 0$ as $n \to \infty$, we find that the limit $E_\infty$ along this subsequence satisfies
\begin{align*}
 E_{\infty} =  \frac{\beta^2}{\alpha^2} \left[ 1  - \frac{1}{1 + E_\infty} \right].
\end{align*}
Solving this relation gives $E_\infty = 0$ or $E_\infty = \frac{\beta^2}{\alpha^2} -1$.  So, it suffices to show that $E_\infty \neq  \frac{\beta^2}{\alpha^2} -1$ as this will then imply that every subsequence has a further subsequence that converges to a common limit.


  Suppose that $\delta = \frac{\beta^2}{\alpha^2} -1$ is a limit point of the sequence $E_n$.  Suppose that $| E_k - \delta| \leq \frac{\epsilon}{2} \Gamma^{-j}$ where $\Gamma = 4 \frac{\alpha^2}{\beta^2}$ and $\epsilon \leq \frac{\beta^2}{2 \alpha^2}$.
  Then it follows that $E_{k+1}$ satisfies
  \begin{align*}
    | E_{k+1} - \delta| \leq 2 |f_{\gamma}(k)| + 4 |f_{\beta}(k)| + 2 |f_{\beta}(k)|^2 + 4 \frac{\alpha^2}{\beta^2} |E_k - \delta|.
  \end{align*}
  And therefore
  \begin{align*}
    |E_{k+i} - \delta| \leq \left ( 4 \frac{\alpha^2}{\beta^2} \right)^i |E_k - \delta| + \max_{k \leq m \leq k+i} \left(2 |f_{\gamma}(m)| + 4 |f_{\beta}(m)| + 2 |f_{\beta}(m)|^2\right) \sum_{m = 1}^i \left ( 4 \frac{\alpha^2}{\beta^2} \right)^{i-m}.
  \end{align*}
  Then provided that
  \begin{align*}
    \max_{k \leq m \leq k+i} \left(2 |f_{\gamma}(m)| + 4 |f_{\beta}(m)| + 2 |f_{\beta}(m)|^2\right) \sum_{m = 1}^i \left ( 4 \frac{\alpha^2}{\beta^2} \right)^{-m} \leq \frac{\epsilon}{2} \Gamma^{-j},
  \end{align*}
    we find that $|E_{k+i} - \delta| < \epsilon$ for $i =1,2,\ldots,j$.  Next, we observe that
  \begin{align*}
    \sqrt{\alpha}\sqrt{1 + \delta}  = \frac{\beta}{\sqrt{\alpha}} = \frac{\sqrt b - \sqrt a}{2},\\
    \frac{\beta}{\sqrt{\alpha}\sqrt{1 + \delta}} = \sqrt{\alpha} = \frac{\sqrt b + \sqrt a}{2}.
  \end{align*}
  We then take the ratio of the elements in the $(k+i+1)$th column of $\varphi(T)$, giving
  \begin{align*}
    \frac{\alpha (1 + E_{k+i})}{\beta (1 + f_{\beta}(k+i))} = \frac{\sqrt b+ \sqrt a}{\sqrt b - \sqrt a} \frac{ 1 + E_{k+i}}{1 + \delta} \frac{1}{1 + f_{\beta}(k+i)} \geq \frac{\sqrt b+ \sqrt a}{\sqrt b - \sqrt a} \frac{ 1 + \delta  - \epsilon}{1 + \delta} \frac{1}{1 + \epsilon} \geq \sigma_0 > 1,
  \end{align*}
  by further reducing $\epsilon$, if necessary.  We then consider applying the conjugate gradient algorithm to $T \vec x = \vec f_1$.  By Theorem~\ref{t:deterministic} we have that, in particular
  \begin{align*} 
    \frac{\|\vec r_{k+j}\|_2}{\|\vec r_{k}\|_2} \geq \sigma_0^j.
  \end{align*} 
  But we know that for any $k$
  \begin{align*}
  \sigma^{-1} \|\vec e_k\|_T^2 =  \sigma^{-1}    \vec e_k^* T \vec e_k \leq \| \vec r_{k}\|_2^2 =  \vec e_k^* T^2 \vec e_k \leq \sigma  \vec e_k^* T \vec e_k = \sigma \|\vec e_k\|_T^2.
  \end{align*}
  This results in the string of inequalities
  \begin{align*}
    \sigma_0^j \leq \frac{\|\vec r_{k+j}\|_2}{\|\vec r_{k}\|_2} \leq \sigma \frac{\|\vec e_{k+j}\|_T}{\|\vec e_{k}\|_T} \leq \sigma,
  \end{align*}
  because $\frac{\|\vec e_{k+j}\|_T}{\|\vec e_{k}\|_T} \leq 1$.  Since $j$ can be made arbitrarily large, we see that $\delta$ cannot be a limit point of $(E_n)_{n \geq 0}$ and $\lim_{n \to \infty} E_n = 0$.
  
\end{proof}

This immediately implies the following.
\begin{corollary}\label{c:Cholesky_rate}
  Given the assumptions of Theorem~\ref{t:Asymptotic_Cholesky}, suppose there exists $C,c > 0$ such that $|f_\beta(n)| + |f_{\gamma}(n)| \leq C \e^{- cn}$ then there exists $C',c' > 0$ such that
  \begin{align*}
      |E_n| \leq C' e^{-c' n}.
  \end{align*}
\end{corollary}


\begin{proposition}\label{p:chol_inv_moment}
Suppose $h: [a,b] \to \mathbb R$, $a > 0$ is a positive real analytic function. Consider a probability measure $\mu$ defined by
  \begin{align*}
    \mu(\sd \lambda) = h(\lambda) \one_{[a,b]}(\lambda) (b - \lambda)^{\alpha} (\lambda - a)^\beta \sd \lambda + \sum_{j=1}^p w_j \delta_{c_j}(\sd \lambda)
  \end{align*}
  where $w_j > 0$ and $c_j > b$ for all $1 \leq j \leq p$.  Suppose, in addition, that $\alpha = \pm \frac 1 2, \beta = \pm \frac 1 2$.  Let
  \begin{align*}
      \mathcal T = \begin{bmatrix} a_0 & b_0 \\
     b_0 & a_1 & b_1 \\
      & b_1 & a_2 & b_2 \\
      && b_2 & a_3 & \ddots \\
      &&& \ddots & \ddots
      \end{bmatrix} 
  \end{align*}
   be the associated Jacobi matrix of three-term recurrence coefficients.  Let $\mathcal L \mathcal L^T = \mathcal T$ be the Cholesky factorization of $\mathcal T$ with
   \begin{align*}
      \mathcal L = \begin{bmatrix} \alpha_0  \\
     \beta_0 & \alpha_1  \\
      & \beta_1 & \alpha_2 \\
      && \beta_2 & \alpha_3  \\
      &&& \ddots & \ddots
      \end{bmatrix} 
   \end{align*}
   then
\begin{align*}
\int_{\mathbb R} \frac{1}{\lambda} \mu(\sd \lambda) = \bm{f}_1^* \mathcal T^{-1} \bm{f}_1 = \frac{1}{\alpha_0^2} \sum_{\ell=0}^{\infty} \prod_{j=1}^\ell \frac{\beta_{j-1}^2}{\alpha_j^2}.
\end{align*}
\end{proposition}
\begin{proof}
This follows from back substitution and the fact that
\begin{align*}
    \int_{\mathbb R} \frac{\mu(\sd \lambda)}{\lambda -z} = \bm{f}_1^* (\mathcal T - z)^{-1} \bm{f}_1,
\end{align*}
for $z$ outside the support of $\mu$ \cite{DeiftOrthogonalPolynomials}.
\end{proof}

We point out that Proposition \ref{p:chol_inv_moment} is true much more generally but this is the version we require.

\begin{proposition}\label{prop:inv_moment_minus}
With the assumptions of Proposition~\ref{p:chol_inv_moment}
\begin{align*}
    \prod_{j=0}^{k-1} \frac{\alpha_{j}^2}{\beta_{j}^2}\left[\int_{\mathbb R} \frac{1}{\lambda} \mu(\sd \lambda) - \frac{1}{\alpha_0^2} \sum_{\ell=0}^{k-1} \prod_{j=1}^\ell \frac{\beta_{j-1}^2}{\alpha_j^2}\right] = \frac{1}{\alpha_k^2} \sum_{\ell=0}^{k-1} \prod_{j=1}^\ell \frac{\beta_{k+j-1}^2}{\alpha_{k+j}^2} \To[k] \frac{1}{\sqrt{ab}}.
\end{align*}
Furthermore, this limit takes place at an exponential rate.
\end{proposition}

\begin{remark}
The convergence of the CGA is determined by the ratio of diagonal to off-diagonal entries in the Cholesky factorization of the associated Jacobi matrix.  For $0 < c_N <1$ following Jacobi matrix
  \begin{align*}
  \mathcal{T}=  \begin{bmatrix}
      c_N &  \sqrt{c_N} \\
      \sqrt{c_N} & 1 + c_N & \sqrt{c_N} \\
      & \sqrt{c_N} & \ddots
    \end{bmatrix},
  \end{align*}
  pathologically has diagonal entries that are smaller than the off-diagonal entries.  Since any finite truncation of this matrix is invertible, CGA will experience residuals that grow exponentially until convergence at $k = N$.  This is an example where, in the notation of Lemma~\ref{l:En}, $E_\infty = \frac{\beta^2}{\alpha^2} -1.$  Since this is an unstable fixed point of $F(x) = \frac{\beta^2}{\alpha^2} \left[ 1 - \frac{1}{1 + x}\right]$, any small (generic) perturbation, that preserves definiteness, will force $E_\infty = 0$.
\end{remark}

\section{Spiked covariance matrix model and VESD}\label{sec_rmt}
In this section, we provide and prove the results on random matrices. We first introduce some notations. For any $N \times N$ symmetric matrix $Z,$  denote $m_Z$ and $m_{Z,\bm{b}}$ as the Stieltjes transforms of $\mu_Z$ as in (\ref{defn_esd}) and $\mu_{Z, \bm{b}}$ as in (\ref{eq_defnvesd}), respectively, i.e.,
\begin{equation*}
m_Z(z)=\int \frac{1}{x-z} \mu_Z(\dd x), \ m_{Z,\bm{b}}(z)=\int \frac{1}{x-z} \mu_{Z,\bm{b}}(\dd x), \ z \in \mathbb{C}_+. 
\end{equation*} 
Recall that the Stieltjes transform can be used to recover the associated probability distribution $\mu$ using the well-known inversion formula (see equation (1.2) of \cite{SC})
\begin{equation}\label{eq_inversion}
\mu\{[a,b]\}=\frac{1}{\pi} \int_a^b  \Im m_{\mu}(x+\ri 0^{+}) \dd x. 
\end{equation}
Moreover, let $G_Z$ be the resolvent of $Z,$ i.e., $G_Z(z)=(Z-z)^{-1}.$ Then
\begin{equation*}
m_Z=\frac{1}{M} \operatorname{Tr} G_Z(z), \ m_{Z, \bm{b}}=\bm{b}^* G_Z(z) \bm{b}.     
\end{equation*}

Next, we introduce the following contour representation for the moments of any given spectral measure $\nu.$ Let $\mathfrak{m}_k(\nu)$ denote the moments of $\nu.$ By  Cauchy's integral formula,
\begin{equation}\label{eq_contourrepresentation}
\mathfrak{m}_k(\nu)=\frac{1}{2 \pi \ri} \oint_{\Gamma} z^k m_{\nu}(z) \dd z,
\end{equation} 
where $\Gamma$ is a smooth simple contour that properly encloses the support of $\nu$.

\subsection{Local laws for the non-spiked model}
In this subsection, we discuss results relating to the so-called anistropic local laws. Denote  by $H$ the $(N+M) \times (N+M)$ linearized matrix
\begin{equation}\label{eq_defnh}
H \equiv H(z,X):=\sqrt{z}
\begin{pmatrix}
0 &   \Sigma_0^{1/2} X \\
 X^* \Sigma_0^{1/2} & 0
\end{pmatrix}.
\end{equation}
$H$ is more convenient since, on one hand the eigenvalues of the sample covariance matrix $W_0$ can be studied via $H,$ and on the other hand  the resovlent of $H$ can be written in terms of those of $W_0$ and $\mathcal{W}_0.$  Let $G_1$ and $G_2$ be the resolvents of $W_0$ and $\mathcal{W}_0,$ and $m_1$ and $m_2$ be the Stieltjes transforms of the ESDs of $W_0$ and $\mathcal{W}_0,$ respectively.

For $z \in \mathbb{C}_+,$ by Schur's complement, we have that 
\begin{equation}\label{eq_defnG}
G(z) \equiv G(z, X):=(H-z)^{-1}= 
\begin{pmatrix}
G_1(z) & \frac{1}{\sqrt{z}} \Sigma_0^{1/2} X G_2(z) \\
\frac{1}{\sqrt{z}} G_2(z) X^* \Sigma_0^{1/2} & G_2(z)
\end{pmatrix}.
\end{equation}
Define the deterministic matrix 
\begin{equation}\label{eq_defnpi}
\Pi(z)\equiv 
\begin{pmatrix}
\Pi_1(z) & 0\\
0 & \Pi_2(z) 
\end{pmatrix}
:=
\begin{pmatrix}
-\frac{1}{z} (1+m(z)\Sigma_0)^{-1} & 0 \\
0 & m(z)
\end{pmatrix}.
\end{equation}

With a slight modification of the results in \cite{Knowles2017}, we have the following result. Fix some small constant $\tau>0$ and denote the set of admissible spectral parameters as 
\begin{equation}\label{eq_setmathcald}
\mathcal{D} \equiv \mathcal{D}(z,\tau)=\left\{z=E+\ri \eta: \tau \leq |z| \leq \tau^{-1}, \ M^{-1+\tau} \leq  \eta \leq  \tau^{-1} \right\}.
\end{equation}
A subset $\mathcal{D}_o$ of $\mathcal{D}$ is defined by
\begin{equation}\label{eq_setmathcaldout}
\mathcal{D}_o \equiv \mathcal{D}_o(z,\tau)=\mathcal{D} \cap \left\{\operatorname{dist}(E, \operatorname{supp}(\varrho))+\eta \geq \tau \right\}.
\end{equation}
\begin{lemma}\label{lem_locallaw} 
Suppose (1)--(3) of Assumption \ref{assum_summary} hold. For any unit deterministic vectors $\mathbf{u}, \mathbf{v} \in \mathbb{R}^{M+N}$ and fixed small constantt $\tau>0,$ we have that for all $z \in \mathcal{D}_o(z,\tau)$
\begin{equation*}
\left| \mathbf{u}^* G(z)\mathbf{v} - \mathbf{u}^* \Pi(z) \mathbf{v} \right| \prec M^{-1/2}.
\end{equation*}
\end{lemma}
\begin{proof}
See Appendix \ref{sec_appendix_technical}. 
\end{proof}

%
We remark that the results of \cite{Knowles2017} are established on the larger domain  $\mathcal{D}$ defined in (\ref{eq_setmathcald}) with the extra assumption that $\gamma_- \geq \tau.$ As discussed in \cite[Remark 1.8]{XYY}, this assumption requires that $|c_N-1| \geq \tau.$ In this sense, on the spectral parameter set $\mathcal{D}_0$ in (\ref{eq_setmathcaldout}), we can handle the case $c_N=1,$ which is an important regime in numerical analysis. We also have the following edge convergence result.  Denote the eigenvalues of $W_0$ in (\ref{eq_definitioncovariance}) as $\lambda_1 \geq \lambda_2 \geq \cdots.$
\begin{proposition}\label{prop_edgeeigenvalue} Suppose (1)--(3) of Assumption \ref{assum_summary} hold, we have that 
\begin{equation*}
\lambda_1=\gamma_++\OO_{\prec}(M^{-2/3}). 
\end{equation*} 
\end{proposition}
\begin{proof}
The proposition follows from \cite[Theorem 3.12]{Knowles2017}. 
\end{proof}

We have focused our discussion on the ESD so far. Armed with these results, we proceed to provide some results for the VESD. For any given deterministic unit vector $\bm{v} \in \mathbb{R}^N,$ denote 
\begin{equation}\label{eq_defnwi}
w_i=\langle \bm{v}, \bm{v}_i \rangle,  \ 1 \leq i \leq N. 
\end{equation}
Recall (\ref{eq_defnpi}). By Lemma \ref{lem_locallaw}, we find that $\bm{v}^* G_1(z) \bm{v}$ is close to 
\begin{equation}\label{eq_vesdlimitingform}
m_{\bm{v}}(z)=\bm{v}^* \Pi_1(z) \bm{v}= -\frac{1}{z}\sum_{i=1}^N \frac{w_i^2}{1+\sigma_i m(z)}.
\end{equation}
We denote the probability measure associated with $m_{\bm{v}}$ as $\varrho_{\bm{v}}.$ Note that
\begin{equation}\label{eq_impart}
\Im m_{\bm{v}}(x+\ri 0^{+})=\frac{1}{x} \sum_{i=1}^N  \frac{w_i^2 \sigma_i \Im m(x+\ri 0^{+})}{|1+\sigma_i m(x+\ri 0^{+})|^2},
\end{equation}
where we denote $\Im m(x+\ri 0^{+})=\lim_{\eta \downarrow 0} \Im m(x+\ri \eta).$ Together with the inversion formula (\ref{eq_inversion}), we see that
\begin{equation}\label{eq_supportidentical}
\operatorname{supp}(\varrho_{\bm{v}})=\operatorname{supp}(\varrho). 
\end{equation}

\subsection{VESD for the spiked covariance matrix model}\label{sec_spikedmatrix} In this subsection, we provide some results regarding the spiked model $W$ as in (\ref{eq_generalmodel}). For the spiked model, we can define $\widetilde{H}$ by replacing $\Sigma_0$ with $\Sigma$ in (\ref{eq_defnh}). 
Analogously, we can define the resolvents as  $\widetilde{G}_1, \widetilde{G}_2$ and $\widetilde{G},$ respectively. 
The following lemma collects the results on the asymptotic convergence of the outlier and extremal non-outlier eigenvalues. Denote the eigenvalues of $W$ in (\ref{eq_generalmodel}) as $\mu_1 \geq \mu_2 \cdots.$

\begin{lemma}[Outlier and extremal non-outlier eigenvalues]\label{lem_spikedmodeloutliereigenvalue} Suppose Assumption \ref{assum_summary}  holds. Recall the function $f$ defined in (\ref{eq_defnstitlesjtransform}). We have that 
\begin{equation*}
\mu_i=f\left(-\widetilde{\sigma}^{-1}_i \right)+\OO_{\prec}(M^{-1/2}), \ i \leq r,
\end{equation*}
and 
\begin{equation*}
\mu_{r+1}=\gamma_++\OO_{\prec}(M^{-2/3}). 
\end{equation*}
\end{lemma}
\begin{proof}
See Theorem 3.2 of \cite{DRMTA}. 
\end{proof}

In the following lemma, we establish the fundamental connection between the VESDs of the  the spiked and non-spiked models. Recall $\mathcal{D}_o$ in (\ref{eq_setmathcaldout}).   Denote the spectral parameter set 
\begin{equation}\label{eq_refineset}
\widetilde{\mathcal{D}}_o:=\mathcal{D}_o \cap \left\{ \min_{1 \leq i \leq r} |z-f(-\widetilde{\sigma}_i^{-1})| \geq \tau \right\},
\end{equation}   
where $\tau>0$ is some small fixed constant. 
\begin{lemma}\label{lem_connectionspikednonspiked} For the eigenvectors $\{\bm{v}_i\}$ of $\Sigma$ and any unit deterministic vector $\bm{v} \in \mathbb{R}^N,$ let $w_i$ as in (\ref{eq_defnwi}) and
\begin{equation*}
\mathcal{L}_i:=
\begin{cases}
z^{-1}(1+m(z) \sigma_i)^{-2}\left[ d_i^{-1}+1-(1+m(z) \sigma_i)^{-1}\right]^{-1} &  i \leq r \\
0 & r+1 \leq i \leq N.     
\end{cases} 
\end{equation*}
Suppose Assumption \ref{assum_summary} holds.  Then for all $z \in \widetilde{\mathcal{D}}_o$ in (\ref{eq_refineset}), 
\begin{equation}\label{eq_leftpertub}
\bm{v}^* \widetilde{G}_1(z) \bm{v}=\sum_{i=1}^N \frac{w_i^2}{1+d_i}\left(\bm{v}_i^* G_1(z) \bm{v}_i- \mathcal{L}_i \right)+\OO_{\prec}(M^{-1/2}).
\end{equation}
Similarly, for any deterministic vector $\bm{u} \in \mathbb{R}^M,$
\begin{equation}\label{eq_rightpertub}
\bm{u}^* \widetilde{G}_2(z) \bm{u}=\bm{u}^* G_2(z) \bm{u}+\OO_{\prec}(M^{-1/2}).
\end{equation}
\end{lemma}

\begin{proof}
See Appendix \ref{sec_appendix_technical}.
\end{proof}

\begin{remark}
Lemma \ref{lem_connectionspikednonspiked} provides useful expressions for the VESD of the spiked model in terms of the non-spiked model. First, for the VESD of $W$ in (\ref{eq_generalmodel}), as illustrated in (\ref{eq_leftpertub}), it can be described using that of $W_0$ in (\ref{eq_definitioncovariance}) after proper scaling and shifting. Especially, when $\bm{b} \in \mathbf{V}_r^{\perp},$ the VESDs of $W$ and $W_0$ coincide asymptotically. Moreover,  the values of $\mathcal{L}_i$ can be calculated explicitly at some specific points. Using the relation (\ref{eq_inverserelation}) that $m(f(-\widetilde{\sigma}_i^{-1}))=-\widetilde{\sigma}_i^{-1},$ we readily find that 
\begin{equation*}
d_i^{-1}+1-(1+m(f(-\widetilde{\sigma}_i^{-1})) \sigma_i)^{-1}=0.
\end{equation*}
Therefore, we conclude that $f(-\widetilde{\sigma}_i^{-1})$ is a pole of $\mathcal{L}_i.$ Second, (\ref{eq_rightpertub}) states that the VESDs of $\mathcal{W}=X^*\Sigma X$ and $\mathcal{W}_0=X^* \Sigma_0 X$ match asymptotically regardless of the existence of the spikes. As will be seen in the proof of Theorem \ref{thm_normalequation}, it explains why the spikes will be ignored when the CGA is applied to normal equation. 

\end{remark}

\subsection{Formulation of the moments of VESDs}
In this subsection, we establish the key relation for the (random) moments of the VESDs for the spiked and non-spiked models. In particular, we represent the moments of the VESD of the spiked model using those of the non-spiked model.  Denote the VESDs of $(W_0, \bm{b})$ and $(W, \bm{b})$ as $\nu_{\bm{b}}$ and $\widetilde{\nu}_{\bm{b}},$ respectively. Recall that their moments are defined  as follows 
\begin{equation}\label{eq_momemntempericialvesd}
\widehat{\mathfrak{m}}_{k, \bm{b}}=\int x^k \nu_{\bm{b}}(\dd x), \   \widehat{\widetilde{\mathfrak{m}}}_{k, \bm{b}}=\int x^k \widetilde{\nu}_{\bm{b}}(\dd x).
\end{equation}
\begin{theorem}\label{lem_momentsconnectionspikednonspiked}
Suppose Assumption \ref{assum_summary} holds. Recall (\ref{eq_defnbi}). We have that  
\begin{equation*}
\widehat{\widetilde{\mathfrak{m}}}_{k, \bm{b}}=\sum_{i=1}^N \frac{\mathtt{b}_i^2}{1+d_i}\left( \widehat{\mathfrak{m}}_{k,\bm{v}_i}- \mathbf{1}(i \leq r)\frac{f'(-\widetilde{\sigma}_i^{-1}) \left( f(-\widetilde{\sigma}_i^{-1})\right)^{k-1}}{\sigma_i} \right)+\OO_{\prec}(M^{-1/2}), \ \text{for all integers} \  k \geq 0.
\end{equation*}
Moreover, if $\gamma_- \geq \tau$ for some constant $\tau>0$ the above results extend to $k=-1.$
\end{theorem}
\begin{proof}
Recall $\varrho_{\bm{b}}$ is the limiting VESD associated with the Stieltjes transform in (\ref{eq_vesdlimitingform}). By \cite[Theorem 1]{BMP}, we have that $\nu_{\bm{b}} \rightarrow \varrho_{\bm{b}}$ weakly a.s.. In order to apply (\ref{eq_contourrepresentation}), we first properly choose a contour.  In light of (\ref{eq_supportidentical}), we can choose a simply connected contour $\Gamma$ that encloses the support of the deformed MP law $\varrho$ and $f(-\widetilde{\sigma}_i^{-1}), 1 \leq i \leq r$ and is also uniformly bounded away from them. 

Then we apply 
(\ref{eq_contourrepresentation}) for the calculation.
It is easy to check that the function $f$ defined in (\ref{eq_defnstitlesjtransform}) is monotonically increasing when $x \geq m(\gamma_+);$ for example, see the discussion below \cite[Lemma 6.1]{DRMTA}.
 Moreover, under Assumption \ref{assum_summary},
we find that for some constant $\tau'>0$
\begin{equation}\label{eq_boundoneone}
-\widetilde{\sigma}_i^{-1}>m(\gamma_+)+\tau'.
\end{equation}
Therefore, we have that
\begin{equation}\label{eq_conditionone}
f(-\widetilde{\sigma}_i^{-1}) \geq f(m(\gamma_+))=\gamma_+.
\end{equation}
Note that we have $m(\gamma_+)=b_1$ and $f'(b_1)=0.$ Further, for $x \geq \gamma_+,$  by the square root behavior of $\varrho,$ we have that \cite[equation (A.11)]{Knowles2017} 
\begin{equation}\label{eq_squarerootequation}
x-\gamma_+=\frac{f^{''}(b_1)}{2}(m(x)-b_1)^2+\OO(|x-b_1|^3). 
\end{equation}
From the proof of \cite[Lemma A.3]{Knowles2017}, we have that for some constant $\tau_2>0,$
\begin{equation*}
f^{''}(b_1) \geq \tau_2,
\end{equation*}
Since (\ref{eq_conditionone}) holds,  we set $x=f(-\widetilde{\sigma}_i^{-1})$ and evaluate (\ref{eq_squarerootequation}).  By (\ref{eq_boundoneone}), we conclude that for some constant $\tau_3>0,$
\begin{equation}\label{eq_bbbbb}
f(-\widetilde{\sigma}_i^{-1})-\gamma_+>\tau_3. 
\end{equation}
Together with (\ref{eq_supportidentical}), Proposition \ref{prop_edgeeigenvalue} and Lemma \ref{lem_spikedmodeloutliereigenvalue}, we find that $f(-\widetilde{\sigma}_i^{-1})$ are isolated points and uniformly far away from the support of $\varrho.$ Therefore, (\ref{eq_contourrepresentation}) implies  
\begin{equation*}
\widehat{\mathfrak{m}}_{k, \bm{b}}=\frac{1}{2 \pi \ri} \oint_{\Gamma} z^k \bm{b}^* G_1(z) \bm{b} \dd z, \ k \geq 0.   
\end{equation*} 
The above results hold for $k<0$ when $0 \notin \operatorname{supp}(\varrho),$ i.e., $\gamma_- \geq \tau$ for some constant $\tau>0.$ Moreover, by Lemma \ref{lem_connectionspikednonspiked}, we have that
\begin{align}\label{key_residual}
\widehat{\widetilde{\mathfrak{m}}}_{k, \bm{b}} & =\frac{1}{2 \pi \ri} \oint_{\Gamma} z^k \bm{b}^* \widetilde{G}_1(z) \bm{b} \dd z \\
&=\sum_{i=1}^N \frac{\mathtt{b}_i^2}{1+d_i}\frac{1}{2 \pi \ri}\left( \oint_{\Gamma} z^k \bm{v}_i^* G_1(z) \bm{v} \dd z-  \oint_{\Gamma} z^k \mathcal{L}_i(z)  \dd z \right)+\OO_{\prec}(M^{-1/2}) \nonumber \\
& =\sum_{i=1}^N \frac{\mathtt{b}_i^2}{1+d_i} \widehat{\mathfrak{m}}_{k, \bm{v}_i}-\frac{1}{2 \pi \ri}\sum_{i=1}^r \frac{\mathtt{b}_i^2}{1+d_i}\oint_{\Gamma} z^k \mathcal{L}_i(z) \dd z+\OO_{\prec}(M^{-1/2}). \nonumber
\end{align}
Next, we discuss the residues. Using (\ref{eq_bbbbb}), Assumption \ref{assum_summary}(3) and the monotonicity of $f$ on the real line, we conclude that the singularities of $\mathcal{L}_i$ are not within the support of $\varrho$.
Then we set $\Upsilon=m(\Gamma), i.e., f(\Upsilon)=\Gamma$ and use residue theorem to calculate
\begin{align*}
\frac{1}{2 \pi \ri}\oint_{\Gamma} z^k \mathcal{L}_i(z) \dd z& =\frac{1}{2 \pi \ri} \oint_{\Upsilon} (f(\zeta))^k  \mathcal{L}_i(f(\zeta)) f'(\zeta) \dd \zeta \\
& =\frac{1}{2 \pi \ri} \oint_{\Upsilon} (f(\zeta))^{k-1} f'(\zeta) \frac{1}{(1+\zeta \sigma_i )^2} \frac{d_i(1+\zeta \sigma_i)}{\widetilde{\sigma}_i} \frac{1}{\zeta+\widetilde{\sigma}_i^{-1}} \dd z \\
&=\frac{f'(-\widetilde{\sigma}_i^{-1}) \left( f(-\widetilde{\sigma}_i^{-1})\right)^{k-1}}{\sigma_i},
\end{align*}
where in the second step we used that $m(f(\zeta))=\zeta$ and in the last step we used Cauchy's integral formula and $\widetilde{\sigma}_i=\sigma_i(1+d_i)$. This completes our proof.  
\end{proof}

\begin{remark}\label{remark_relationbetweenrandomanddeterministic}
We remark that $\widehat{\mathfrak{m}}_{k, \bm{b}}$ can be replaced by  some deterministic quantities using the limiting VESD (c.f. $\varrho_{\bm{b}}$ in (\ref{eq_varrhob})). Recall $\mathfrak{m}_{k, \bm{b}}$ defined in (\ref{eq_momentdefinition}). According to \cite[Theorem 1]{BMP}, we have that $\widehat{\mathfrak{m}}_{k, \bm{b}} \rightarrow \mathfrak{m}_{k, \bm{b}}$ a.s.. The convergence rates have also been established under different assumptions in the literature. For example, by \cite[Theorem 1.6]{XQB}, it can be shown that  $\widehat{\mathfrak{m}}_{k, \bm{b}}=\mathfrak{m}_{k, \bm{b}}+\OO_{\mathbb{P}}(M^{-1/8}).$ Moreover, when $\gamma_- \geq \tau,$ the result can be updated to $\widehat{\mathfrak{m}}_{k, \bm{b}}=\mathfrak{m}_{k, \bm{b}}+\OO_{\mathbb{P}}(M^{-1/4}).$ Later on, under the assumption $|c_N-1| \geq \tau$ (or $\gamma_- \geq \tau$), the authors established that $\widehat{\mathfrak{m}}_{k, \bm{b}}=\mathfrak{m}_{k, \bm{b}}+\OO_{\prec}(M^{-1/2})$ in \cite[Theorem 1.5]{XYY}. 

\end{remark}

\section{Theoretical analysis of the algorithms}\label{sec_mainproof}
Armed with the results established in Sections \ref{sec_theoryoforthogonalpolynomial} and \ref{sec_rmt}, in this section, we provide the error analysis of the CGA and MINRES algorithms. Due to similarity, we focus on Theorem \ref{thm:general} and only briefly discuss that of Theorem \ref{thm:lan_general}.
\begin{proof}[\bf Proof of Theorem \ref{thm:general}] We focus our discussion on the non-spiked model and will only briefly discuss the spiked case. Recall (\ref{eq_defnvesd}). Denote by $\widehat{M}_k$ the Hankel determinant matrix using the VESD of $\mu_{W_0, \bm{b}}$ and recall that $M_k$ is its limiting version defined in (\ref{eq_hankeldeterminant}). Note that for any nonsingular matrix $A$, square matrix $B$ and small $\epsilon>0$ \cite{MR286814}
\begin{equation}\label{eq_elecontrol}
\det(A+\epsilon B)=(1+\epsilon \operatorname{tr}(BA^{-1}) )\det A+\OO(\epsilon^2).  
\end{equation}
Under the assumption that $\gamma_- \geq \tau_1,$ by Remark \ref{remark_relationbetweenrandomanddeterministic},  we find that 
\begin{equation}\label{eq_foundbound}
\det(\widehat{M}_k)=\det M_k+ \OO_{\prec}(C_k M^{-1/2}), 
\end{equation}
where $C_k$ is some constant which depends on $k.$
In fact, by (\ref{eq_elecontrol}), we have
\begin{align*}
\det \widehat{M}_k & =\det \left( M_k+M^{-1/2}(\sqrt{M}(\widehat{M}_k-M_k)) \right) \\
& =\det M_k+\OO(M^{-1})+M^{-1/2}\OO_{\prec}( \det M_k \operatorname{tr}(\sqrt{M}(\widehat{M}_k-M_k)M_k^{-1})).
\end{align*}
Note that $M_k$ is positive definite. Applying Hadamard's inequality to $\det M_k$ and the inequality that $\operatorname{tr}(AB) \leq \lambda_{\max}(A) \operatorname{tr}(B),$ where $B$ is a positive-definite matrix, by Remark \ref{remark_relationbetweenrandomanddeterministic}, we readily see that $C_k \leq \mathsf{a}^k$ for some constant $\mathsf{a}>0.$  

 Let $\widehat{\mathfrak{b}}_k$ be defined similarly as in (\ref{eq_defngeneralb}) using the moments of $\mu_{W_0, \bm{b}}.$ By (\ref{eq_foundbound})  and (\ref{eq_lnsn}) with (\ref{eq:sandl}), we readily see that  
\begin{equation}\label{eq_closebk}
\widehat{\mathfrak{b}}_k=\mathfrak{b}_k+ \OO_{\prec}(C'_k M^{-1/2}), 
\end{equation}
for some constant $C_k'$ which depends on $k.$ Similarly, we can show that 
\begin{equation}\label{eq_closeak}
\widehat{\mathfrak{a}}_k=\mathfrak{a}_k+ \OO_{\prec}(C'_k M^{-1/2}).  
\end{equation}
Let $\widehat{T}$ be the tridiagonal matrix constructed using $\{\widehat{\mathfrak{a}}_i\}$ and $\{\widehat{\mathfrak{b}}_i\}$ as in (\ref{eq_defnTnno}). Analogous to $L$ in (\ref{eq_cholesky}), we can apply Algorithm \ref{a:chol} to $\widehat{T}$  to obtain the Cholesky factorization $\widehat{L},$ whose entries are denoted as  $\{\widehat{\alpha}_j\}$ and $\{\widehat{\beta}_j\}$. By (\ref{eq_closebk}) and (\ref{eq_closeak}), it is easy to see that 
\begin{equation*}
\widehat{\alpha}_k=\alpha_k+ \OO_{\prec}(C''_k M^{-1/2}), \ \widehat{\beta}_k=\beta_k+ \OO_{\prec}(C''_k M^{-1/2}), 
\end{equation*} 
where $C''_k$ is some constant depending on $k.$ Consequently, by Lemma \ref{t:deterministic}, we conclude that for some constant $\mathtt{C}_{r,k}>0,$
\begin{equation*}
\| \bm{r}_k \|_2=\prod_{j=0}^{k-1} \frac{\widehat{\beta}_j}{\widehat{\alpha}_j}=\prod_{j=0}^{k-1} \frac{\beta_j}{\alpha_j}+\OO_{\prec}(\mathtt{C}_{r,k} M^{-1/2}). 
\end{equation*}
Similarly, we can prove the results for $\| \mathbf{e}_k \|_{W_0}. $ Finally, for the spiked model, (\ref{eq_closebk}) and (\ref{eq_closeak}) can be proved similarly using Lemma \ref{lem_connectionspikednonspiked} and Theorem \ref{lem_momentsconnectionspikednonspiked}. This completes our proof. 
\end{proof}

\begin{proof}[\bf Proof of Theorem \ref{thm:lan_general}] The first part of the results follow from (\ref{eq_closebk}), (\ref{eq_closeak}),  the fact that $T_k$ is banded  and the Gershgorin circle theorem. The second part of the results follows from Remark \ref{rmk_moments}, Theorem \ref{lem_momentsconnectionspikednonspiked} and Remark \ref{remark_relationbetweenrandomanddeterministic}.  

\end{proof}

\begin{proof}[\bf Proof of Theorem \ref{thm:mainone}] 

First, we consider the non-spiked case. Using (\ref{eq_defnbi}) and (\ref{eq_varrhob}), we see that 
\begin{equation*}
\varrho_{\bm{b}}(x)=h_1(x) \varrho(x),
\end{equation*}
where $\varrho(x)$ is the deformed MP law and $h_1(x)$ is analytic and is given by
\begin{equation*}
h_1(x)=\sum_{i=1}^N  \frac{\mathtt{b}_i \sigma_i }{x(1+2 \sigma_i \Re m(x)+|m(x)|^2 \sigma_i^2)}. 
\end{equation*}
For the deformed MP law, by \cite[Section A.2]{Knowles2017},  we obtain that there exists some analytic function $h_2(x)$ such that
\begin{equation*}
\varrho=h_2(x) \sqrt{(\gamma_+-x)(x-\gamma_-)}.
\end{equation*}
Consequently, we have that
\begin{equation*}
\varrho_{\bm{b}}(x)=h(x)\sqrt{(\gamma_+-x)(x-\gamma_-)}, \ h(x)=h_1(x) h_2(x). 
\end{equation*}
Recall (\ref{eq_aabb}). By Theorem \ref{thm_relationOP}, we immediately obtain that
\begin{equation*}
\mathfrak{a}_k = \mathfrak{a}+\OO(e^{-ck}), \ \mathfrak{a}_k=\mathfrak{b}+\OO(e^{-ck}), 
\end{equation*}
where $c>0$ is some constant. Applying Corollary \ref{c:Cholesky_rate} to the Jacobi matrix defined in (\ref{eq_defnTnn}), under the assumption that $\gamma_- \geq \tau_1$, it is easy to see that the diagonal and off-diagonal entries, respectively, satisfy,
\begin{equation*}
\alpha_k = \frac{\sqrt{\gamma_+}+\sqrt{\gamma_-}}{2} + \OO(e^{- c' k}), \quad\quad \beta_k = \frac{\sqrt{\gamma_+}-\sqrt{\gamma_-}}{2} + \OO(e^{- c' k}).
\end{equation*}
This completes (1) and (2) using Theorem \ref{thm:general} and Remark \ref{rmk_explicit}.

Second, for the spiked case, when $\bm{b} \in \mathbf{V}_r$, according to (\ref{eq_pertubedmoment}), we find that
\begin{equation*}
\widetilde{\mathfrak{m}}_{k,\bm{b}}=\mathfrak{m}_{k, \bm{b}}. 
\end{equation*}
Since all the $\beta_j, \alpha_j$ and $S_k$ are functions constructed via the Hankel moment matrices, (1) and (2) hold for the spiked model. When $|\langle \bm{b}, \bm{v}_i \rangle| \tau_1$ for some $1 \leq i \leq r,$ the results follow from (1) and (2) using Lemma \ref{lem_connectionspikednonspiked} and Theorem \ref{thm_relationOP}.

\end{proof}


\begin{proof}[\bf Proof of Theorem \ref{thm_haltingtime}]
The proof follows directly from Theorem \ref{thm:mainone}. 
\end{proof}

\begin{proof}[\bf Proof of Theorem \ref{thm_normalequation}] Recall (\ref{eq_averagetrace}).  Since $\bm{a}$ is a unit vector, using \cite[Lemma A.4]{NEK}, it is easy to see that 
\begin{equation*}
\bm{a}^* Y^* Y \bm{a}=\mathtt{w}+\OO_{\prec}(M^{-1/2}).
\end{equation*}
Moreover, it is clear that Theorem \ref{thm:general} applies to 
$$Y_0Y_0^* \frac{\bm{x}}{\| Y_0 \bm{a} \|_2}=\frac{Y_0 \bm{a}}{\| Y_0 \bm{a} \|_2}.$$
Note that the VESD satisfies that
\begin{equation*}
\bm{a}^* Y_0^* G_1(z) Y_0 \bm{a} =\bm{a}^* G_2(z)Y^* Y \bm{a}=\bm{a}^*G_2 (Y^*Y-z+z)\bm{a}=1+z \bm{a}^* G_2(z) \bm{a}. 
\end{equation*}
Consequently, by Lemma \ref{lem_locallaw} and (\ref{eq_inversion}), its limiting asymptotic density will be
\begin{equation*}
\varrho'(x)=x \varrho(x),  
\end{equation*}
and its moments are as in (\ref{eq_normalequationform}). This completes the proof of the non-spiked model.

For the spiked model, since the formulas are functions of the moments of the VESD, it suffices to show the closeness of the moments of the VESDs of the spiked and non-spiked model, denoted as $\widehat{\widetilde{\mathfrak{m}}}_k$ and $\widehat{\mathfrak{m}}_k,$ respectively. When $Y=\Sigma^{1/2} X,$ the VESD satisfies that   
\begin{equation}\label{eq_transfer}
\bm{a}^* Y^* \widetilde{G}_1(z) Y \bm{a} =\bm{a}^* \widetilde{G}_2(z)Y^* Y \bm{a}=\bm{a}^*\widetilde{G}_2 (Y^*Y-z+z)\bm{a}=1+z \bm{a}^* \widetilde{G}_2(z) \bm{a}. 
\end{equation}
Together with (\ref{eq_rightpertub}), we immediately obtain that 
\begin{equation*}
\bm{a}^* Y^* \widetilde{G}_1(z) Y \bm{a}=1+z \bm{a}^* G_2(z) \bm{a}+\OO_{\prec}(M^{-1/2}). 
\end{equation*}
By a discussion similar to (\ref{key_residual}), we can show that 
\begin{equation*}
\widehat{\widetilde{\mathfrak{m}}}_k=\widehat{\mathfrak{m}}_k+\OO_{\prec}(M^{-1/2}),
\end{equation*}
This completes our proof. 
\end{proof}

\begin{proof}[\bf Proof of Theorem \ref{thm_minres}] The proof is similar to that of Theorem \ref{thm:general} except that we use the deterministic formula for MINRES in Lemma \ref{t:deterministic}.  

\end{proof}


\section{Universality: Proof of Theorem \ref{thm:mainthree}}\label{sec_proofuniversality}
In this section, we study the universality of the fluctuations of the norms of the residual and error vectors for the CGA and prove Theorem \ref{thm:mainthree}. Until the end of this section, for simplicity, we denote $\mu_{x}$ and $\mu_{y}$  as the VESDs of $(\Sigma^{1/2}_0XX^* \Sigma_0^{1/2}, \bm{b})$ and $(\Sigma^{1/2}_0YY^* \Sigma_0^{1/2},\bm{b}),$ respectively, where  $Y$ is some random matrix whose first four moments are specified. Denote by $\mathfrak{m}_{k}(x)$ and $\mathfrak{m}_{k}(y)$ as the moments of $\mu_x$ and $\mu_y,$ respectively. 
Moreover,  we set 
\begin{equation*}
c_0(z; \mu)=m_{\mu}(z), \ s_{c_N}(z)=m_{\varrho_{\bm{b}}}(z),
\end{equation*}
where $\varrho_{\bm{b}}$ is defined in (\ref{eq_varrhob}). 

\subsection{Proof of Theorem \ref{thm:mainthree}}
In this subsection, we prove Theorem \ref{thm:mainthree}. We will use the following definition. 
\begin{definition}\label{defn_admissible}
Fix some integer $0<r\leq M^{\epsilon_0}$ for some sufficiently small constant $\epsilon_0>0.$ Let $\Phi: \mathbb{R}^r \rightarrow \mathbb{R}$ be bounded. Suppose, in addition, that for any multi-index $\bm{\alpha}=(\alpha_1,\cdots, \alpha_n), 1 \leq |\bm{\alpha}| \leq 5$ and for any $\epsilon'>0$ sufficiently small, we have 
\begin{equation*}
\max\{|\partial^{\bm{\alpha}} \Phi(x_1, \cdots, x_r): \max_j|x_j| \leq M^{\epsilon'}|\} \leq M^{C_0 \epsilon'},
\end{equation*}
for some $C_0>0.$ Then $\Phi$ is called an admissible test function. \\

Here we use the convention that for any positive integer $m,$ some function $\Phi: \mathbb{R}^m \rightarrow \mathbb{R}$ and $\bm{x}=(x_1, \cdots, x_m) \in \mathbb{R}^m,$ we denote 
\begin{equation}\label{eq_convention1}
\partial^{\bm{k}} \Phi(\bm{x})=\frac{\partial^{|\bm{k}|} \Phi}{\partial x_1^{k_1} \partial x_2^{k_2} \cdots \partial x_m^{k_m}}, \ \bm{k}=(k_1, \cdots, k_m),
\end{equation}
and 
\begin{equation}\label{eq_convention2}
\bm{x}^{\bm{k}}=\prod_{i=1}^m x_i^{k_i}, \ \bm{k}!=\prod_{i=1}^m k_i!. 
\end{equation}

\end{definition}

\begin{proof}[\bf Proof of Theorem  \ref{thm:mainthree}]
 According to Lemma \ref{t:deterministic} and Remark \ref{rmk_moments}, since $\{\alpha_j\}$ and $\{\beta_j\}$ are locally analytic of the moments of the VESDs (c.f. (\ref{eq_momemntempericialvesd})), it suffices to establish the university for  smooth functions of the moments. According to (\ref{eq_contourrepresentation}), for some properly chosen contour $\Gamma,$ we have that
\begin{equation*}
\mathfrak{m}_k(\ell)=\frac{1}{2 \pi \ri} \oint_{\Gamma} z^k c_0(z; \mu_{\ell}) \dd z, \ \ell=x,y. 
\end{equation*}    
Therefore, it suffices to handle the integral. We  point out that we only need to focus on the non-spiked model. Note $c_0(z; \mu_x)= \bm{b}^* (\Sigma_0^{1/2} XX^* \Sigma_0^{1/2}-z)^{-1} \bm{b}=\bm{b}^* G_1(z) \bm{b}.$ Denote $c_0(z; \widetilde{\mu}_x)$ as the associated Stieltjes transform for the spiked model, i.e., $c_0(z; \widetilde{\mu}_x)=\bm{b}^* \widetilde{G}_1(z) \bm{b}.$ By (\ref{eq_expansioncorrect}), it is easy to see that $c_0(z; \widetilde{\mu}_x)$ can be  expressed in terms of $c_0(z; \mu_x).$

Based on the above arguments, it is clear that the proof follows from the proposition below. 

%




\begin{proposition}\label{thm_unversality}
Suppose the assumptions of Theorem \ref{thm:mainthree} hold. 
For each $j,$ let $\Gamma_j=\partial \Omega_j,$ $\Omega_j=\overline{\Omega}_j$ be a simple smooth positively-oriented curve that is uniformly bounded away from the support of the deformed MP law $\varrho.$ Assume that $f_j, 1 \leq j \leq r,$ is a collection of functions that are analytic in a neighborhood of $\Omega_j, 1 \leq r.$ The for any admissible function $\Phi:\mathbb{R}^r \rightarrow \mathbb{R},$ we have that
\begin{align*}
& \left|  \mathbb{E} \Phi \left( \frac{\sqrt{M}}{2 \pi \ri} \oint_{\Gamma_1} f_1(z)(c_0(z; \mu_x)-s_{c_N}) \dd z, \cdots,  \frac{\sqrt{M}}{2 \pi \ri} \oint_{\Gamma_r} f_r(z)(c_0(z; \mu_x)-s_{c_N}) \dd z  \right) \right. \\
&- \left.  \mathbb{E} \Phi \left( \frac{\sqrt{M}}{2 \pi \ri} \oint_{\Gamma_1} f_1(z)(c_0(z; \mu_{y})-s_{c_N}) \dd z, \cdots,  \frac{\sqrt{M}}{2 \pi \ri} \oint_{\Gamma_r} f_r(z)(c_0(z; \mu_{y})-s_{c_N}) \dd z  \right)  \right| \leq C M^{-\delta},
\end{align*}
for some constants $C, \delta>0.$
\end{proposition}
\end{proof}

The proof of Proposition \ref{thm_unversality} will be provided in the next subsection. We provide some remarks before concluding this subsection. 


\begin{remark}
We point out that some relevant results have been established in the literature under various assumptions. In \cite[Theorem 2]{BMP}, provided the ESD of $\Sigma_0$ converges to some deterministic limiting distribution and  $c_N$ converges to some limit $c,$ 
under the assumption that Lemma \ref{lem_locallaw} holds and $\mathbb{E} x_{ij}^4=3/M, $ the authors proved that  
$(\int f_1(x) \mu_T(\dd x), \cdots,  \int f_r(x) \mu_T(\dd x))$ converges to some Gaussian random vector.  More recently, in \cite{Yang2020}, the authors generalized the above results without assuming  convergence of $\Sigma_0$ and $c_N$ and the moment matching conditions (\ref{assum_momentmatching}). Further,  \cite{Yang2020} considers a more general class of functions. However, the results of \cite{Yang2020} are established under the assumption that  $|c_N-1| \geq \tau.$ Our Theorem  \ref{thm_unversality} considers completely general population covariance matrices as in \cite{Yang2020} with $r$ being possibly slowly divergent,  but under the moment matching condition (\ref{assum_momentmatching}). Under (\ref{assum_momentmatching}), our results also hold even $c_N=1.$ Finally, we mention that for all $c_N \in (0, \infty)$, based on the results established in \cite{BDWW,Yang2020}, it is possible to derive the explicit distribution for the functional forms of the VESDs of $W_0$ in Proposition \ref{thm_unversality}, which depend on all the first four moments of $X$. We will pursue this direction in the future.    
\end{remark}




\begin{remark}
Since the support of $\varrho_{\bm{b}}$ is the same with that of $\varrho$ (c.f. (\ref{eq_supportidentical})), an immediate consequence of Proposition \ref{thm_unversality} is that 
\begin{equation*}
\left( \int \lambda^k \mu_{W_0, \bm{b}}(\dd \lambda) \right)_k \simeq \left( \int \lambda^k \varrho_{\bm{b}} \dd \lambda \right)_k,
\end{equation*}
in the sense of convergence of finite-dimensional marginals where $k \geq 0$ for $|c_N-1|<\tau$ and $k \in \mathbb{Z}$ if $c_N \leq 1 = \tau$ for some constant $\tau>0.$
\end{remark}

\subsection{Proof of Proposition \ref{thm_unversality}}
We proceed to the proof of Proposition \ref{thm_unversality} in this subsection. We point out that in  \cite[Theorem 5.11]{Paquette2020}, a similar result has been established  when $\Sigma_0=I$ and $Y$ is Gaussian. The proof of \cite[Theorem 5.11]{Paquette2020} relies on a discrete comparison method which only works for diagonal $\Sigma_0.$ For general $\Sigma_0,$ we need to use the interpolation method as developed in \cite{Knowles2017}.


For simplicity of notation, define the index sets
$$\mathcal I_1:=\{1,...,N\}, \ \ \mathcal I_2:=\{N+1,...,N+M\}, \ \ \mathcal I:=\mathcal I_1\cup\mathcal I_2.$$
We shall consistently use the latin letters $i,j\in\mathcal I_1$, greek letters $\mu,\nu\in\mathcal I_2$, and $a,b \in\mathcal I$. 

\begin{definition}[Interpolating matrices]
Introduce the notations $X^0:=Y$ and $X^1:=X$. Let $\rho_{i\mu}^0$ and $\rho_{i\mu}^1$ be the laws of $X_{i\mu}^0$ and $X_{i\mu}^1$, respectively. For $\theta\in [0,1]$, we define the interpolated law
$$\rho_{i\mu}^\theta := (1-\theta)\rho_{i\mu}^0+\theta\rho_{i\mu}^1.$$
We shall work on the probability space consisting of triples $(X^0,X^\theta, X^1)$ of independent $\mathcal I_1\times \mathcal I_2$ random matrices, where the matrix $X^\theta=(X_{i\mu}^\theta)$ has law
\begin{equation}\label{law_interpol}
\prod_{i\in \mathcal I_1}\prod_{\mu\in \mathcal I_2} \rho_{i\mu}^\theta(\dd X_{i\mu}^\theta).
\end{equation}
For $\lambda \in \mathbb R$, $i\in \mathcal I_1$ and $\mu\in \mathcal I_2$, we define the matrix $X_{(i\mu)}^{\theta,\lambda}$ through
\begin{equation}\label{eq_construction}
\left(X_{(i\mu)}^{\theta,\lambda}\right)_{j\nu}:=\begin{cases}X_{i\mu}^{\theta}, &\text{ if }(j,\nu)\ne (i,\mu)\\ \lambda, &\text{ if }(j,\nu)=(i,\mu)\end{cases}.
\end{equation}
In view of (\ref{eq_defnh}) and (\ref{eq_defnG}), we introduce the matrices \[G^{\theta}(z):=G\left(z,X^{\theta}\right),\ \ \ G^{\theta, \lambda}_{(i\mu)}(z):=G\left(z,X_{(i\mu)}^{\theta,\lambda}\right).\]
Furthermore, we denote the matrix
\begin{equation}\label{eq_defndelta}
\Delta_{(i \mu)}^{\lambda}:= \lambda \sqrt{z}
\begin{pmatrix}
0 & \Sigma_0^{1/2} \bm{f}_i \bm{f}_{\mu}^* \\
\bm{f}_{\mu} \bm{f}_i^* \Sigma_0^{1/2}& 0
\end{pmatrix}.
\end{equation}
\end{definition}
By resolvent expansion, we readily obtain that for $\lambda, \lambda' \in \mathbb{R}$ 
\begin{equation}\label{eq_resolventexpansion}
G_{(i \mu)}^{\theta, \lambda'}=G_{(i \mu)}^{\theta, \lambda}+\sum_{k=1}^K G_{(i \mu)}^{\theta, \lambda}  \left( \Delta_{(i \mu)}^{\lambda-\lambda'} G_{(i \mu)}^{\theta,\lambda} \right)^k+G_{(i \mu)}^{\theta, \lambda'}\left( \Delta_{(i \mu)}^{\lambda-\lambda'} G_{(i \mu)}^{\theta,\lambda} \right)^{K+1}  
\end{equation}
Setting $\lambda=X_{i \mu}^{\theta},$ by Lemma \ref{lem_locallaw}, for $z \in \mathcal{D}_o,$ since $\| \Sigma_0 \|<\infty,$ we readily obtain that 
\begin{equation}\label{eq_preparationbound}
 \left \langle \mathbf{u} (G_{(i \mu)}^{\theta, \lambda}-\Pi(z)), \mathbf{v} \right \rangle \prec M^{-1/2}, \ \left\| G_{(i \mu)}^{\theta, \lambda} \right\|=\OO_{\prec}(1).       
\end{equation}
Moreover, we set $\lambda'=0.$ Under Assumption \ref{assum_summary}, it is easy to see that $X_{i \mu}^{\theta}=\OO_{\prec}(M^{-1/2}).$ Using the definition of Stieltjes transform, it is trivial to see that  $\| G_{(i \mu)}^{\theta, \lambda} \| \leq C \eta^{-1}$ for some constant $C>0.$ Therefore, we can choose $K=2$ in (\ref{eq_resolventexpansion}) such that for all $z \in \mathcal{D}_o(z)$
\begin{equation*}
\left\| G_{(i \mu)}^{\theta, \lambda'}\left( \Delta_{(i \mu)}^{\lambda-\lambda'} G_{(i \mu)}^{\theta,\lambda} \right)^{K+1} \right\|=\OO_{\prec}(M^{-1/2+\tau}),
\end{equation*}
where used the structure of (\ref{eq_defndelta}). Together with (\ref{eq_resolventexpansion}) and (\ref{eq_preparationbound}), we readily obtain that 
\begin{equation}\label{eq_reducebound}
\left \langle \mathbf{u} (G_{(i \mu)}^{\theta, 0}-\Pi(z)), \mathbf{v} \right \rangle \prec M^{-1/2}.      
\end{equation}
\begin{lemma}\label{lemm_diff}
 For any differentiable function $F:\mathbb R^{\mathcal I_1 \times\mathcal I_2}\rightarrow \mathbb C$, we have that
\begin{equation}\label{basic_interp}
\frac{\dd}{\dd\theta}\mathbb E F(X^\theta)=\sum_{i\in\mathcal I_1}\sum_{\mu\in\mathcal I_2}\left[\mathbb E F\left(X^{\theta,X_{i\mu}^1}_{(i\mu)}\right)-\mathbb E F\left(X^{\theta,X_{i\mu}^0}_{(i\mu)}\right)\right]
\end{equation}
 provided all the expectations exist.
\end{lemma}
\begin{proof}
This is an immediate result from (\ref{law_interpol}) and fundamental theorem of calculus. 
\end{proof}
For any deterministic vector $\bm{v} \in \mathbb{R}^N,$ we denote its natural embedding into $\mathbb{R}^{N+M}$ as 
\begin{equation}\label{eq_defnvhat}
\bm{\widehat{v}}:=
\begin{pmatrix}
\bm{v} \\
0
\end{pmatrix} \in \mathbb{R}^{N+M}
\end{equation}
To establish an analogous result of Proposition 5.1 of \cite{Paquette2020}, i.e., Theorem \ref{thm_unversality}, for any fixed integer $r$ and a sequence of deterministic vectors $\bm{q}_k, \bm{p}_k, 1 \leq k \leq r,$ it suffices to set 
\begin{align}\label{eq_defnF}
F\left(X\right)=\Phi(Z_1,\cdots, Z_r),    
\end{align}
where we denote
\begin{equation}\label{eq_defnzk}
Z_k \equiv Z_k(X):=\sqrt{M}\widehat{\bm{q}}_k^*(G(z_k, X)-\Pi(z_k))\widehat{\bm{p}}_k, 1 \leq k \leq r, 
\end{equation}
and $\{z_k\}$ is a sequence of points away from the support of deformed MP law.  In view of Lemma \ref{lemm_diff}, we will need the following lemma. Its proof can be found in Appendix \ref{appenxi_sub_proofkey}.  
\begin{lemma}\label{lem_telecopingsummation}
For some simple smooth positively-oriented contour $\Omega$ which encloses the support of $\varrho,$ and its boundary $\Gamma=\partial \Omega,$ suppose that for some small constant $\tau>0,$
\begin{equation}\label{eq_contourassumption}
\inf_{z=E+\mathrm{i} \eta \in \Gamma} \max \{ \operatorname{dist}(E, \operatorname{supp} \varrho), \eta\}>\tau. 
\end{equation} 
Then there exists some $0<\delta<0.5$ such that for all $\theta \in [0,1],$ we have 
\begin{equation*}
\left| \sum_{i\in\mathcal I_1}\sum_{\mu\in\mathcal I_2}\left[\mathbb E F\left(X^{\theta,X_{i\mu}^1}_{(i\mu)}\right)-\mathbb E F\left(X^{\theta,X_{i\mu}^0}_{(i\mu)}\right)\right] \right| \leq  N^{-\delta}.
\end{equation*}
\end{lemma}

We first show how Lemma \ref{lem_telecopingsummation} implies Proposition \ref{thm_unversality}.

\begin{proof}[\bf Proof of Proposition \ref{thm_unversality}] The proof relies on the trapezoidal rule (see Lemma \ref{lem_trapezoidalrule}) and is similar to the arguments of the proof of \cite[Theorem 5.11]{Paquette2020}. We sketch the proof here for the purpose of completeness.  Without loss of generality, we assume that $\Gamma_j =\Gamma$ for all $j.$ Denote 
\begin{equation*}
\mathcal{Z}_j:=\frac{\sqrt{M}}{2 \pi \ri} \oint_{\Gamma} f_j(z)(c_0(z;\mu_T)-s_{c_N}) \dd z.
\end{equation*}
We use Lemma \ref{lem_trapezoidalrule} to approximate  $\mathcal{Z}_j$ and denote 
\begin{equation*}
\mathcal{Z}_{j,m}=\frac{\sqrt{M}}{2 \pi \ri} \sum_{k=1}^m f_j(z_k)(c_0(z_k; \mu_T)-s_{c_N}) w_j, 
\end{equation*}
where $z_j$ and $w_j$ are defined in (\ref{eq_wjzj}). Consider that
\begin{equation*}
\Delta_{M,m}:=\Phi\left(\mathcal{Z}_1, \cdots, \mathcal{Z}_r \right)-\Phi \left( \mathcal{Z}_{1,m}, \cdots, \mathcal{Z}_{r,m} \right).
\end{equation*}
Denote 
\begin{equation*}
\mathfrak{L}:= \lim \inf_N \sigma_N  \mathbf{1}_{M \leq N} (1-\sqrt{c_N})^2, \ \mathfrak{U}:=\gamma_+.
\end{equation*}
It is easy to see that both $\mathfrak{L}$ and $\mathfrak{U}$ are bounded. Since $\Gamma$ is uniformly bounded away from the support of $\varrho,$ we can choose a small constant $\delta>0$ such that $[\mathfrak{L}-\delta, \gamma_++\delta] \subset \Omega.$ For any given small $\epsilon>0,$ we define a high probability event $\Xi \equiv \Xi(\delta, \epsilon)$ such that the following conditions hold:
 \begin{enumerate}
 \item[(i).] For $z \in \Gamma$ uniformly and any deterministic units $\mathbf{u}, \mathbf{v} \in \mathbb{R}^{N+M}$ 
\begin{equation}\label{eq_locallawrestricted}
\left| \mathbf{u}^* G(z)\mathbf{v} - \mathbf{u}^* \Pi(z) \mathbf{v} \right| \leq M^{-1/2+\epsilon}.
\end{equation} 
 \item[(ii).] For the given $\delta>0,$ when $M$ is large enough  
 \end{enumerate}
\begin{equation}\label{eq_conditionbound}
\lambda_N \geq \mathfrak{L}-\delta, \ \lambda_1 \leq \gamma_++\delta.
\end{equation}
Note that by Lemma \ref{lem_locallaw}, the definition of $\mathfrak{L}$ and Proposition \ref{prop_edgeeigenvalue}, such an event exists. For the sequel, we fix some  realization $  X \in \Xi $ or $Y \in \Xi$ satisfying the above conditions (i) and (ii). Hence, the rest of the proof is purely deterministic.

 Recall Definition \ref{defn_admissible}. Applying Lemma \ref{lem_trapezoidalrule} for $\mathcal{Z}_j-\mathcal{Z}_{j,m}$ with $D=5,$ we obtain that for some constant $C>0$
\begin{equation*}
|\mathcal{Z}_j-\mathcal{Z}_{j,m}| \leq C \sqrt{M} m^{-5},
\end{equation*}
where we used the assumption that $f_j$ is analytic. 
We can choose $m$ such that $\sqrt{M} m^5=\mathrm{o}(1);$ for example, $m=M^{1/9}.$ Consequently, we have that for some constant $C_1>0$
\begin{equation}\label{eq_controlpart}
\left| \Delta_{M,m} \right| \leq C_1M^{-1/18}. 
\end{equation}
Denote 
\begin{equation*}
\widetilde{\mathcal{Z}}_j:=\frac{\sqrt{M}}{2 \pi \ri} \oint_{\Gamma} f_j(z)(c_0(z;\mu_{\widetilde{T}})-s_{c_N}) \dd z, \  \widetilde{\mathcal{Z}}_{j,m}=\frac{\sqrt{M}}{2 \pi \ri} \sum_{k=1}^m f_j(z_k)(c_0(z_k; \mu_{\widetilde{T}})-s_{c_N}) w_j.
\end{equation*}
Using (\ref{eq_controlpart}), an analogous discussion  for $\mu_{\widetilde{T}}$ and triangle inequality,  it suffices to control 
\begin{equation*}
\widetilde{\Delta}_m:=\Phi(\mathcal{Z}_{1,m}, \cdots, \mathcal{Z}_{r,m})-\Phi(\widetilde{\mathcal{Z}}_{1, m}, \cdots, \widetilde{\mathcal{Z}}_{r, m}).
\end{equation*} 
Recall (\ref{eq_defnzk}). We can consider a function $\Psi: \mathbb{R}^m \rightarrow \mathbb{R}$ such that  
\begin{equation*}
\Psi(Z_1, \cdots,  Z_m):= \Phi(\mathcal{Z}_{1,m}, \cdots, \mathcal{Z}_{r,m})=\Phi\left( \sum_{j=1}^m f_1(z_j) \frac{w_j}{2 \pi \ri} Z_j, \cdots, \sum_{j=1}^m f_r(z_j) \frac{w_j}{2 \pi \ri} Z_j \right).
\end{equation*}
In fact, it is easy to see that $\widetilde{\Delta}_m$ can be controlled using  Lemmas \ref{lemm_diff} and \ref{lem_telecopingsummation}, if we can   show that $\Psi(\cdot)$ is admissible with respect to $Z_k, 1 \leq k \leq m,$ in terms of Definition \ref{defn_admissible}. The rest of the proof is devoted to justifying this aspect.
We first prepare some notations. Note that by Chain rule
\begin{equation*}
\partial_{Z_{j_1}, \cdots, Z_{j_q}} \Psi(Z_1, \cdots, Z_m)=\sum_{k_1, k_2, \cdots, k_p=1}^r \partial_{y_{k_1}, \cdots, y_{k_p}} \Phi(y_1, \cdots, y_r) \left( \prod_{p=1}^q W_{k_p, j_p} \right),
\end{equation*}
where $Y_i, 1 \leq i \leq r, $ are defined as 
\begin{equation*}
Y_i:=\sum_{j=1}^m f_i(z_j) \frac{w_j}{2 \pi \ri} Z_j,
\end{equation*}
and $W=(W_{\ell j}) \in \mathbb{R}^{r \times m}$ are denoted by
\begin{equation}\label{eq_entryform}
W_{\ell j}=f_{\ell}(z_j) \frac{w_j}{2 \pi \ri}.
\end{equation}
Recall the definition of $w_j$ as in (\ref{eq_wjzj}).  Using (\ref{eq_locallawrestricted}), we find that there exists some small constant $\epsilon'\equiv \epsilon'(\epsilon)$ such that
\begin{equation}\label{eq_entrybound}
\max_i\{|Y_i|, |Z_i|\} \leq M^{\epsilon'},
\end{equation}
where we used the fact that $\|f_j\|_{\infty}^q<\infty, \ 1 \leq q \leq r. $ Since $\Phi$ is admissible, by Definition \ref{defn_admissible}, we have that for some constant $C_0>0$
\begin{equation*}
\left|\partial_{y_{k_1}, \cdots, y_{k_p}} \Phi(y_1, \cdots, y_r) \right|\leq M^{C_0 \epsilon'}. 
\end{equation*}
Moreover, since $q \leq m$ and $r$ is fixed, we conclude that there exists some constant $C_1$ such that 
\begin{equation*}
\left| \partial_{Z_{j_1}, \cdots, Z_{j_q}} \Psi(Z_1, \cdots, Z_m)\right| \leq M^{C_1 \epsilon'}.
\end{equation*}
Since $\epsilon$ is arbitrary, using (\ref{eq_entrybound}), we see that $\Psi$ is admissible. This completes our proof. 
\end{proof}

%
%
%

%

\appendix
\section{Some algorithms and the deterministic formulae}\label{sec_appa}
In this appendix, we provide the Jacobi matrix Cholesky factorization algorithm, some deterministic formulas and the MINRES algorithm.  

\subsection{Cholesky factorization algorithm} In this subsection, we provide the following algorithm,  Algorithm \ref{a:chol}, which is designed to calculate the Cholesky decomposition for a Jacobi matrix. 

\vspace{.05in}
\noindent\fbox{%
\refstepcounter{alg}
    \parbox{\textwidth}{%
\flushright \boxed{\text{Algorithm~\arabic{alg}: Jacobi matrix Cholesky factorization \label{a:chol}}}
\begin{enumerate}
    \item Suppose $T$ is an $N \times N$ positive-definite Jacobi matrix, set $H = T$
    \item For $k = 1,2,\ldots,N-1$
    \begin{enumerate}
        \item Set $H_{k+1,k+1} = H_{k+1,k+1} - \displaystyle\frac{H_{k+1,k}^2}{H_{kk}}$
        \item Set $H_{k,k+1} = 0$
        \item Set $H_{k:k+1,k} = H_{k:k+1,k}/\sqrt{H_{k,k}}$
    \end{enumerate}
    \item Set $H_{N,N} = \sqrt{H_{N,N}}$
    \item  Return $\varphi(T) = H$
\end{enumerate}
    }%
}
\vspace{.05in}

\subsection{The MINRES algorithm}\label{appendix_minres} In this subsection, we record the MINRES algorithm \cite[Lecture 38]{MR1444820}

\vspace{.1in}
\noindent\fbox{%
\refstepcounter{alg}
    \parbox{\textwidth}{%
\flushright \boxed{\text{Algorithm~\arabic{alg}: MINRES Algorithm \label{a:minres}}}
\begin{enumerate}
    \item Given some threshold $\epsilon>0$ and set $\bm{q}_1 = \bm{b}/\|b\|_2.$
    \item For $k = 1,2,\ldots,n$, $n \leq N$
    \begin{enumerate}
        \item Compute $\displaystyle a_{k-1} = \frac{\bm{r}^*_{k-1} \bm{r}_{k-1}}{\bm{r}^*_{k-1} W \bm{p}_{k-1}}$.
        \item Set $\vec x_k = \vec x_{k-1}+a_{k-1} \bm{p}_{k-1}$.
        \item Form
        \begin{equation*}
        \begin{bmatrix} a_0 & b_0 \\
    b_0 & a_1 & \ddots \\
    & \ddots & \ddots & b_{k-2} \\
    & & b_{k-2} & a_{k-1} \end{bmatrix}
        \end{equation*}
       \item Set $\vec r_k = \vec r_{k-1}-a_{k-1} W \bm{p}_{k-1}$.
        \item Compute $\displaystyle b_{k-1} = -\frac{\bm{r}^*_{k-1} \bm{r}_{k-1}}{\bm{r}^*_{k-1}  \bm{r}_{k-1}}$.
        \item Set $\vec p_k=\vec r_k-b_{k-1} \vec p_{k-1}.$
    \end{enumerate}
\end{enumerate}
    }%
}

\vspace{.1in}

\subsection{Deterministic formulae}
In this subsection, we provide some deterministic formulas for the numerical algorithms. 

\begin{lemma}[Deterministic formulae]\label{t:deterministic}
Consider the Lanczos iteration applied to the pair $(W,\vec b)$ with $W > 0$ and $\| \vec b \|_2 = 1$. Suppose the iteration runs until step $n \leq N$, $\vec r_n = 0$, producing a tridiagonal matrix $T = T_n(W,\vec b)$.   Let $T = HH^T$ be the Cholesky factorization (see Algorithm~\ref{a:chol} below) of $T$ where
\begin{align*}
    H = \begin{bmatrix} \alpha_0 \\
   \beta_0 & \alpha_1 \\
   & \beta_1 & \alpha_2 \\
   && \ddots & \ddots\\
   &&& \beta_{n-2} & \alpha_{n-1}
   \end{bmatrix}.
\end{align*}
Then
for the CGA  on $W \vec x = \vec b$ with $\vec x_0 = 0$, for $k \leq N$,
    \begin{align}\label{eq_detone}
        \|\vec r_k\|_2 &= \prod_{j=0}^{k-1} \frac{\beta_j}{\alpha_j}.
    \end{align}    
Moreover, we have that 
\begin{equation}\label{eq_dettwo}
\|\vec e_k\|_W=\int \frac{1}{\lambda} \dd \mu_{Z, \bm{b}}(\lambda) -\frac{1}{\alpha_0^2} \sum_{\ell=0}^{k-1} \prod_{j=1}^\ell \frac{\beta_{j-1}^2}{\alpha_j^2},
\end{equation}
%
%
or equivalently,  
    \begin{align}\label{eq_detthree}
        \|\vec e_k\|_W &= \|\vec r_k\|_2\sqrt{\vec f_1^* (L_{k} L_{k}^T)^{-1} \vec f_1}, \quad L_{k} = H_{k+1:N,k+1:N}.
    \end{align}
 For the MINRES algorithm on $W \vec x = \vec b$, for $k < n$,
    \begin{align}\label{eq_detfour}
      \|\vec r_k\|_2 = \left(1 + \displaystyle\sum_{j=1}^k \prod_{\ell = 0}^{j-1} \frac{\alpha_\ell^2}{\beta_\ell^2} \right)^{-1/2}.
    \end{align}   
\end{lemma}
\begin{proof}
(\ref{eq_detone}) and (\ref{eq_detfour}) follows from Propositions 4.1, 4.2 and the calculations of Section 6 of \cite{Paquette2020}. (\ref{eq_dettwo}) and (\ref{eq_detthree}) can be obtained with slightly modification using the calculation below (34) of \cite{Paquette2020}. 
\end{proof}

\section{Additional technical proofs}\label{sec_appb}

\subsection{Proofs of Lemmas \ref{lem_locallaw} and \ref{lem_connectionspikednonspiked}}\label{sec_appendix_technical}

\begin{proof}[\bf Proof of Lemma \ref{lem_locallaw}]
The results have essentially been proved in \cite{Knowles2017} with slightly different assumptions, we only point out how to conform our setting to that of \cite{Knowles2017}.

First, in \cite[Definition 3.2]{Knowles2017}, the linearizing block matrix is defined as 
 \begin{equation}\label{linearize_block1}
   H_{0}: = \left( {\begin{array}{*{20}c}
   { -\Sigma_0^{-1}} & X  \\
   X^{*} & {-zI}  \\
   \end{array}} \right).
 \end{equation}
 It is easy to check the following relation between (\ref{eq_defnh}) and (\ref{linearize_block1})
 \begin{equation} \label{linearblockrelation}
 H=\left( {\begin{array}{*{20}c}
   { z^{1/2}\Sigma_0^{1/2}} & 0  \\
   0 & {I}  \\
   \end{array}} \right) H_0 \left( {\begin{array}{*{20}c}
   { z^{1/2}\Sigma_0^{1/2}} & 0  \\
   0 & {I}  \\
   \end{array}} \right).
\end{equation} 
In \cite{Knowles2017}, the deterministic convergent limit of $H_0^{-1}$ is
\begin{equation} \label{convergentlimit1}
 \Pi_0(z)=\left( {\begin{array}{*{20}c}
   { -\Sigma_0 (1+m(z)\Sigma_0)^{-1}}& 0  \\
   0 & {m(z)}  \\
   \end{array}} \right).
\end{equation} 
Therefore, by (\ref{linearblockrelation}), we can get a similar relation between (\ref{eq_defnpi}) and (\ref{convergentlimit1})
\begin{equation}\label{pirelationship}
\Pi(z)=\left( {\begin{array}{*{20}c}
   { z^{-1/2}\Sigma_0^{-1/2}} & 0  \\
   0 & {I}  \\
   \end{array}} \right) \Pi_0(z) \left( {\begin{array}{*{20}c}
   { z^{-1/2}\Sigma_0^{-1/2}} & 0  \\
   0 & {I}  \\
   \end{array}} \right).
\end{equation}

Second, when $\operatorname{dist}(E, \operatorname{supp}(\varrho))\geq \tau,$ the results have been established for $(H_0^{-1}, \Pi_0)$ in \cite[Theorem 3.16]{Knowles2017} Since $|z|<\infty,$ together with (3) of Assumption \ref{assum_summary}, we can conclude that, the results should also hold for $(G, \Pi).$ Moreover, when $\eta \geq \tau,$ it is easy to see that for some constant $C>0,$
\begin{equation*}
\Im m(z)=\int \frac{\eta}{(x-E)^2+\eta^2} \varrho(x) \dd x \geq C \eta. 
\end{equation*}  
Consequently,  when $\eta \geq  \tau ,$ we have that for some constant $c>0$ 
\begin{equation}\label{eq_assumption}
\inf_{z \in \mathcal{D}_o}\min_i |1+m(z) \sigma_i | \geq c.
\end{equation}
According to (3.20) of \cite{Knowles2017}, once (\ref{eq_assumption}) holds, under (1)--(3) of Assumption \ref{assum_summary}, the results for $(H_0^{-1}, \Pi_0)$ can be obtained as stated in \cite[Theorem 3.6]{Knowles2017}. This completes the proof using $|z|<\infty$ and (3) of Assumption \ref{assum_summary}. 
\end{proof}

\begin{proof}[\bf Proof of Lemma \ref{lem_connectionspikednonspiked}]

We start with (\ref{eq_leftpertub}). For $\bm{v}_i, 1 \leq i \leq N,$ multiplying it on both sides of (\ref{eq_expansionformula}) yields that
\begin{equation}\label{eq_expansioncorrect}
\bm{v}_i^* \widetilde{G}_1(z) \bm{v}_i= \frac{\sigma_i}{\widetilde{\sigma}_i}\left( \bm{v}_i^* G_1(z) \bm{v}_i-z\bm{v}_i^* G_1 \Vb_r \left(\Db^{-1}+1+z\Vb_r^* G_1 \Vb_r \right)^{-1} \Vb_r^* G_1 \bm{v}_i \right).
\end{equation}
First, when $\bm{v}=\bm{v}_i,  i > r,$ since $\bm{v}_i^* \bm{v}_j=0, \ \bm{v}_j \in \Vb_r,$ by Lemma \ref{lem_locallaw}, we conclude that  for $z \in \widetilde{\mathcal{D}}_o,$
\begin{equation*}
\bm{v}_i^* \widetilde{G}_1(z) \bm{v}_i= \bm{v}_i^* G_1(z) \bm{v}_i+\OO_{\prec}(M^{-1/2}),
\end{equation*}
where we used the fact that $\sigma_i=\widetilde{\sigma}_i, i>r.$ Second, when $\bm{v}=\bm{v}_i, i \leq r,$ we obtain that 
\begin{align*}
\bm{v}_i^* \widetilde{G}_1(z) \bm{v}_i = \frac{1}{1+d_i} \left( \bm{v}_i^* G_1(z) \bm{v}_i -\mathcal{L}_i \right) +\OO_{\prec}(M^{-1/2}), 
\end{align*}
where we used Lemma \ref{lem_locallaw}.  
This completes our proof of (\ref{eq_leftpertub}) using the expansion $\bm{v}=\sum_{i=1}^N w_i \bm{v}_i.$

For (\ref{eq_rightpertub}), let $\Delta(z)=G(z)-\Pi(z),$ by Lemma \ref{lem_generalnonpertubed}, we have that 
\begin{align*}
{\bm{u}}^* \widetilde{G}_2(z) {\bm{u}}= \bm{u}^* G_2(z) \bm{u}+z \widetilde{\bm{u}}^* \Pi(z) \widehat{\Vb}_r \left(\Db^{-1}+1+z\widehat{\Vb}_r^* G(z) \widehat{\Vb}_r \right)^{-1} \widehat{\Vb}_r^* G(z) \widetilde{\bm{u}}
\\- z \widetilde{\bm{u}}^* \Delta(z) \widehat{\Vb}_r \left(\Db^{-1}+1+z\widehat{\Vb}_r^* G(z) \widehat{\Vb}_r \right)^{-1} \widehat{\Vb}_r^* G(z) \widetilde{\bm{u}}.
\end{align*}
Using the structure of (\ref{eq_defnpi}), (\ref{eq_embeddingone}) and (\ref{eq_embeddingtwo}), for the first term, we have that
\begin{equation*}
z \widetilde{\bm{u}}^* \Pi(z) \widehat{\Vb}_r \left(\Db^{-1}+1+z\widehat{\Vb}_r^* G(z) \widehat{\Vb}_r \right)^{-1} \widehat{\Vb}_r^* G(z) \widetilde{\bm{u}}=0.
\end{equation*}
By Lemma \ref{lem_locallaw}, we have that 
\begin{equation*}
\left \| \widetilde{\bm{u}}^* \Delta(z) \widehat{\Vb}_r \right \|=\OO_{\prec}(M^{-1/2}), \ \left\| \widehat{\Vb}_r^* G(z) \widetilde{\bm{u}} \right\|=\OO_{\prec}(1), 
\end{equation*}
and for some constant $C>0,$
\begin{align*}
\left\| \left(\Db^{-1}+1+z\widehat{\Vb}_r^* G(z) \widehat{\Vb}_r \right)^{-1} \right\| \leq \frac{C}{\tau-\Psi(z)}+\OO_{\prec}(M^{-1/2}),
\end{align*}
where we used the definition $\widetilde{\mathcal{D}}_o$ in (\ref{eq_refineset}). This completes our proof for  (\ref{eq_rightpertub}). 

\end{proof}

\subsection{Proof of Lemma \ref{lem_telecopingsummation}} \label{appenxi_sub_proofkey}
In this subsection, we proceed to the proof of Lemma \ref{lem_telecopingsummation}. Its proof relies on the following decomposition, which is an analog of Lemma 5.15 of \cite{Paquette2020}. Define
\begin{equation*}
S(X) \equiv S(z, X):= \sqrt{M} \left(G(z, X)-\Pi(z) \right).   
\end{equation*}
We use the shorthand notation $S(X) \equiv S(z,X)$ if there is no confusion on the spectral parameter.

For each pair $(i, \mu),$ since $\Sigma_0^{1/2} \bm{f}_i \bm{f}_\mu^*$ is a rank one matrix, we write
\begin{equation*}
\Sigma_0^{1/2} \bm{f}_i \bm{f}_\mu^*=\ell \bm{\xi} \bm{\zeta}^*.
\end{equation*}
Note that $\ell<\infty.$ Recall (\ref{eq_defndelta}). Note that 
\begin{equation}\label{eq_ranktwomatrix}
\begin{pmatrix}
0 & \Sigma_0^{1/2} \bm{f}_i \bm{f}_{\mu}^* \\
\bm{f}_{\mu} \bm{f}_i^* \Sigma_0^{1/2}& 0
\end{pmatrix}=\mathbf{U} \mathbf{D} \mathbf{U}^*,    
\end{equation}
where $\mathbf{D} \in \mathbb{R}^{2 \times 2}$ and $\mathbf{U} \in \mathbb{R}^{(N+M) \times 2}$ are defined as 
\begin{equation*}
\mathbf{D}:=
\begin{pmatrix}
0 & \ell \\
\ell & 0
\end{pmatrix}, \ 
\mathbf{U}:=
\begin{pmatrix}
\bm{\xi} & 0 \\
0 & \bm{\zeta} 
\end{pmatrix}.
\end{equation*}
\begin{lemma}\label{lem_firstordercorrection}
For $\beta=0,1,$ any deterministic unit vectors $\bm{u}, \bm{v} \in \mathbb{R}^N$ and any spectral parameter $z \in \mathcal{D}(z, \tau)$ in (\ref{eq_setmathcald}),  we have 
\begin{equation*}
\widehat{u}^*S(X_{i \mu}^{\theta, X_{i \mu}^\beta}) \widehat{\bm{v}}=\widehat{u}^*S(X_{i \mu}^{\theta, 0}) \widehat{\bm{v}}+J_0+\sum_{k=1}^4 M^{-k/2} J_k+\OO_{\prec}\left( M^{-5/2} \right).
\end{equation*}
where $J_0$ is defined as 
\begin{align*}
J_0:=\sqrt{M \eta} \sum_{k \in \{2, 4\}} (-\sqrt{z} X_{i \mu}^{\beta})^k \mathfrak{s}_k,
\end{align*}
and $\mathfrak{s}_k$ is independent of $\beta$ and defined as
\begin{equation}\label{eq_defnsk}
\mathsf{s}_k:=\widehat{\bm{u}}^* \Pi \left(  \mathbf{U} \mathbf{D} \mathbf{U}^* \Pi \right)^k  \widehat{\bm{v}}. 
\end{equation}
and $J_k, 1 \leq k \leq 4, $ has the following form
\begin{equation*}
J_k=(-\sqrt{M} X_{i \mu}^{\beta})^k g_k,
\end{equation*}
where $g_k$ only depends on $X_{(i \mu)}^{\theta, 0},$ i.e., independent of $X_{i \mu}^{\beta}$ satisfying that 
\begin{equation*}
g_k=\OO_{\prec}(1). 
\end{equation*}
\end{lemma}
\begin{proof}
Using (\ref{eq_resolventexpansion}) with $K=4,$ we obtain 
\begin{align}\label{eq_realdecomposition}
S\left(X_{(i \mu)}^{\theta, X_{i \mu}^{\beta}}\right)= S\left(X_{(i \mu)}^{\theta, 0}\right)&+ \sqrt{M} \sum_{k=1}^4 G_{(i \mu)}^{\theta, 0}  \left( \Delta_{(i \mu)}^{-X_{i \mu}^{\beta}} G_{(i \mu)}^{\theta,0} \right)^k \\
&+\sqrt{M} G_{(i \mu)}^{\theta, X_{i \mu}^{\beta}}\left( \Delta_{(i \mu)}^{-X_{i \mu}^{\beta}} G_{(i \mu)}^{\theta,0} \right)^{5}    \nonumber
\end{align}
We now consider the terms on the right-hand side of (\ref{eq_realdecomposition}). When $k=1,$ using (\ref{eq_defndelta}) and (\ref{eq_ranktwomatrix}), we have that 
\begin{equation*}
\sqrt{M} \widehat{\bm{u}}^* G_{(i \mu)}^{\theta, 0}  \Delta_{(i \mu)}^{-X_{i \mu}^{\beta}} G_{(i \mu)}^{\theta,0} \widehat{\bm{v}}=-\sqrt{M}  X_{i \mu}^\beta \widehat{\bm{u}}^* G_{(i \mu)}^{\theta, 0} \mathbf{U} \mathbf{D} \mathbf{U}^* G_{(i \mu)}^{\theta, 0} \widehat{\bm{v}}.
\end{equation*}
By construction of (\ref{eq_construction}),  we have that $\widehat{\bm{u}}^* G_{(i \mu)}^{\theta, 0} \mathbf{U} \mathbf{D} \mathbf{U}^* G_{(i \mu)}^{\theta, 0} \widehat{\bm{v}}$ is independent of $X_{i \mu}.$
We decompose  
\begin{align}\label{eq_basicexpansion}
    \widehat{\bm{u}}^* G_{(i \mu)}^{\theta, 0} \mathbf{U} \mathbf{D} \mathbf{U}^* G_{(i \mu)}^{\theta, 0} \widehat{\bm{v}}=  \widehat{\bm{u}}^* \Pi \mathbf{U} \mathbf{D} \mathbf{U}^* \Pi \widehat{\bm{v}} + \mathcal{E}_{i \mu, 1}, 
\end{align}
where 
\begin{align*}
\mathcal{E}_{i \mu,1}:= M^{-1/2}\widehat{\bm{u}}^* \left[ S(X_{(i \mu)}^{\theta, 0}) \mathbf{U} \mathbf{D} \mathbf{U}^*   G_{(i \mu)}^{\theta, 0}+ \Pi \mathbf{U} \mathbf{D} \mathbf{U}^* S(X_{(i \mu)}^{\theta, 0}) \right] \widehat{\bm{v}}. 
\end{align*}
Since $X_{(i \mu)}^{\theta, 0}$ is independent of $X_{i \mu}^{\beta},$ we can see that $\mathcal{E}_{i \mu,1}$ is independent of $X_{i \mu}^{\beta}.$ We proceed to the analysis of (\ref{eq_basicexpansion}). First, invoking the structure of (\ref{eq_defnpi}) and (\ref{eq_defnvhat}), we find that 
\begin{align}\label{eq_meanzeroornot}
\widehat{\bm{u}}^* \Pi \mathbf{U} \mathbf{D} \mathbf{U}^* \Pi \widehat{\bm{v}}& =\left( \bm{u}^* \Pi_1(z), 0 \right)\mathbf{U} \mathbf{D} \begin{pmatrix}
\bm{\xi}^* \Pi_1 \bm{v} \\
0
\end{pmatrix}  \nonumber \\
& =(\bm{u} \Pi_1(z) \bm{\xi},0) \mathbf{D} \begin{pmatrix}
\bm{\xi}^* \Pi_1 \bm{v} \\
0
\end{pmatrix} \nonumber \\
& = 0. 
\end{align}
Second, by Lemma \ref{lem_locallaw} and the fact $\ell<\infty$, we have that
\begin{equation*}
\mathcal{E}_{i \mu,1}=\OO_{\prec}(M^{-1/2}).     
\end{equation*}
Combining the above arguments, it is easy to see that we have that $$ J_1= \left(-\sqrt{M} X^{\beta}_{i \mu} \right) g_1, \  g_1:= \sqrt{M} \sqrt{z} \mathcal{E}_{i \mu,1}.$$

The other terms when $k=2,3,4,$ can be analyzed in a similar fashion. We only point out the differences. In particular, on one hand, by an argument similar to (\ref{eq_meanzeroornot}), we have that 
\begin{equation}\label{eq_finalconlusionofd}
 \widehat{\bm{u}}^* \Pi (\mathbf{U} \mathbf{D} \mathbf{U}^* \Pi )^k \widehat{\bm{v}}=0, \ k \ \text{is an odd integer}. 
\end{equation}
Consequently, for $k=2, 4,$ we collect these two terms as $J_0.$ On the other hand, we define
\begin{equation*}
J_k=\left( -\sqrt{M} X_{i \mu}^{\beta} \right)^k  g_k, \ g_k:=z^{k/2} \sqrt{M} \mathcal{E}_{i \mu, k}, \ k=2,3,4,
\end{equation*}
where $\mathcal{E}_{i \mu, k},  2 \leq k \leq 4,$ are defined as 
\begin{align*}
\mathcal{E}_{i \mu, 2}:&=M^{-1/2} \widehat{\bm{u}}^* \left[  S\left(X_{(i \mu)}^{\theta, 0}\right) \left(\mathbf{U} \mathbf{D} \mathbf{U}^*   G_{(i \mu)}^{\theta, 0}\right)^2+ \Pi  \mathbf{U} \mathbf{D} \mathbf{U}^*   S\left(X_{(i \mu)}^{\theta, 0}\right) \mathbf{U} \mathbf{D} \mathbf{U}^*   G_{(i \mu)}^{\theta, 0} \right. \\
&+ \left. \left( \Pi  \mathbf{U} \mathbf{D} \mathbf{U}^* \right)^2 S\left(X_{(i \mu)}^{\theta, 0}\right) \right] \widehat{\bm{v}},
\end{align*}
\begin{align*}
\mathcal{E}_{i \mu, 3}:&=M^{-1/2}\widehat{\bm{u}}^* \left[  S\left(X_{(i \mu)}^{\theta, 0}\right) \left(\mathbf{U} \mathbf{D} \mathbf{U}^*   G_{(i \mu)}^{\theta, 0}\right)^3+ \Pi  \mathbf{U} \mathbf{D} \mathbf{U}^*   S\left(X_{(i \mu)}^{\theta, 0}\right) \left( \mathbf{U} \mathbf{D} \mathbf{U}^*   G_{(i \mu)}^{\theta, 0} \right)^2 \right. \\
&+ \left. \left( \Pi  \mathbf{U} \mathbf{D} \mathbf{U}^*\right)^2 S\left(X_{(i \mu)}^{\theta, 0} \right)  \mathbf{U} \mathbf{D} \mathbf{U}^*G_{(i \mu)}^{\theta, 0}+ \left( \Pi  \mathbf{U} \mathbf{D} \mathbf{U}^*\right)^3 S\left(X_{(i \mu)}^{\theta, 0} \right)  \right] \widehat{\bm{v}},
\end{align*}
\begin{align*}
\mathcal{E}_{i \mu, 4}:&=M^{-1/2}\widehat{\bm{u}}^* \left[  S\left(X_{(i \mu)}^{\theta, 0}\right) \left(\mathbf{U} \mathbf{D} \mathbf{U}^*   G_{(i \mu)}^{\theta, 0}\right)^4+ \Pi  \mathbf{U} \mathbf{D} \mathbf{U}^*   S\left(X_{(i \mu)}^{\theta, 0}\right) \left( \mathbf{U} \mathbf{D} \mathbf{U}^*   G_{(i \mu)}^{\theta, 0} \right)^3 \right. \\
&+ \left. \left( \Pi  \mathbf{U} \mathbf{D} \mathbf{U}^*\right)^2 S\left(X_{(i \mu)}^{\theta, 0} \right)  \left(\mathbf{U} \mathbf{D} \mathbf{U}^*G_{(i \mu)}^{\theta, 0} \right)^2+ \left( \Pi  \mathbf{U} \mathbf{D} \mathbf{U}^*\right)^3 S\left(X_{(i \mu)}^{\theta, 0} \right) \mathbf{U} \mathbf{D} \mathbf{U}^*G_{(i \mu)}^{\theta, 0} \right. \\
&+\left. \left( \Pi  \mathbf{U} \mathbf{D} \mathbf{U}^* \right)^4 S\left(X_{(i \mu)}^{\theta, 0}\right) \right] \widehat{\bm{v}}.
\end{align*}
Moreover, it is easy to see from Lemma \ref{lem_locallaw} that 
\begin{equation*}
\mathcal{E}_{i \mu, k}=\OO_{\prec}(M^{-1/2}), \ 2 \leq k \leq 4. 
\end{equation*}

Finally, for $k=5,$ by a discussion similar to (\ref{eq_reducebound}), we have that
\begin{equation}\label{eq_newbound}
\left \langle \widehat{\bm{u}} (G_{(i \mu)}^{\theta, X_{i \mu}^\beta}-\Pi(z)), \widehat{\bm{v}}  \right \rangle \prec M^{-1/2}.      
\end{equation}
Since $k$ is odd, using (\ref{eq_finalconlusionofd}) and a discussion similar to (\ref{eq_basicexpansion}), together with (\ref{eq_newbound}),  we obtain that
\begin{equation*}
 \widehat{\bm{u}}^* G_{(i \mu)}^{\theta, X_{i \mu}^{\beta}}\left(\mathbf{U} \mathbf{D} \mathbf{U}^* G_{(i \mu)}^{\theta, 0}\right)^5 \widehat{\bm{v}}=\mathcal{E}_{i \mu, 5},
\end{equation*}
where $\mathcal{E}_{i \mu, 5}$ is defined similarly as $\mathcal{E}_{i \mu, k}, 1 \leq k \leq 4,$ and satisfies $\mathcal{E}_{i \mu,5}=\OO_{\prec}(M^{-1/2}).$ Consequently, by Assumption \ref{assum_summary}, we conclude  
\begin{equation*}
\sqrt{M} \widehat{\bm{u}}^* G_{(i \mu)}^{\theta, X_{i \mu}^{\beta}}\left( \Delta_{(i \mu)}^{-X_{i \mu}^{\beta}} G_{(i \mu)}^{\theta,0} \right)^{5}  \widehat{\bm{v}}=\OO_{\prec}(M^{-5/2}).
\end{equation*}
This completes our proof. 
\end{proof}
Armed with the above lemma, we proceed to the proof of Lemma \ref{lem_telecopingsummation}.
\begin{proof}[\bf Proof of Lemma \ref{lem_telecopingsummation}] 
We claim that, for $\beta=0,1,$ any $\theta \in [0,1]$ and some small constant $\epsilon>0,$ the following holds
\begin{equation}\label{eq_claimfinal}
\left|\mathbb{E} F\left(X_{(i \mu)}^{\theta, X_{i \mu}^\beta}\right)-\mathbb{E} F\left(X_{(i \mu)}^{\theta,0}\right)- \mathcal{J} \right| \leq M^{-5/2+\epsilon},
\end{equation}
where $\mathcal{J}$ only depends on $X_{(i \mu)}^{\theta, 0},$ $\mathfrak{s}_k, k=2,4,$ defined in (\ref{eq_defnsk}) and the moments of $X_{i \mu}^\beta$ up to order of four. (\ref{eq_claimfinal}) implies Lemma \ref{lem_telecopingsummation}. In fact, since $\mathfrak{s}_k, k=2,4,$ are independent of $\beta=0,1,$ by (\ref{assum_momentmatching}) and (\ref{eq_claimfinal}), we readily obtain that 
\begin{equation*}
\left| \mathbb{E} F \left( X_{(i \mu)}^{\theta, X_{i \mu}^1} \right)-\mathbb{E} F \left( X_{(i \mu)}^{\theta, X_{i \mu}^0} \right) \right| \leq M^{-5/2+\epsilon}.
\end{equation*}
This completes the proof of Lemma \ref{lem_telecopingsummation}. 

The following arguments now lead to the proof of (\ref{eq_claimfinal}). These arguments are similar to those in Proposition 5.16 of \cite{Paquette2020} utilizing Lemma \ref{lem_firstordercorrection}  and we only point out the main differences.
Denote $\bm{\gamma}=(\gamma_1, \cdots, \gamma_r)$ such that
\begin{equation}\label{eq_gammadecomposition}
\gamma_i=J_{0,i}+\sum_{k=1}^4 M^{-k/2} J_{k,i}+\OO_{\prec}(M^{-5/2}),
\end{equation}
where this represents the term in Lemma \ref{lem_firstordercorrection} applied to $\widehat{\bm{q}}_i,\widehat{\bm{p}}_i, z_i$ and $X_{(i \mu)}^{\theta,0}.$ Applying a fifth order  Taylor expansion to  $F$ defined in (\ref{eq_defnF}), using the conventions (\ref{eq_convention1}) and (\ref{eq_convention2}), we have that for $\beta=0,1,$ 
\begin{align*}
F\left(X_{(i \mu)}^{\theta, X_{i \mu}^\beta}\right)& =\Phi\left(Z_1\left(X_{(i \mu)}^{\theta,0}\right)+\gamma_1, \cdots, Z_r\left(X_{(i \mu)}^{\theta, 0}\right)+\gamma_r\right) \\
& =F\left(X_{(i \mu)}^{\theta,0}\right)+\sum_{k=1}^4 \sum_{|\bm{\alpha}|=k} \frac{\partial^{\bm{\alpha}} \Phi\left(Z_1\left(X_{(i \mu)}^{\theta,0}\right), \cdots, Z_r\left(X_{(i \mu)}^{\theta, 0}\right)\right)}{\bm{\alpha}!}  \bm{\gamma}^{\bm{\alpha}} \\
&+\sum_{|\bm{\alpha}|=5}\frac{\partial^{\bm{\alpha}}\Phi \left(Z_1\left(X_{(i \mu)}^{\theta,0}+h\gamma_1\right), \cdots, Z_r\left(X_{(i \mu)}^{\theta, 0}+h \gamma_r\right)\right)}{\bm{\alpha}!}  \bm{\gamma}^{\bm{\alpha}},
\end{align*}
for some constant $0 \leq h \leq 1. $ Here $\bm{\alpha}  \in \mathbb{R}^m$ contains nonnegative integers. We first handle the error term when $|\bm{\alpha}|=5.$ Recall the definitions of $J_0$ and $J_k$ in Lemma \ref{lem_firstordercorrection}. We readily conclude that for all $1 \leq i \leq r$
\begin{equation*}
J_{0,i}=\OO_{\prec}(M^{-1/2}), \ J_{k,i}=\OO_{\prec}(1).
\end{equation*}
Consequently, according to (\ref{eq_gammadecomposition}), we find that 
\begin{equation*}
\sum_{|\bm{\alpha}|=5}\frac{\partial^{\bm{\alpha}}\Phi \left(Z_1\left(X_{(i \mu)}^{\theta,0}+h\gamma_1\right), \cdots, Z_r\left(X_{(i \mu)}^{\theta, 0}+h \gamma_r\right)\right)}{\bm{\alpha}!}  \bm{\gamma}^{\bm{\alpha}}=\OO_{\prec}(M^{-5/2}). 
\end{equation*}
Next, we can set
\begin{equation*}
\mathcal{J}:=\sum_{k=1}^4 \sum_{|\bm{\alpha}|=k} \frac{\partial^{\bm{\alpha}} \Phi\left(Z_1\left(X_{(i \mu)}^{\theta,0}\right), \cdots, Z_r\left(X_{(i \mu)}^{\theta, 0}\right)\right)}{\bm{\alpha}!}  \bm{\widehat{\gamma}}^{\bm{\alpha}},
\end{equation*}
where $\bm{\widehat{\gamma}}=(\widehat{\gamma}_{1}, \cdots, \widehat{\gamma}_r)$ and 
\begin{equation*}
\widehat{\gamma}_i=J_{0,i}+\sum_{k=1}^4 M^{-k/2} J_{k,i}, \ 1 \leq i \leq r. 
\end{equation*} 
It is clear that $\mathcal{J}$ only depends on  $X_{(i \mu)}^{\theta, 0},$ $\mathfrak{s}_k, k=2,4,$ defined in (\ref{eq_defnsk}) and the moments of $X_{i \mu}^\beta$ up to order of four. Moreover, by (\ref{eq_gammadecomposition}), we conclude that 
\begin{equation*}
\sum_{k=1}^4 \sum_{|\bm{\alpha}|=k} \frac{\partial^{\bm{\alpha}} \Phi\left(Z_1\left(X_{(i \mu)}^{\theta,0}\right), \cdots, Z_r\left(X_{(i \mu)}^{\theta, 0}\right)\right)}{\bm{\alpha}!}  \bm{\gamma}^{\bm{\alpha}}=\mathcal{J}+\OO_{\prec}(M^{-5/2}).
\end{equation*}
This concludes the proof of (\ref{eq_claimfinal}) and hence Lemma \ref{lem_telecopingsummation}.

\end{proof}

\section{Some extra lemmas}\label{sec_appc}

\subsection{Some technical lemmas}\label{appendix_techinicallemma}
In this subsection, we prove some  lemmas. These lemmas provide  key connections between the VESDs of the spiked and non-spiked sample covariance matrices.

\begin{lemma} Let $\Vb_r$ be the collection of the first $r$ spiked eigenvectors of $\Sigma$ and $\Db=\operatorname{diag}\{d_1, \cdots, d_r\}.$  We have that  
\begin{align}\label{eq_expansionformula}
\Sigma_0^{-1/2} \Sigma^{1/2} & \widetilde{G}_1(z) \Sigma^{1/2} \Sigma_0^{-1/2} \nonumber \\
&=-z G_1(z)\Vb_r \left(\Db^{-1}+1+z\Vb_r^* G_1(z) \Vb_r \right)^{-1} \Vb_r^* G_1(z)   +G_1(z).
\end{align}
\end{lemma}
\begin{proof}
Note that
\begin{align}
\Sigma_0^{-1/2} \Sigma^{1/2} \widetilde{G}_1(z) \Sigma^{1/2} \Sigma_0^{-1/2} &=\Sigma_0^{-1/2} \left( XX^\top-z \Sigma^{-1} \right)^{-1}  \Sigma_0^{-1/2} \nonumber \\
&=\left( W_0-z+z-z \Sigma_0^{1/2} \Sigma^{-1} \Sigma_0^{1/2} \right)^{-1} \nonumber \\
& =\left( [G_1(z)]^{-1}+z \Vb_r \Db (1+\Db)^{-1} \Vb^*_r \right)^{-1}. \label{spikedgeneralexpansion}
\end{align}
Using the Woodbury's identity 
\begin{equation}\label{eq_woodburyidentity}
(A+SBT)^{-1}=A^{-1}-A^{-1}S(B^{-1}+TA^{-1}S)^{-1}TA^{-1},
\end{equation}
we have that
\begin{align*}
& \left( [G_1(z)]^{-1}  +z \Vb_r \Db (1+\Db)^{-1} \Vb_r^* \right)^{-1} \\
&=G_1(z)-z G_1(z)\Vb_r \left(\Db^{-1}+1+z\Vb_r^* G_1(z) \Vb_r \right)^{-1} \Vb_r^* G_1(z).
\end{align*}
This completes our proof. 
\end{proof}

The second lemma provides the connection of the VESDs of the right singular vectors of the spiked and non-spiked covariance matrices.  
\begin{lemma}\label{lem_generalnonpertubed} For any deterministic vector $\bm{u} \in \mathbb{R}^M,$ denote $\widetilde{\bm{u}} \in \mathbb{R}^{N+M}$ as the natural embedding of $\bm{u}$ such that
\begin{equation}\label{eq_embeddingone}
\widetilde{\bm{u}}=
\begin{pmatrix}
0  \\
 \bm{u} 
\end{pmatrix}.
\end{equation}
Moreover, denote $\widehat{\Vb}_r \in \mathbb{R}^{(N+M) \times r}$ as the natural embedding of $\Vb_r$ such that that
\begin{equation}\label{eq_embeddingtwo}
\widehat{\Vb}_r=
\begin{pmatrix}
\Vb_r  \\
 0 
\end{pmatrix}.
\end{equation}
Then we have that 
\begin{equation*}
{\bm{u}}^* \widetilde{G}_2 {\bm{u}}= \bm{u}^* G_2(z) \bm{u}- z \widetilde{\bm{u}}^*G(z) \widehat{\Vb}_r \left(\Db^{-1}+1+z\widehat{\Vb}_r^* G(z) \widehat{\Vb}_r \right)^{-1} \widehat{\Vb}_r^* G(z) \widetilde{\bm{u}}.
\end{equation*}
\end{lemma}

\begin{proof}
Recall (\ref{eq_defnh}) and (\ref{eq_defnG}). We define the analogous quantities for the spiked model as 
\begin{equation*}
\widetilde{H} \equiv \widetilde{H}(z,X):=\sqrt{z}
\begin{pmatrix}
0 &   \Sigma^{1/2} X \\
 X^* \Sigma^{1/2} & 0
\end{pmatrix},
\end{equation*}
and $\widetilde{G}(z)=(\widetilde{H}-z)^{-1}.$ Denote $\widehat{\Sigma}_0 \in \mathbb{R}^{N+M}$ as 
\begin{equation*}
\widehat{\Sigma}_0:=
\begin{pmatrix}
\Sigma^{1/2}_0 &   0 \\
 0 & I
\end{pmatrix}.
\end{equation*}
Similarly, we can define $\widehat{\Sigma}.$ 
With a discussion similar to (\ref{spikedgeneralexpansion}), we find that 
\begin{align*}
\widehat{\Sigma}_0^{-1} \widehat{\Sigma} \widetilde{G} \widehat{\Sigma}  \widehat{\Sigma}_0^{-1}=\left(\left[ G(z) \right]^{-1}+z \widehat{\Vb}_r \Db(1+\Db)^{-1} \widehat{\Vb}_r^* \right)^{-1}
\end{align*}
Then by the Woodbury's identity (\ref{eq_woodburyidentity}), we have that 
\begin{equation}\label{eq_expressionafterwood}
\widehat{\Sigma}_0^{-1} \widehat{\Sigma} \widetilde{G} \widehat{\Sigma}  \widehat{\Sigma}_0^{-1}=G(z)-zG(z) \widehat{\Vb}_r \left(\Db^{-1}+1+z\widehat{\Vb}_r^* G(z) \widehat{\Vb}_r \right)^{-1} \widehat{\Vb}_r^* G(z)
\end{equation}
Recall (\ref{eq_defnG}). Similar expression holds for $\widetilde{G}.$  We have that 
\begin{equation*}
\bm{u}^* \widetilde{G}_2(z) \bm{u}=\widetilde{\bm{u}}^* \widetilde{G} \widetilde{\bm{u}}.
\end{equation*} 
Moreover, by (\ref{eq_expressionafterwood}), we have 
\begin{align*}
\widetilde{\bm{u}}^* \widetilde{G} \widetilde{\bm{u}}= \bm{u}^* G_2(z) \bm{u}- z \widetilde{\bm{u}}^*G(z) \widehat{\Vb}_r \left(\Db^{-1}+1+z\widehat{\Vb}_r^* G(z) \widehat{\Vb}_r \right)^{-1} \widehat{\Vb}_r^* G(z) \widetilde{\bm{u}}.
\end{align*}
This completes our proof. 
\end{proof}

\subsection{Some auxiliary lemmas} In this subsection, we collect some auxiliary lemmas. 
\begin{lemma}\label{lem_trapezoidalrule}
Suppose $\Gamma$ is a curve of length one with infinitely differentiable arc length parameterization $\ell:[0,1] \rightarrow \Gamma$ such that $\ell(0)<\ell(1/2).$ Given some large integer $m,$ denote $t_j=(2j+1)/2m, j=0,1,2,\cdots,m$ with the convention $s_m=s_0.$ Then for every $D>0,$ there exists some $C_D \equiv C_D(\Gamma)>0,$ such that 
\begin{equation*}
\left| \oint_{\Gamma} f(z) \dd z-\sum_{j=0}^{m-1} f(z_j) w_j \right|\leq C_D \| f^{(D)} \|_{\infty} m^{-D},
\end{equation*}
where $z_j$ are $w_j, j=0,1,2,\cdots,m,$ are defined as
\begin{equation}\label{eq_wjzj}
z_j=\ell(s_j), \ w_j=\frac{\ell'(s_j)}{m}.
\end{equation} 
\end{lemma}
\begin{proof}
The proofs follows from a standard approximation argument using Euler-Maclaurin formula. For example, see the arguments above the proof of Theorem 5.11 in \cite{Paquette2020}. 
\end{proof}

\begin{lemma}\label{lem_expcilitformula}
Denote the standard Marchenko-Pastur law \cite{Marcenko1967} by $\mu_{\mathtt{MP}}$  with parameter $\mathfrak c$, i.e., 
\begin{equation}\label{eq_mpdensity}
\dd \mu_{\mathtt{MP}}(x)=\frac{1}{2 \pi \mathfrak c} \sqrt{\frac{[(x-\gamma_-)(\gamma_+-x)]_+}{x^2}} \dd x+(1-\mathfrak c^{-1})_+ \delta_0(\dd x), \ \text{where} \ \gamma_{\pm}=(1\pm \sqrt{\mathfrak c})^2,
\end{equation}
where $[\cdot]_+$ gives the positive part of $(\cdot)$. Suppose the spectrum of $\Sigma_0$ is given by the typical locations of $\mu_{\mathtt{MP}}$: 
\begin{equation*}
\int^{\gamma_+}_{\sigma_i} \dd \mu_{\mathtt{MP}}(x)=\frac{i-1/2}{N}, 1 \leq i \leq N.   
\end{equation*}
Set  $\alpha_2=\mathfrak c^{-1}$ and assume that $\alpha_1, \alpha_2>1,$ and $c_N^{-1} \rightarrow \alpha_1.$ Let $\rho(\lambda)$ as the asymptotic density function of the ESD of $W_0$  defined in (\ref{eq_definitioncovariance}). Then
\begin{align*}
& \rho(\lambda)=\frac{\sqrt{3}}{6 \pi 2^{1/3} \lambda } \left( \sqrt[3]{9 \alpha_1 (1+\alpha_1+\alpha_2)(\lambda-\xi_0)+6 \sqrt{3 \alpha_1^3(\lambda-\lambda_-)(\lambda_{+,2}-\lambda)(\lambda-\lambda_{+,1}) }} \right. \\
 &-\left.  \sqrt[3]{9 \alpha_1 (1+\alpha_1+\alpha_2)(\lambda-\xi_0)-6 \sqrt{3 \alpha_1^3(\lambda-\lambda_-)(\lambda_{+,2}-\lambda)(\lambda-\lambda_{+,1}) }} \right) \mathbf{1}\left( \lambda \in [\lambda_{+,1}, \lambda_{+,2}]\right),
\end{align*}
where $\xi_0 \equiv \xi_0(\alpha_1, \alpha_2),  \lambda_- \equiv \lambda_-(\alpha_1,\alpha_2)$ and $ \lambda_{+,k} \equiv \lambda_{+,k}(\alpha_1,\alpha_2), k=1,2,$ can be calculated explicitly and defined in \cite{2014arXiv1401.7802D}.  As a special case, if $\alpha_1=\alpha_2=\alpha, $ we have that $\lambda_{+,k}=\lambda_+, k=1,2,$ and 
\begin{equation*}
\lambda_{\pm}=\frac{-1+20\alpha+8 \alpha^2 \pm (1+8 \alpha)^{3/2}}{8 \alpha^2}, \ \xi_0=\frac{2(\alpha-1)^3}{9 \alpha (1+2\alpha)}.
\end{equation*}
\end{lemma}
\begin{proof}
See \cite[Section 4.2]{2014arXiv1401.7802D}. 
\end{proof}

%

\bibliographystyle{abbrv}
\bibliography{references,rmtref}

\end{document}